\documentclass[leqno,11pt]{article}
\usepackage{amssymb,amsmath}

\usepackage{bbold}

\usepackage[all]{xy}

\newenvironment{proof}{{\bf Proof. }}{\par}{\bigskip}
\newtheorem{theo}{Theorem}[section]
\newtheorem{defi}[theo]{Definition}

\newtheorem{lem}[theo]{Lemma}
\newtheorem{prop}[theo]{Proposition}
\newtheorem{rem}[theo]{Remark}
\newtheorem{coro}[theo]{Corollary}
\newtheorem{exam}[theo]{Example}

\newcommand{\agot}{\ensuremath{\mathfrak{a}}}

\newcommand{\ggot}{\ensuremath{\mathfrak{g}}}
\newcommand{\hgot}{\ensuremath{\mathfrak{h}}}
\newcommand{\kgot}{\ensuremath{\mathfrak{k}}}

\newcommand{\rgot}{\ensuremath{\mathfrak{r}}}

\newcommand{\qgot}{\ensuremath{\mathfrak{q}}}
\newcommand{\sgot}{\ensuremath{\mathfrak{s}}}
\newcommand{\tgot}{\ensuremath{\mathfrak{t}}}

\newcommand{\Bcal}{\ensuremath{\mathcal{B}}}
\newcommand{\Ccal}{\ensuremath{\mathcal{C}}}
\newcommand{\Dcal}{\ensuremath{\mathcal{D}}}
\newcommand{\Ecal}{\ensuremath{\mathcal{E}}}
\newcommand{\Fcal}{\ensuremath{\mathcal{F}}}
\newcommand{\Gcal}{\ensuremath{\mathcal{G}}}
\newcommand{\Hcal}{\ensuremath{\mathcal{H}}}

\newcommand{\Lcal}{\ensuremath{\mathcal{L}}}

\newcommand{\Ncal}{\ensuremath{\mathcal{N}}}
\newcommand{\Ocal}{\ensuremath{\mathcal{O}}}
\newcommand{\Pcal}{\ensuremath{\mathcal{P}}}
\newcommand{\Qcal}{\ensuremath{\mathcal{Q}}}
\newcommand{\Rcal}{\ensuremath{\mathcal{R}}}

\newcommand{\Scal}{\ensuremath{\mathcal{S}}}

\newcommand{\Ucal}{\ensuremath{\mathcal{U}}}
\newcommand{\Vcal}{\ensuremath{\mathcal{V}}}
\newcommand{\Wcal}{\ensuremath{\mathcal{W}}}
\newcommand{\Xcal}{\ensuremath{\mathcal{X}}}

\newcommand{\Z}{\ensuremath{\mathbb{Z}}}
\newcommand{\C}{\ensuremath{\mathbb{C}}}

\newcommand{\R}{\ensuremath{\mathbb{R}}}

\newcommand{\N}{\ensuremath{\mathbb{N}}}
\newcommand{\X}{\ensuremath{\mathbb{X}}}
\newcommand{\Lbb}{\ensuremath{\mathbb{L}}}
\newcommand{\Tbb}{\ensuremath{\mathbb{T}}}

\newcommand{\mm}{\ensuremath{\hbox{\rm m}}}
\newcommand{\tr}{\ensuremath{\hbox{\bf Tr}}}

\newcommand{\croc}{\ensuremath{\hookrightarrow}}
\newcommand{\Id}{\ensuremath{\hbox{\rm Id}}}

\newcommand{\indice}{\ensuremath{\hbox{\rm Index}}}
\newcommand{\Rfor}{\ensuremath{\hat R}}

\newcommand{\Char}{\ensuremath{\hbox{\rm Char}}}
\newcommand{\End}{\ensuremath{\hbox{\rm End}}}

\newcommand{\supp}{\operatorname{\hbox{\rm \small supp}}}

\newcommand{\Ko}{\ensuremath{{\mathbf K}^0}}

\newcommand{\KK}{\ensuremath{\mathbf{K}^*}}

\newcommand{\Sym}{\ensuremath{{\rm Sym}}}

\newcommand{\spin}{\ensuremath{\rm Spin}}
\newcommand{\spinc}{\ensuremath{{\rm Spin}^{c}}}

\def \at {{\rm At}}
\def \At {{\rm \bf At}}
\def \K {{\rm \bf K}}
\def \T {{\rm T}}

\def \what {\widehat}

\def \ind {{\rm Ind}}
\def \rang {{\rm rank}}

\def \clif {{\bf c}}
\def \Cl {{\rm Cl}}

\newcommand{\JJbeta}{\ensuremath{{{J_\beta}}}}
\newcommand{\JJa}{\ensuremath{{{J_a}}}}

\newcommand{\rrmu}{\ensuremath{{{\mu}}}}
\newcommand{\rra}{\ensuremath{{{a}}}}
\newcommand{\ddY}{\ensuremath{{{dY}}}}

\usepackage[usenames,dvipsnames,svgnames,table]{xcolor}

\begin{document}

\title{Witten non abelian localization  for equivariant $K$-theory, and the $[Q,R]=0$ theorem}

\author{Paul-Emile PARADAN\footnote{Institut Montpelli\'erain Alexander Grothendieck, CNRS UMR 5149,
Universit\'e de Montpellier, \texttt{paul-emile.paradan@umontpellier.fr}}
\hspace{1mm}  and
Mich\`ele VERGNE\footnote{Institut de Math\'ematiques de Jussieu, CNRS UMR 7586,
Universit\'e Paris 7, \texttt{vergne@math.jussieu.fr}}}

\maketitle

{\small
\tableofcontents}

\newpage

\section{Introduction}

Let  $M$ be a  manifold provided with an action of a compact
Lie group  $K$ with Lie algebra $\kgot$.
Our main additional  data is a $K$-invariant map  $\Phi: M\to \kgot^*.$ This gives rise to  a $K$-invariant
vector field $\kappa_\Phi$, the Kirwan vector field,  on $M$.
In \cite{Witten}, Witten proposed a non abelian localization procedure on the zero set of the Kirwan vector field for
the integration of equivariant classes in the Hamiltonian context.
The $K$-theoretic analogue of the Witten non abelian localization procedure was introduced
by the first author \cite{pep-RR}, still in the Hamiltonian context, in order to give
a topological proof of the $[Q,R]=0$ theorem,  after the first input of the second author in \cite{Vergne96}. The  $[Q,R]=0$ theorem, conjectured by Guillemin-Sternberg and proved by \cite{Meinrenken-Sjamaar}, is a remarkable result on the   geometric quantization of symplectic manifolds, and  we describe it below.
Still with the aim of understanding the $[Q,R]=0$ theorem,  this type of
 non abelian localization procedure   has been transposed recently in the
setting of equivariant $K$-homology by Y. Song \cite{Song}. The analytic approach of the Witten deformation
has been studied by several people, see \cite{Tian-Zhang98,Braverman-02,Mathai-Zhang,Ma-Zhang14,Braverman-14,Mathai-Hochs14}.

The purpose
of the present paper is three-fold.

First, we   develop  our $\K$-theoretical analog of Witten deformation,  when $M$ is any even dimensional oriented compact manifold,
 $\Phi$ a moment map in a weak sense,
 and $\sigma$ a symbol of any  twisted Dirac operator. We  use the vector field $\kappa_\Phi$ to deform  an element
$\sigma$ of the equivariant $\K$-theory of $\T^*M$.
We obtain a general non abelian localization theorem,
see Section \ref{subsec:nonabelianlocalization}, in equivariant $K$-theory.
It localizes the study of the equivariant index of any $\sigma\in \K_K^0(\T^*M)$ to the study of the equivariant
index of  the deformed element on a neighborhood of the zeroes of the map $\Phi$, together with the study
of the same scenario on lower dimensional manifolds. In this deformation, an elliptic symbol is deformed in
a sum of transversally elliptic symbols.

Our second purpose is to illustrate the  power of this tool by reproving the $[Q,R]=0$ theorem of Meinrenken-Sjamaar \cite{Meinrenken-Sjamaar} in a straightforward way.
In other words, we clarify the strategy used by the first author in \cite{pep-RR}.
We  also prove some generalizations, in the context of almost complex manifolds.
In particular, we prove an asymptotic $[Q,R]=0$ theorem for Clifford bundles
(Theorem \ref{eq:QR-Phi-0})
 which  plays an essential role when one studies the asymptotic behavior of branching law coefficients (see \cite{pep-stability}) and the notion of stability.

As another application of our
 $K$-theoretical  non-abelian localization theorem, we have obtained a geometric description of the multiplicities of the index of   general twisted Dirac operators on general $K$-manifolds. This is the object  of a separate article \cite{pep-vergne:spinc}, as the proof requires further delicate information on admissible orbits, see \cite{pep-vergne:magic}.

Finally, our third purpose is to describe in detail some applications of the $[Q,R]=0$ theorem
on multiplicities. If $M$ is a symplectic manifold and $L$ a Kostant line bundle on $M$,
the geometric quantization $RR_K(M,L)$ of the symplectic manifold $M$ is a (virtual) representation space of $K$.
We write $RR_K(M,L)=\sum_{\mu\in\hat{K}} \mm_\mu(L) V_\mu $.
We give the proofs on the  piecewise quasi-polynomial behavior of the multiplicity function $\mm_\mu(L)$   as well as their ``reduction behavior" on faces.
All these results follow from understanding the geometric description
of $\mm_\mu(L)$ given by Meinrenken-Sjamaar, but  this is sometimes not  so easy to see while reading \cite{Meinrenken-Sjamaar}. So we give a complete proof of all this facts here.
 We feel that it is worthwhile to explain a unified proof of these results, via Hamiltonian geometry and the $[Q,R]=0$ theorem.
 It seems that indeed these consequences were unnoticed,
 and  other interesting  proofs were given for interesting  cases by several authors
 \cite{Kumar-Prasad,Ressayre-reduction, Derksen-Weyman}.
 For example, the quasi polynomial behavior in $k$ of stretched Littlewood-Richardson coefficients
 \cite{Derksen-Weyman}, or of Kronecker coefficients \cite{Mulmuley},
  the reduction properties of these coefficients on faces \cite{Ressayre-reduction}, the factorization theorem of LR
  coefficients \cite{King-Tollu-Toumazet}, the locally quasi-polynomial behavior of plethysm coefficients
  \cite{Kahle-Michalek},  are  consequences of the geometric description of multiplicities obtained by Meinrenken-Sjamaar.
Another interesting consequence of the asymptotic $[Q,R]=0$ theorem for Clifford bundles
(Theorem \ref{eq:QR-Phi-0})
  plays an essential role when one studies the asymptotic behavior of branching law coefficients. This is the object of another article by the first author \cite{pep-stability}.

\bigskip

Before describing our localization result, and its consequences,  in more details,
let us present some motivation and previous results.
Let $A$ be a $K$-invariant elliptic operator on $M$ with principal symbol $\sigma$ and let
$\indice_K^M(\sigma)(k)$ its index, a central function of $K$. Write
$\indice_K^M(\sigma)(k)=\sum_{\mu\in \hat K} \mm(\mu,\sigma) \tr_\mu(k)$ its
Fourier expansion in terms of irreducible characters of $K$.
The numbers $\mm(\mu,\sigma)\in \Z$ are the multiplicities of the irreducible representation $V_\mu$
in the virtual space of solutions of $A$.
In \cite{Guillemin-Sternberg82}, Guillemin-Sternberg formulated the $[Q,R]=0$ principle (quantization commutes with reduction),
which gave a conjectural geometric receipe to compute the multiplicities $\mm(\mu,\sigma_L)$, in the case where $M$ is a
Hamiltonian manifold equipped  with a Kostant-Souriau line bundle $L$, in terms of  the  fibers of the moment map
$\Phi: M\to \kgot^*$ given by the Hamiltonian structure.  Here $\sigma_L$ is the symbol of the  Dolbeaut-Dirac
operator with coefficients in $L$.
 Introduce the space $M_{red}=\Phi^{-1}(0)/K$ (and assume to simplify that
$M_{red}$ is smooth). The space $M_{red}$ is the reduced Marsden-Weinstein symplectic space. It plays an important role,
as for example many moduli spaces appear naturally as Marsden-Weinstein quotients. Also  when $M$ is a projective manifold,
the reduced manifold $M_{red}$ is the  GIT quotient $M/\!\!/K_\C$ of $M$, where $K_\C$ is the complexification of $K$.
Guillemin-Sternberg proved that the multiplicity of the trivial representation of $K$
is equal to the Riemann-Roch number of $M_{red}$, in the case where $M$ is a compact complex manifold
and $L$ is ample. They conjectured that this equality is still valid in the Hamiltonian context
\cite{Guillemin-Sternberg82}.
Using symplectic cuts, Meinrenken-Sjamaar \cite{Meinrenken-Sjamaar} gave the first proof of this fundamental theorem.

In view of describing  cohomology rings  of moduli spaces, Witten showed that  integrals of cohomology classes on $M_{red}$ can be computed as the value  at $0$ of the Fourier transform of a function on $\kgot$, obtained by  integrating  an equivariant cohomology class on the manifold $M$, using a deformation argument.
This result suggested that a similar deformation argument could be used  to study the Fourier expansion of equivariant indices.
Indeed, in \cite{pep-RR}, the first author gave another proof of the $[Q,R]=0$ theorem, using deformation of the symbol $\sigma_L$.
This proof is straightforward, once some functorial properties of the deformation are established. These functorial properties were established in  \cite{pep-RR}, by several ad-hoc arguments.
In this article, we prove these functorial properties in a general framework.
We  believe that this powerful technique can be useful  to study equivariant indices in many different contexts.
We see, for example, that a version of $[Q,R]=0$ theorem is true in this general context, in the asymptotic sense.
Furthermore, we use the results of this article to prove a $[Q,R]=0$ exact principle in the context of $\spinc$-manifolds
\cite{pep-vergne:spinc}.

\medskip

Let us now describe more precisely our context.
Assume $M$ is even dimensional and consider a graded Clifford bundle $\Ecal\to M$ for the tangent bundle $\T M$.
We denote by $\clif_{\Ecal_m}$ the Clifford action of $\T_m M$ on the fiber $\Ecal_m$. The symbol $\sigma(\Ecal)$ attached to the graded Clifford bundle is
$$
\sigma(\Ecal)(m,\nu)
=\clif_{\Ecal_m}(\tilde{\nu}): \Ecal_m^+\longrightarrow \Ecal^-_m,\quad (m,\nu)\in\T^* M,
$$
where $\nu\in\T^* M\mapsto \tilde{\nu}\in\T M$ is an identification given by a $K$-invariant Riemannian structure on $M$.
Note that such a symbol is invertible outside the zero section of $\T^*M$. Hence, when $M$ is compact,
the symbols $\sigma(\Ecal)$ determines a class in the $\K$-group $\Ko_K(\T^*M)$, and we denote
$\Qcal_K(M,\Ecal)$ its equivariant index.
If $K=\{1\}$ is the trivial group, $\Qcal_K(M,\Ecal)\in \Z$ is an integer, denoted simply by  $\Qcal(M,\Ecal).$

We attach the Kirwan vector field $\kappa_\Phi$ to any equivariant map $\Phi:M\to \kgot^*$: for $m\in M$, the vector
$-\kappa_{\Phi}(m)$ is equal to the the tangent vector $\Phi(m)\cdot m$ associated to the infinitesimal action of
$\Phi(m)$ on $m$. Here we have identified $\kgot$ and $\kgot^*$ by choosing a $K$-invariant inner product.
We denote by $Z_\Phi$ the set of zeroes of $\kappa_{\Phi}$. Clearly $Z_\Phi$ contains $\Phi^{-1}(0)$.
The deformed symbol  $\sigma(\Ecal,\Phi)$ is defined by
$$
\sigma(\Ecal,\Phi)(m,\nu)
=\clif_{\Ecal_m}(\tilde{\nu}+\Phi(m)\cdot m): \Ecal_m^+\longrightarrow \Ecal^-_m.
$$
The symbol $\sigma(\Ecal,\Phi)$ can be restricted to a neighborhood of the set $Z_\Phi$ of zeroes of $\kappa_\Phi$
as a transversally elliptic symbol. If $Z$ is a union of connected components of $Z_\Phi$, we denote by  $\sigma(\Ecal,Z,\Phi)$ the restriction of  $\sigma(\Ecal,\Phi)$ to a small neighborhood of $Z$ and by
$\Qcal_K(M,\Ecal, Z,\Phi)$ the equivariant index of the symbol $\sigma(\Ecal,Z,\Phi)$.

\medskip

In this paper, an equivariant map $\Phi:M\to \kgot^*$ is called\footnote{It corresponds to the notion of
``abstract moment map'' of Karshon.} a ``moment map'' if there is a closed two form $\Omega$
(possibly degenerate) such that the
data $(\Omega,\Phi)$ satisfies the Kostant relations (\ref{eq:hamiltonian-action}).

\medskip

In Section \ref{sec:non-abelian-localization}, we  prove the following non-abelian localization theorem
(Theorem  \ref{theo;nonabelian}).
 When $\Phi$ is a moment map, the set  $C_\Phi=\Phi(Z_\Phi)$ is
a finite union of coadjoint orbits, and we define $Z_\Ocal=Z_\Phi\cap \Phi^{-1}(\Ocal)$ for any $\Ocal\subset C_\Phi$.
We then have, for any $K$-equivariant graded Clifford bundle $\Ecal$,
 \begin{equation}
\label{eq:localization-Phi-intro}
\Qcal_K(M,\Ecal)=\sum_{\Ocal\subset C_\Phi} \Qcal_K(M,\Ecal, Z_\Ocal,\Phi).
\end{equation}
When $0$ is a regular value of $\Phi$, the character $\Qcal_K(M,\Ecal, \Phi^{-1}(0),\Phi)$ can
be described explicitly in terms of indices of elliptic symbols on the orbifold $M_{red}=\Phi^{-1}(0)/K$
(see Subsection \ref{subsec:loc at zero}). In particular we have that $[\Qcal_K(M,\Ecal, \Phi^{-1}(0),\Phi)]^K=\Qcal(M_{red},\Ecal_{red})$, where $\Ecal_{red}$ is a graded Clifford bundle on $M_{red}$ induced by $\Ecal$.
In other words, the symbol $\sigma(\Ecal,Z,\Phi)$ satisfies the $[Q,R]=0$ principle, for
any Clifford bundle $\Ecal$.

\medskip

Let us describe the other components associated to $\Ocal=K\beta$, with $\beta\neq 0$.
Let  $M^{\beta}$  be the set of  zeroes of the vector field $\beta$,
and $K_{\beta}$ the centralizer of $\beta$ in $K$ with Lie algebra $\kgot_{\beta}$. The graded Clifford bundle $\Ecal$
on $M$ induces a graded Clifford bundle $\mathbb{d}_\beta(\Ecal)$ on $M^\beta$. The action of $\beta$
provides a complex structure  on the normal bundle to $M^{\beta}$, and we denote by $\Ncal_{\JJbeta}$
this complex vector bundle. Let  $\Sym(\Ncal_\JJbeta)=\oplus_{k=0}^{\infty} \Sym^k(\Ncal_\JJbeta)$.
Then $\Qcal_K(M,\Ecal, Z_\Ocal,\Phi)$  is the induced representation from $K_{\beta}$ to $K$ of the
representation
$$
\Qcal_{K_\beta}\left(M^{\beta},\mathbb{d}_\beta(\Ecal)\otimes \Sym(\Ncal_\JJbeta), \Phi^{-1}(\beta),\Phi|_{M^{\beta}}\right)\otimes \bigwedge (\kgot/\kgot_{\beta})_\C.
$$
Here again, if $\beta$ is a regular value of the map $\Phi:M^{\beta}\to \kgot_{\beta}^*$,
this index can be explicitly described in terms of $(\Phi^{-1}(\beta)\cap M^{\beta})/ K_{\beta}$.

\medskip

Starting in Section \ref{sec:QR-theorem}, we analyze the implications  of the non-abelian localization theorem to the geometric study of the multiplicities.
Let us remark that we do not assume before
 Section \ref{sec:hamiltonian} that nor our group $K$ neither our manifold $M$ are connected. Indeed
the study of the ring of invariants for classical groups such as $O(n)$ is an instance of the $[Q,R]=0$ theorem, and we
want to include these cases.

We say that the $[Q,R]=0$ principle of Guillemin-Sternberg holds for the data $(\Ecal,\Phi)$ if the multiplicity of
the trivial representation in the equivariant index $\Qcal_K(M,\Ecal, Z_\Ocal,\Phi)$ vanishes if $\Ocal\neq \{0\}$. Thanks to
(\ref{eq:localization-Phi-intro}), this implies the identity
$$
[\Qcal_K(M,\Ecal)]^K= [\Qcal_K(M,\Ecal, \Phi^{-1}(0),\Phi)]^K,
$$
and this last space is described explicitly as  $\Qcal(M_{red},\Ecal_{red})$ when $0$ is a regular value.
If $\Phi$ is a moment map associated to  an equivariant line bundle $L$,
 it is a straightforward consequence of our localization theorem that the $[Q,R]=0$ principle holds   for the data
$(\Ecal \otimes L^k, \Phi)$, if $k$ is large enough
(this result would also follow easily from the asymptotic analysis by Meinrenken
in \cite{Meinrenken96}).  Consequences of this asymptotic result \cite{pep-stability} includes many
stability results on asymptotic behavior of branching coefficients.

We introduce a notion of $\Phi$-positivity for equivariant vector bundles and
Clifford bundles which extends a similar notion introduced by Tian-Zhang in the Hamiltonian context \cite{Tian-Zhang98}.

When the Clifford bundle $\Ecal$ is strictly $\Phi$-positive, we show that the $[Q,R]=0$ principle
holds for the data $(\Ecal,\Phi)$: the multiplicity of trivial representation in the characters $\Qcal_K(M,\Ecal)$ and
$\Qcal_K(M,\Ecal, \Phi^{-1}(0),\Phi)$ are equal. If furthermore $0$ is a regular value of $\Phi$, we get finally that
$$
\left[\Qcal_K(M,\Ecal)\right]^K= \Qcal(M_{red},\Ecal_{red}).
$$

\medskip

In Section \ref{sec:almost-complex}, we specialize the study to the case of a $K$-manifold $M$ equipped with
an invariant almost complex structure $J$. Then a natural Clifford bundle on $M$ is $\bigwedge_J \T M$, and we consider the
Riemann-Roch character
$$
RR^J_K(M,\Fcal):=\Qcal_K(M,\bigwedge_J \T M\otimes\Fcal)
$$
that is attached to any equivariant vector bundle $\Fcal$.

When $\Phi$ is a moment map associated to a closed $2$-form
$\Omega$, we show that the Clifford bundle $\bigwedge_J \T M$ is weakly $\Phi$-positive when the quadratic form
$v\mapsto \Omega(v,Jv)$ is non negative (we say that $J$ is adapted to $\Omega$). Hence
the Clifford bundle $\Ecal:=\bigwedge_J \T M\otimes \Fcal$ is strictly $\Phi$-positive when $J$ is adapted to $\Omega$, and the vector bundle $\Fcal$ is strictly $\Phi$-positive.

An interesting example are the {\em semi-ample} line bundles $L$ on $(M,J)$:  we define those as line bundles $L$ for which
there exists a $K$-invariant Hermitian connection $\nabla^L$ with curvature $(\nabla^L)^2=i\Omega_L$ such that $\Omega_L(v,Jv)$ is semi-positive definite. If $\Phi_L$ is the moment map attached to $\nabla^L$, we see then
that the Clifford bundle $\bigwedge_J \T M\otimes L$ is strictly $\Phi_L$-positive. In particular we get that
$\left[RR^J_K(M,L^{\otimes k})\right]^K=0$ for any $k\geq 1$ if $\Phi_L^{-1}(0)=\emptyset$.

\medskip

In Section \ref{sec:hamiltonian}, we return to the case of a $K$-Hamiltonian manifold $(M,\Omega)$ with a
Kostant-Souriau line bundle $L$ and moment map $\Phi$. Here we choose an almost complex structure $J$ that is
compatible with the non degenerate $2$-form $\Omega$: here the quadratic form $v\mapsto \Omega(v,Jv)$
is positive definite. The corresponding quantized space $RR_K(M,L)=\sum_{\mu\in\hat{K}} \mm_\mu(L) V_\mu $
is a representation of $K$. As we have already seen in the case of semi-ample line bundles, the number
$\mm_0(L)=[RR_K(M,L)]^K$ is equal to
\begin{equation}\label{eq:intro-Q-0}
[\Qcal_K(M,\bigwedge_J \T M\otimes L, \Phi^{-1}(0),\Phi)]^K.
\end{equation}

If $0$ is regular, we thus obtain that $[RR_K(M,L)]^K$ is the Riemann-Roch number $RR(M_0,L_0)$ of the orbifold bundle $L_0=L\vert_{\Phi^{-1}(0)}/K$ on
the reduced symplectic orbifold $M_0=\Phi^{-1}(0)/K$.
Recall that this striking result was obtained by  Meinrenken \cite{Meinrenken98} using symplectic cutting. The core of
Meinrenken-Sjamaar fundamental theorem is that (\ref{eq:intro-Q-0}) can be computed geometrically as the Riemann-Roch number of the (possibly singular) reduced space $M_0=\Phi^{-1}(0)/K$:

$\bullet$ Either directly when $0$ is a quasi-regular value of $\Phi$.
The computation\footnote{The quasi-regular case
was not considered in \cite{pep-RR}.}, which is quite involved when $0$ is quasi-regular, but not regular, is
done in Section \ref{subsec:quasi-regular}. In this case, the reduced space $M_0$ is again an orbifold.

$\bullet$ Or by moving the localization from $\Phi^{-1}(0)$ to the fiber of a quasi-regular point close enough to $0$. This deformation argument is explained in Section \ref{subsec:multiRiemannRoch}.

Up to this point, our compact group $K$ could be not connected.
As usual, for $K$ connected, one deals with the other multiplicities $\mm_\mu(L)$ by using the shifting trick.

At the end of the paper, we describe the consequences of the $[Q,R]=0$ theorem  on the piecewise locally polynomial behavior on
the coefficients $\mm_\mu(L)$,  as well as their ``reduction behavior" on faces.

\medskip

\section*{Acknowledgments}
We wish to thank the Research in Pairs program at Mathematisches \break Forschungsinstitut Oberwolfach
(February 2014), where this work was started. The second author wishes to thank Michel Duflo.

\section{Notations}\label{secNotations}

$\bullet$ Let  $K$ be a compact (non necessarily connected) Lie group with Lie algebra $\kgot$.
We denote by  $\hat K$  the set of isomorphism classes of irreducible complex finite dimensional representations of $K$.
If $\mu \in \hat K$, we may also denote $\mu$ by $V_\mu$
(a vector space with an action of $K$).
If $\mu \in \hat K$, we denote  by $\mu^*\in \hat K$  the dual representation of $K$ in $V_{\mu^*}=V_\mu^*$.

Any finite dimensional representation $E$ of $K$ is  isomorphic to a direct sum $E=\oplus_{\mu\in\what{K}}\mm(\mu) V_\mu$,
and we say that $\mm(\mu)$ is the multiplicity of $\mu$ in $E$.
We denote by $[E]^K$ the subspace of $E$ formed by the $K$-fixed vectors. Thus $\mm(\mu)=\dim [E\otimes V_\mu^*]^K$.

If $H$ is a subgroup of $K$, we denote by $E|_H$ the space $E$   with the action of $H$.

 $\bullet$ We denote by $R(K)$ the representation ring of   $K$ : an element $E\in R(K)$ can be represented
as a finite sum $E=\sum_{\mu\in\what{K}}\mm(\mu) V_\mu$, with $\mm(\mu)\in\Z$. We also say that $E$ is a virtual representation of $K$.
The ring structure on $R(K)$ is induced by the tensor product and we denote by $\chi\otimes \mu$ the product.
If $H$ is a subgroup of $K$, we denote by $\chi\to \chi|_H$ the restriction $R(K)\to R(H)$.

We denote by $\Rfor(K)$ the space of $\Z$-valued functions on $\hat K$. An element $E\in\Rfor(K)$ can be represented
as an infinite sum $E=\sum_{\mu\in\what{K}}\mm(\mu) V_\mu$, with $\mm(\mu)\in\Z$.
We say that $\mm(\mu)$  is the multiplicity of $\mu$ in $E$.
Then $[E]^K\in \Z$ and is the multiplicity of the trivial representation in $E$.

The space $\Rfor(K)$ is a $R(K)$-module.
If $H$ is a closed subgroup of $K$,  the induction map
${\rm Ind}^{K}_{H}: \Rfor(H)\to \Rfor(K)$ is the dual of the restriction morphism $R(K)\to R(H)$.
If $\chi\in \hat R(H)$,
then
${\rm Ind}^{K}_{H}(\chi)=\sum_{\lambda\in \hat K} [\chi\otimes V_\lambda^*|_H]^{H} V_\lambda$.
In particular  $[{\rm Ind}^{K}_{H}(E)]^K= [E]^H$.

$\bullet$
We consider a positive definite  $K$ invariant scalar product on $\kgot$.
This defines an isomorphism between $\kgot$ and its dual vector space $\kgot^*$.

$\bullet$ If $\sgot$ is a subalgebra of $\kgot$, we denote by $(\sgot)$ the conjugacy class of $\sgot$ in the set of subalgebras of $\kgot$.
 When $K$ acts on a set $\Xcal$, the stabilizer subgroup of $x\in \Xcal$ is denoted $K_x:=\{k\in K\ \vert\ k\cdot x=x\}$.
The Lie algebra of $K_x$ is denoted by $\kgot_x$.
We denote by $\Xcal_{\sgot}$ the set of elements $x$ of $\Xcal$ such $\kgot_x=\sgot$,
and by $\Xcal_{(\sgot)}=K\Xcal_{\sgot}$ the set of elements $x$ of $\Xcal$ such $\kgot_x\in (\sgot)$.

When $K$ acts on a manifold $M$, we denote by $X_M(m):=\frac{d}{dt}\vert_{t=0} e^{-tX}\cdot m$ the vector
field generated by $-X\in \kgot$.
We may denote also $X_M(m)$ by $-X\cdot m$.
We denote by $M^X$ the set of zeroes of the vector field $X_M$. This is a smooth manifold.

In particular, if $b\in \kgot$ and $E$ is a representation space of $K$, then $E^b$ is the subspace of vectors in $E$
invariant by the one parameter subgroup $\exp t b$ of $K$.

$\bullet$ When $V$ is a real vector space, we equip the direct sum $V\oplus V^*$ with the symplectic structure
$\Omega$ defined by : $\Omega(v,v')=\Omega(f,f')=0$ and $\Omega(v,f)=\langle f,v\rangle$, for $v,v'\in V$ and $f,f'\in V^*$.
An Euclidean structure on $V$ induces a isomorphism $v\in V\mapsto \tilde{v}\in V^*$, and a complex structure $J_V$
on $V\oplus V^*$ defined by $J_V v=\tilde{v}$. Note  that
$\Omega(-, J_V -)$ is the scalar product on $V\oplus V^*$.

$\bullet$ Let $\Vcal\to M$ be a  vector bundle. An element $n$ of $\Vcal$ may be written as $(m,v)$ with $m\in M, v\in \Vcal\vert_m$.

$\bullet$
The letter $E$ will denote a complex finite dimensional  (eventually Hermitian) vector space, and $\Ecal\to M$ a  complex vector bundle.
The space of smooth sections of $\Ecal$ is denoted by $\Gamma(M,\Ecal)$.

Most of the times, the space $E$ will be $\Z/2\Z$ graded, $E=E^+\oplus E^-$. We may leave the $\Z/2\Z$ grading implicit.
We denote by $E^{op}$ the shift of parity on $E$: we interchange $E^+$ and $E^-$.
Ungraded spaces $N$ will be considered as graded, with $N^+=N$ and $N^-=\{0\}$.  Thus $(-1)^kE$  denotes  the  same graded vector space $E$, but where the change of parity has been performed $k$ times.

If $N$ is a complex (ungraded) vector space, then the exterior space $\bigwedge N$ will be $\Z/2\Z$ graded in even and odd elements.
If $E=E^+\oplus E^-$ and $F=F^+\oplus F^-$ are two
$\Z/2\Z$ graded  Hermitian vector spaces, the tensor product $E\otimes F$ is $\Z/2\Z$-graded with
$(E\otimes F)^+=E^+\otimes F^+\oplus E^-\otimes F^-$ and
$(E\otimes F)^-=E^-\otimes F^+\oplus E^+\otimes F^-$.
If $f\in \hom(E^+,E^-)$, $g\in \hom(F^+,F^-)$, we define
$f\boxtimes g \in \hom((E\otimes F)^+, (E\otimes F)^-)$ by
\begin{equation}\label{eq:product}
f\boxtimes g=\left(
\begin{array}{cc}
f\otimes 1_{F_+} & -1_{E_-}\otimes g^* \\
1_{E_+}\otimes g& f^*\otimes 1_{F_-} \\
\end{array}
\right).
\end{equation}

$\bullet$  If $M$ is a $K$-manifold, and $E$ a representation space for $K$, we denote by $[E]$ the trivial vector bundle $M\times E$, with diagonal action of $K$.

\section{Elliptic and transversally elliptic symbols}\label{subsec:trans-elliptic}
In this section,  we recall the basic definitions from the theory of transversally
elliptic symbols (or operators) defined by Atiyah-Singer in
\cite{Atiyah74}.

Let $M$ be a compact $K$-manifold with cotangent bundle $\T^*M$. Let $p:\T^*
M\to M$ be the projection.
If $\Ecal$ is a vector bundle on $M$, we may denote still by $\Ecal$ the vector bundle $p^*\Ecal$ on the cotangent bundle $\T^*M$.
If $\Ecal^{+},\Ecal^{-}$ are
$K$-equivariant  complex vector bundles over $M$, a
$K$-equivariant morphism $\sigma \in \Gamma(\T^*
M,\hom(\Ecal^{+},\Ecal^{-}))$ is called a {\em symbol} on $M$.
 For $x\in  M$, and $\nu\in T_x^*M$,  thus $\sigma(x,\nu):\Ecal\vert_x^{+}\to
\Ecal\vert_x^{-}$
is a linear map.
The
subset of all $(x,\nu)\in \T^* M$ where the map $\sigma(x,\nu)$  is not invertible is called the {\em characteristic set}
of $\sigma$, and is denoted by $\Char(\sigma)$.
A symbol  is elliptic if its characteristic set is compact.

Define the exterior product of symbols.
Let $M_1,M_2$ be two manifolds with action of the  compact groups $K_1,K_2$ and   let $M=M_1\times M_2$.
We write an element of $\T^*M=\T^* M_1\times \T^*M_2$ as $((x_1,\nu_1),(x_2,\nu_2))$.
Let $f\in \Gamma(\T^*
M_1,\hom(\Ecal^{+},\Ecal^{-}))$ and $g
\in \Gamma(\T^*
M_2,\hom(\Fcal^{+},\Fcal^{-}))$ be two symbols.
 Using $K$-invariant  Hermitian metrics on the bundles $\Ecal^\pm,\Fcal^\pm$ we  define the  symbol
$f \boxtimes g\in \Gamma(T^*M,(\Ecal\otimes \Fcal)^+,  (\Ecal\otimes \Fcal)^-)$
by $$(f\boxtimes g)((x_1,\nu_1),(x_2,\nu_2))=f(x_1,\nu_1)\boxtimes g(x_2,\nu_2).$$

The symbol $f\boxtimes g$ is called the exterior product of the symbols $f,g$.
Note that $\Char( f\boxtimes g)=\Char(f)\times \Char(g) $.

Consider now $M_1=M_2=M$ with an action of $K$. Let $f,g$ be two symbols on $M$. We still denote by
$f\boxtimes g$ the restriction of the exterior product  (a symbol  on $M\times M$) to $M$ embedded as the diagonal in $M\times M$.

The product of a symbol  $\sigma$
by a  $K$-equivariant complex vector bundle $\Fcal\to M$ is the symbol
$\sigma\otimes \Fcal$
defined by
$$(\sigma\otimes \Fcal)(x,\nu)=\sigma(x,\nu)\otimes 1_{\Fcal_x}.$$

If $[E]$ is the  vector bundle $M\times E$  associated to a representation space $E$ of $K$, we denote
$\sigma\otimes [E]$ simply by $\sigma\otimes E$.

If $E=E^+\oplus E^-$ is $\Z/2\Z$ graded , then
$$\sigma \otimes E=\left(
                \begin{array}{cc}
                  \sigma\otimes 1_{E^+} &0 \\
                 0 & \sigma^*\otimes 1_{E^-} \\
                \end{array}
              \right).$$

 An elliptic symbol $\sigma$  on $M$ defines an
element $[\sigma]$ in the equivariant $\K$-theory of $\T^*M$ with compact
support, which is denoted by $\Ko_{K}(\T^* M)$.
The class  $[\sigma^*]$ is the inverse of $[\sigma]$ in $\K$-theory: $[\sigma^*]=-[\sigma]$.
The
index of $\sigma$ is a virtual finite dimensional representation of
$K$, that we denote by $\indice_{K}^M(\sigma)\in R(K)$
\cite{Atiyah-Segal68,Atiyah-Singer-1,Atiyah-Singer-2,Atiyah-Singer-3}.
Recall its definition.

Let $A$ be a pseudo-differential operator  $A:\Gamma(M,\Ecal^+)\to \Gamma(M,\Ecal^-)$ from  the space $\Gamma(M,\Ecal^+)$ of
smooth sections of $\Ecal^+$ to  the space $\Gamma(M,\Ecal^-)$ of  smooth sections of $\Ecal^-$, which commutes with the action of $K$.
Here we follow the convention of \cite{Atiyah74}[Lecture 1], so our pseudo-differential operator  $A$ belongs to
$\mathcal{P}^m(M,\Ecal^-,\Ecal^+)$ and we can consider its principal {\it  symbol  $\sigma_p$} which is defined outside
the zero section of $\T^*M$. The operator $A$ is said to be elliptic  if its principal symbol $\sigma_p(x,\nu)$ is
invertible for all $(x,\nu)\in \T^* M$ such that $\nu\neq 0$.
Then the space $\ker(A)$ of smooth solutions of $A$ is a finite dimensional  representation space for $K$.
Using  a $K$-invariant function $\chi$ on $\T^* M$ identically
equal to $1$ in a neighborhood of $M$ and compactly supported,
then $\sigma(x,\nu):=(1-\chi(x,\nu)) \sigma_p(x,\nu)$ is  defined on the whole space $\T^*M$ and is an elliptic symbol.
The class of $\sigma$ in $\Ko_K(\T^*M)$ does not depend of the choice of $\chi$. We still say that this class $\sigma$ is the symbol of $A$.
We choose  $K$-invariant  Hermitian structures on $\Ecal^+,\Ecal^-$. Then
the formal adjoint $A^*: \Gamma(M,\Ecal^-)\to \Gamma(M,\Ecal^+)$ is also an elliptic pseudodifferential operator.
The index of $A$ is the virtual representation space $\ker(A)-\ker(A^*)\in R(K)$.
This virtual representation {\em depends  only of the class of the symbol} $\sigma$ of $A$ in $\Ko_{K}(\T^* M)$, and by definition,  $$\indice_{K}^M(\sigma) =\ker(A)-\ker(A^*).$$

Recall the notion of {\it transversally elliptic symbol}.
Let  $\T^*_K M$ be the following $K$-invariant closed subset of $\T^*M$
$$
   \T^*_{K}M\ = \left\{(x,\nu)\in \T^* M,\ \langle \nu,X\cdot x\rangle=0 \quad {\rm for\ all}\
   X\in\kgot \right\} .
$$
 Its fiber over a point $x\in M$ is  formed by all the cotangent vectors $v\in T^*_xM$  which vanish on the tangent space to the orbit of $x$  under $K$, in the point $x$. Thus   each fiber $(\T^*_K M)_x$ is a linear subspace  of $T_x^* M$. In general the dimension of $(\T^*_KM)_x$  is not constant and this space is not a vector bundle.
A symbol $\sigma$ is  $K$-{\em transversally elliptic} if the
restriction of $\sigma$ to $\T^*_{K} M$ is invertible outside a
compact subset of $\T^*_{K} M$ (i.e. $\Char(\sigma)\cap
\T_{K}^*M$ is compact).

A $K$-{\em transversally elliptic} symbol $\sigma$ defines an
element of $\Ko_{K}(\T^*_{K}M)$, and the index of
$\sigma$
defines an element $\indice_K^M(\sigma)$ of $\hat{R}(K)$.
Recall its definition.
A transversally elliptic  operator
 $A:\Gamma(M,\Ecal^+)\to \Gamma(M,\Ecal^-)$ is a pseudo-differential operator which is  {\it elliptic} in the directions transversal to the orbits  of $K$:  its principal symbol $\sigma_p(x,\nu)$ is invertible for all $(x,\nu)\in \T^*_K M$
such that $\nu\neq 0$.
Thus, by restriction to $\T^*_KM$, the symbol
$\sigma(x,\nu)=(1-\chi(x,\nu))\sigma_p(x,\nu)$ defines a $K$-theory class\footnote{To simplify the notations, we will also denote $\sigma \in\Ko_K(\T^*_K M)$ the
class defined by a transversally elliptic symbol $\sigma$.} $[\sigma]$ in the topological equivariant $K$-theory group  $\Ko_K(\T^*_K M)$ which does not depend
of the choice of $\chi$. The {\it index map} associates to a  transversally elliptic operator  $A$ an element of  $\hat{R}(K)$ constructed as follows.  For every
$\mu\in \hat{K}$, the space $\hom_K(V_\mu,\ker(A))$ is a finite dimensional vector space of dimension $\mm(V_\mu,A)$.  Thus $\mm(V_\mu,A)$ is the multiplicity of $V_\mu$ in  the space
$\ker(A)$ of smooth solutions of $A$.
The number $\mm(V_\mu,A)-\mm(V_\mu,A^*)$ {\em depends  only of the class of the symbol} $\sigma$ in
$\Ko_K(\T_K^*M)$ and by definition
$$
\indice_K^M(\sigma) =\sum_{\mu\in \hat K} (\mm(V_\mu,A)-\mm(V_\mu,A^*)) V_\mu.
$$

In particular two symbols $\sigma_1$ and $\sigma_2$ having the same restriction to $\T^*_K M$ have the same index.

\begin{rem}
By definition,  $\indice_K^M(\sigma)$ is isomorphic to a graded subspace of  smooth (and $K$-finite) sections of  $\Ecal$, and we
will not consider how to complete
$\indice_K^M(\sigma)$ as an Hilbert space.
\end{rem}

\begin{exam}
A trivial, but important, example is the following. Consider
$M=K$, with action of $K\times K$ by left and right translations,
$\Ecal^+=M\times \C$ the trivial vector bundle, $\Ecal^-=M\times \{0\}$.
Then $\Gamma(M,\Ecal^+)=C^{\infty}(K)$,
 $\Gamma(M,\Ecal^-)=0$, and the operator $A=0$ on $C^{\infty}(K)$ is transversally elliptic.
 Its symbol is the symbol $\sigma_0(k,\nu)=0$,
 and Peter-Weyl theorem shows that
$$
\indice_{K\times K}^M(\sigma_0)=\sum_{\mu\in \hat K} V_\mu^*\otimes V_\mu.$$

\end{exam}

The index map $\indice_{K}^M: \Ko_K(\T_K^*M)\to \hat R(K)$ is a morphism of $R(K)$ module
:  for $E$ a representation space of $K$,
\begin{equation}\label{eq:indextimesVmu}
\indice_{K}^M(\sigma\otimes E)= \indice_{K}^M(\sigma)\otimes E.
\end{equation}

Any elliptic symbol  is $K$-transversally
elliptic, hence we have a restriction map
$\K_{K}^0(\T^* M)\to \Ko_{K}(\T_{K}^*M)$, and a commutative
diagram
\begin{equation}\label{indice.generalise}
\xymatrix{ \Ko_{K}(\T^* M) \ar[r]\ar[d]_{\indice_K^M}
&
\Ko_{K}(\T_{K}^*M)\ar[d]^{\indice_K^M}\\
R(K)\ar[r] & \Rfor(K)\ .
   }
\end{equation}

\bigskip

Let us explain how the index map $\indice_K^M$ is defined when $M$ is non-compact. Let $U$ be a non-compact $G$-manifold.
Lemma 3.6 and Theorem 3.7 of \cite{Atiyah74} tell us that for any open $K$-embedding $j:U\croc N$ into a compact manifold
we have a pushforward map $j_*:\K_{K}^0(\T_{K}^*U)\to \K_{K}^0(\T_{K}^*N)$ such that the composition
$$
\K_{K}^0(\T_{K}^*U)\stackrel{j_*}{\longrightarrow} \K_{K}^0(\T_{K}^*N)\stackrel{\mathrm{Index}_K^N}{\longrightarrow}\Rfor(K)
$$
is independant of the choice of $j:U\croc N$: this map is denoted $\indice^U_K$.

When $M$ is a {\em non-compact} manifold, the map $\indice_K^M: \K_{K}^0(\T_{K}^*M)\to \Rfor(K)$ is defined as follows.
Any class $\sigma\in \K_{K}^0(\T_{K}^*M)$ is represented by a symbol on $M$ with a characteristic set $\Char(\sigma)\subset \T^* M$ intersecting
$\T^*_K M$ in a compact set. Let $\Ucal$ be a $K$-invariant relatively compact open subset of $M$ such that
$\Char(\sigma)\cap \T^*_K M\subset \T^*\Ucal$: hence the restriction $\sigma\vert_\Ucal$ defines a class in $\K_{K}^0(\T_{K}^*\Ucal)$. The index map
$\indice_K^\Ucal$ is well-defined since there exists an embedding $\Ucal\croc N$ into a compact $K$-manifold  (see Lemma 3.1 in \cite{pep-RR}).
We take
$$
\indice_K^M(\sigma):=\indice_K^\Ucal(\sigma\vert_\Ucal).
$$

In particular, if $V$ is a vector space with a linear  action of a compact group $K$, we can define the index of any $\sigma\in \K_K^0(\T_K^*V)$.

\medskip

{\bf Remark :} In the following, the manifold $M$ will carry a $K$-invariant Riemannian metric and we will denote by
$\nu\in \T^*M \mapsto \tilde{\nu}\in \T M$ the corresponding identification.

\section{Functoriality}

\subsection{Invariant part}

We will use the following obvious remark. Let $\sigma \in \Gamma(\T^* M,\hom(\Ecal^{+},\Ecal^{-}))$ be a $K$-transversally elliptic
symbol on $M$.  We assume that an invariant element $b\in \kgot$ acts trivially on $M$. We can then define the subbundles
$[\Ecal^{\pm}]^b\subset \Ecal^{\pm}$ and the $K$-transversally elliptic symbol
$\sigma^b \in \Gamma(\T^* M,\hom([\Ecal^{+}]^b,[\Ecal^{-}]^b))$ which is induced by $\sigma$.

\begin{lem}\label{lem:triv}
We have  $[\indice_K^M(\sigma)]^K=[\indice_K^M(\sigma^b)]^K$. In particular \break $[\indice_K^M(\sigma)]^K=0$ if $[\Ecal^{\pm}]^b=0$.
\end{lem}
\begin{proof}
We use a $K$-equivariant decomposition of vector bundles $\Ecal^{\pm}=[\Ecal^{\pm}]^b\oplus \Fcal^{\pm}$. The symbol
$\sigma$ is thus equal to $\sigma^b\oplus \sigma'$, where  the symbol $\sigma' \in\Gamma(\T^* M,\hom(\Fcal^{+},\Fcal^{-}))$ is  also $K$-transversally elliptic. The index map is additive so $[\indice_K^M(\sigma)]^K=[\indice_K^M(\sigma^b)]^K+[\indice_K^M(\sigma')]^K$.
The space $[\indice_K^M(\sigma')]^K$ is constructed as the (virtual) subspace of invariant  $C^{\infty}$-sections of the bundle $\Fcal^{\pm}$ which are solutions of a $K$-invariant pseudo-differential operator on $M$ with symbol $\sigma'$.
But, as the action of $b$ is trivial on the basis, and
$[\Fcal^{\pm}]^b=\{0\}$, the space of
$b$-invariant  $C^{\infty}$-sections of the bundle $\Fcal^{\pm}$ is reduced to $0$. Hence $[\indice_K^M(\sigma')]^K=0$.
\end{proof}

\subsection{Free action}\label{subsec:freeaction}

Let $G$ and $K$ be two compact  Lie groups. Let $P$ be a
manifold provided with an action of $G\times K$. We assume that the action of $K$ is free. Then the manifold $M:=P/K$ is
provided with an action of $G$ and the quotient map $\pi:P\to M$
is a $G$-equivariant fibration.  The action of $K$ on the bundle $\T^*_K P$ is free and the
quotient $(\T^*_K P)/K$ admits a canonical identification with $\T^*M$. Then we still denote by
$$
q:\T^*_K P\to \T^* M
$$
the quotient map by $K$.
The quotient map $q$ induces an isomorphism
$$
q^*: \Ko_G(\T^*_{G} M)\to \Ko_{G\times K}(\T^*_{G\times K} P)
$$
that makes $\Ko_G(\T^*_{G} M)$ into a $R(K)$-module. For $V_\mu$ an irreducible representation of $K$, we denote by $\Vcal_\mu$ the complex
vector bundle $P\times_K V_\mu$ over  $M$.

\begin{theo}[Free action property] \label{th:free-action}
We have the following equality in $\hat{R}(G\times K)$
$$
\indice^P_{G\times K} (q^*\sigma)=\sum_{\mu\in \hat{K}} \indice^M_G(\sigma\otimes \Vcal_\mu^*)\otimes V_\mu.
$$
\end{theo}

\subsection{Multiplicative property}\label{sec:Multiplicative property}

Let us recall the multiplicative property of the index for the product
of manifolds. Consider a compact Lie group $G_2$ acting on two manifolds $M_1$ and $M_2$,
and assume that another compact Lie group $G_1$ acts on $M_1$ commuting with the action of $G_2$.
The external product of symbols induces
a multiplication (see (\ref{eq:product})):
$$
\boxtimes:\Ko_{G_1\times G_2}(\T^*_{G_1} M_1)\times \Ko_{G_2}(\T^*_{G_2} M_2)
\longrightarrow\Ko_{G_1\times G_2}(\T^*_{G_1\times G_2} (M_1\times M_2)).
$$

\medskip

Since $\T^*_{G_1\times G_2}(M_1\times M_2)\neq \T^*_{G_1} M_1\times\T^*_{G_2} M_2$ in general, the definition of the  product
$[\sigma_1]\boxtimes[\sigma_2]$ of transversally elliptic symbols needs some care: we have to take a representative
$\sigma_2$ of $[\sigma_2]$ which is almost homogeneous.

\begin{theo}[Multiplicative property] \label{theo:multiplicative-property}
For any $\sigma_1\in \Ko_{G_1\times G_2}(\T^*_{G_1} M_1)$ and
any $\sigma_2\in\Ko_{G_2}(\T^*_{G_2} M_2)$, we have
$$
\indice_{G_1\times G_2}^{M_1\times M_2}(\sigma_1\boxtimes \sigma_2)
=\indice_{G_1\times G_2}^{M_1}(\sigma_1)\indice_{G_2}^{M_2}(\sigma_2).
$$
\end{theo}

In the last theorem, the product of  $\indice_{G_1\times G_2}^{M_1}(\sigma_1)\in \hat{R}(G_1\times G_2)$ with \break
$\indice_{G_2}^{M_2}(\sigma_2)\in \hat{R}(G_2)$ is well defined since $\indice_{G_1\times G_2}^{M_1}(\sigma_1)$
is {\em smooth} relatively to $G_2$.

\subsection{Fiber product}\label{sec:fiber-product}

Let $\Vcal\to M$ be a $K$-equivariant Euclidean vector bundle.

Suppose that a  central connected subgroup $\Tbb\subset K$ acts trivially on $M$, so the torus $\Tbb$ acts linearly on
the fibers of $\Vcal$. Let $V:=\Vcal\vert_{m_o}$ be a fiber which is equipped with a linear action of $\Tbb$,
and an $\Tbb$-invariant scalar product. We denote  by $O\subset O(V)$ the group formed by  orthogonal endomorphisms of $V$ commuting with the $\Tbb$-action. Let $P_O\to M$ be the $O$-principal bundle formed by the orthogonal linear maps
$f:V\to \Vcal\vert_m$ commuting with the $\Tbb$-action. Then $P_O$ is a $K\times O$ manifold with free action of $O$.

\begin{rem}
If $\Vcal\to M$ is a $K$-equivariant Hermitian vector bundle,
we  consider the group $U\subset U(V)$ formed by the unitary endomorphisms of $V$ commuting with the $\Tbb$-action,
and the $U$-principal bundle $P_U\to M$ formed by the unitary linear maps
$f:V\to \Vcal\vert_m$ commuting with the $\Tbb$-action.
\end{rem}

We have two principal fibrations : $q_M:P_O\to M:=P_O/O$ and $q_\Vcal:P_O\times V\to \Vcal:= (P_O\times V)/O$.
These quotient maps induce the following  isomorphisms : $q_M^*:  \Ko_K(\T^*_{K} M)\to \Ko_{K\times O}(\T^*_{K\times O} P_O)$
and $q_\Vcal^* : \Ko_K(\T^*_{K} \Vcal)\to \Ko_{K\times O}(\T^*_{K\times O} (P_O\times V))$.

 Let us use the multiplicative property (see Section \ref{sec:Multiplicative property}) with the groups
 $G_2=K\times O, G_1=\Tbb$ acting on the manifolds
 $M_1=V, M_2= P_O$
(the subgroup $K$ of $G_2$ acting trivially on $V$).

 We have a product
 $$
\boxtimes:   \Ko_{K\times O}(\T^*_{K\times O} P_O)\times \Ko_{\Tbb\times O}(\T^*_{\Tbb} V)\longrightarrow
  \Ko_{\Tbb\times K\times O}(\T^*_{\Tbb\times K\times O} (P_O\times V)).
$$
Let us consider the isomorphism
\begin{equation}\label{eq:Q-2}
q^*_\Vcal:\KK_{\Tbb\times K}(\T^*_{\Tbb\times K} \Vcal)\to \KK_{\Tbb\times K\times O}(\T^*_{\Tbb\times K\times O} (P_O\times V)).
\end{equation}

Since the action of $\Tbb$ on $\Vcal$ comes from the $K$-action, we have $\T^*_{\Tbb\times K} \Vcal=\T^*_{K} \Vcal$,
and then we can define a forgetful map $\mathrm{F}:\Ko_{\Tbb\times K}(\T^*_{K} \Vcal)\to \Ko_{K}(\T^*_{K} \Vcal)$.

\begin{defi}
The fiber product
 \begin{equation}\label{produit-fibre-H-2}
\lozenge:   \Ko_{K}(\T^*_{K} M)\times\Ko_{\Tbb\times O}(\T^*_{\Tbb} V) \longrightarrow
  \Ko_{K}(\T^*_{K}\Vcal)
\end{equation}
is defined by the relation
$$
\sigma\,\lozenge\,\tau:=\mathrm{F}\circ (q^*_\Vcal)^{-1}\left(q^*_M(\sigma)\boxtimes \tau   \right).
$$
\end{defi}

A  symbol representing $\sigma\,\lozenge\,\tau$ can be constructed as follows.
 We write an element $n\in  \Vcal$ as $(x,v)$ with $x\in M$ and $v\in \Ncal\vert_x$.
We choose a $K$-invariant connection on the bundle $\Ncal\to M$, which allows us to parameterize
 $\T^* \Ncal=\{(n,(\xi,\nu))\}$, with
$\xi\in \T^*_xM$, $\nu\in \Ncal^*\vert_x$ (here $\xi$ is lifted to an element of $\T_n^* \Ncal$ using the connection).
Let
$\sigma(x,\xi)$
be a symbol on $M$ representing $[\sigma]$ and $\tau(v,\eta)$ an almost homogeneous  symbol on $V$ representing $\tau$.
Then
$$(\sigma\,\lozenge\,\tau )((x,v),(\xi,\nu))=\sigma(x,\xi) \boxtimes \tau(v,\eta).$$

If $E$ is a representation space of  the group $O$, we denote
by $\left[\left[E \right]\right]$  the vector bundle $P_O\times_O E$ over $M$, and by
$\sigma \otimes \left[\left[E \right]\right]$ the symbol $\sigma$ twisted by the vector bundle
$\left[\left[E\right]\right]$.

Consider $\indice^V_{O}(\tau)$ as a representation space for $O$. As $\tau$ is
$\Tbb$-transversally elliptic,
we have $\indice^V_{O}(\tau)=\sum_{\alpha\in \hat \Tbb} E_\alpha$ where each $E_\alpha$ is a
 finite dimensional representation of $O$ where the  action of $\Tbb$  on $E_\alpha$ is scalar and acts via the character $\alpha$.
  We still denote  by
$\left[\left[\indice^V_{O}(\tau)\right]\right]$ the sum
$\sum_{\alpha\in\hat{\Tbb}} \left[\left[E_\alpha\right]\right]$ of finite dimensional $K$-vector bundles on $M$.

If we combine Theorems \ref{th:free-action} and \ref{theo:multiplicative-property}, we get

\begin{lem} \label{lem:indice-lozenge}
 $\indice^\Vcal_K(\sigma \lozenge \tau)$ is equal to
$\indice^M_K(\sigma\otimes \left[\left[\indice^V_{ O}(\tau)\right]\right])$ in  $\hat{R}(K)$.
\end{lem}

Let us explain briefly why $\indice^M_K(\sigma\otimes \left[\left[\indice^V_{O}(\tau)\right]\right])$
is well defined in $\hat R(K)$. By definition $\indice^M_K(\sigma\otimes \left[\left [\indice^V_{O}(\tau)\right]\right])$
is equal to the infinite sum
\begin{equation}\label{eq:infinite-sum}
\sum_{\alpha\in\hat{\Tbb}} \indice^M_K(\sigma\otimes \left[\left[E_\alpha\right]\right]).
\end{equation}
Since $\Tbb$ acts trivially on $M$, and commutes with $K$, it is not hard to see that for each
$\mu\in\hat{K}$ the multiplicity of
$V_\mu$ in $\indice^M_K(\sigma\otimes\left[\left[E_\alpha\right]\right])$ is zero when
$\|\alpha\|$ is large enough.  This shows that the sum (\ref{eq:infinite-sum}) converges in
$\hat{R}(K)$.

\medskip

An interesting case is already when the group $\Tbb$ is trivial (see Section \ref{sec:directimage}).
 Any elliptic element $\tau\in\Ko_{O(V)}(\T^* V)$ defines a pushforward map
$\Ko_{K}(\T^*_{K} M)\to\Ko_{K}(\T^*_{K}\Vcal), \sigma\mapsto \sigma\lozenge \tau$ such that
$\indice^\Vcal_K(\sigma \lozenge \tau)=\indice^M_K(\sigma\otimes [ [\indice^V_{O(V)}(\tau)] ])$.

\subsection{Induction property of the index}

Let $Y$ be a $H$-manifold and let $M:=K\times_H Y$. Since, at the level of the tangent spaces, we have the decomposition
$\T M\vert_Y\simeq \T Y\oplus Y\times \kgot/\hgot$,  one sees that $\T^*_K M$ is diffeomorphic to
$K\times_H \T^*_H Y$. Hence we have an isomorphism $\Ko_K(\T_K^* M)\simeq \Ko_H(\T_H^* Y)$ : if $\sigma$
is a $K$-transversally elliptic symbol on $M$, we denoted $\sigma\vert_Y$ the corresponding $H$-transversally elliptic symbol on $Y$.

\begin{prop}\label{Induction formula}For any $K$-transversally elliptic symbol $\sigma$ on $M$
$$
\indice_K^M(\sigma)=\ind_H^K\left(\indice_H^Y(\sigma|_Y)\right).
$$
\end{prop}

\begin{proof}
Let us consider $N=K\times Y$ with action of $H$ on $K\times Y$ by $(k,y)\cdot h=(kh^{-1},hy)$.
Then $\T^* N=\T^* K\times \T^* Y$, and $\sigma$ defines a $K\times H$ transversally elliptic symbol $\tilde \sigma$ on $N$.
The restriction of the symbol $\tilde \sigma$  to
$\T^*_{K\times H} N=(K\times\{0\}) \times \T^*_H Y$  is equal to $\sigma_0\boxtimes \sigma_Y$, where $\sigma_0$ is the symbol of the zero operator on $C^{\infty}(K)\simeq\sum_{\lambda\in \hat K}V_\lambda \otimes V^*_\lambda$.
Here we consider $C^{\infty}(K)$ with left action of $K$ and right action of $H$.
Thus in $\hat R(K)\otimes \hat R(H)$, we have
$$
\indice_{K\times H}^N(\tilde \sigma)= \indice_{K\times H}^N(\sigma_0\boxtimes \sigma_Y)=\sum_{\lambda\in  \hat K}
V_\lambda\otimes (V_\lambda^*|_H\otimes \indice_H^Y(\sigma|_Y)).
$$
The index of $\sigma$ is obtained by taking the $H$-invariant part of  $\indice_{K\times H}^N (\tilde \sigma)$.
We thus obtain the proposition.

\end{proof}\bigskip

\medskip

Conversely, a $H$-equivariant symbol $\sigma_Y$  on $Y$ allows to define an induced $K$-equivariant symbol $\sigma$ on $K\times_H Y$ with restriction $\sigma_Y$, via
$\sigma(y,q,\eta)=\sigma(y,\eta)$ on
$\T^* M|Y\simeq\qgot^* \oplus T^*_y Y$.
Then we have
$$\indice_K^M(\sigma)=\ind_H^K\left(\indice_H^Y(\sigma_Y)\right).$$

\section{ Clifford bundles and Dirac operators}\label{sec:spin-c-structure}

\subsection{Clifford algebra}

Let $V$ be an Euclidean space. We denote by  $\mathrm{Cl}(V)$  its Clifford algebra.
We follow the conventions of \cite{B-G-V}. Denote by $(v,w)$ the scalar product of two vectors $v,w$ of $V$.
Then, in $\mathrm{Cl}(V)$, $vw+wv=-2(v,w)$. In particular $v^2=-\|v\|^2$ for any $v\in V\subset \mathrm{Cl}(V)$.

When $V$ is even dimensional and oriented, we define $\Gamma=i^{n/2}e_1 e_2\cdots e_n$ in $\Cl(V)\otimes \C$
modulo the choice of an oriented orthonormal basis $e_1,\dots, e_n$. We remark that $\Gamma$ depends only of the orientation.
Moreover we have $\Gamma^2=1$ and $\Gamma v=-v \Gamma$ for any $v\in V$.

If $E$ is a $\mathrm{Cl}(V)$-module, the Clifford map is denoted $\clif_E : \mathrm{Cl}(V)\to \End(E)$ : we have
$\clif_E(v)^2=-\|v\|^2 \mathrm{Id}_E$ for any $v\in V$. A $\Z/2\Z$-graduation on $E$ is a vector space decomposition
$E=E^+\oplus E^-$ such that the map $\clif_E(v):E\to E$ are odd maps for any $v\in V$. Equivalently, a grading on $E$ is defined
by an  involution $\theta:E\to E$ such that
$\clif_E(v)\circ \theta=-\clif_E(v)$ for any $v\in V$.

For example, when $V$ is even dimensional and oriented, we may consider the $\Z/2\Z$-graduation on $E$ defined by
the involution $\Gamma_E:=\clif_E(\Gamma)$. This graduation will be called the canonical graduation
of the $\mathrm{Cl}(V)$-module $E$.

Let us give the first example of Clifford module. Consider the real vector space $\bigwedge_\R V$.
We denote by $m(v): \bigwedge_\R V \to \bigwedge_\R V$  the exterior multiplication by an element of $V$ ,
and by  $\iota(v): \bigwedge_\R V \to \bigwedge_\R V$  the contraction: $\iota(v)$ is
the odd derivation of $\bigwedge_\R V$ such that $\iota(v_2)v_1=(v_2,v_1)$ for $v_1\in V\subset \bigwedge_\R V$.
Then $\clif(v)=m(v)-\iota(v)$ satisfies
$\clif(v)^2=-\|v\|^2$.
Thus the complexified space  $(\bigwedge_\R V)\otimes\C\simeq \bigwedge_\C (V\otimes\C) $ is a Clifford module.

If $E$ is a graded Clifford module, and $W$ a graded complex vector space, the space $E\otimes W$
with Clifford action $\clif_E(v)\otimes {\rm Id}_W$ is again a graded Clifford module where
$(E\otimes W)^+=E^+\otimes W^+\bigoplus E^-\otimes W^-$ and $(E\otimes W)^-=E^-\otimes W^+\bigoplus E^+\otimes W^-$.
It will be  called the twisted Clifford module  of $E$ with twisting space $W$.

If $E$ is a graded Clifford module,  it is convenient to denote $-E$ the Clifford module with opposite grading. We can see
$-E$ as the tensor of $E$ with the graded vector space $W$ where $W^+=\{0\}$ and $W^-=\C$.

\begin{prop}
Let $V$ be an Euclidean space of even dimension $n=2\ell$.  There exists a $\mathrm{Cl}(V)$-module
$S$ of such that the map $\clif_S :\mathrm{Cl}(V)\to \End(S)$ extends to an
isomorphism of complex algebras :
$$
\mathrm{Cl}(V)\otimes\C\simeq \End(S).
$$
If an orientation of $V$ is given, we define $S^+=\{s\in S\,\vert\, \Gamma s= s\}$ and
$S^-=\{s\in S\,\vert\, \Gamma s=- s\}$.
\end{prop}
Then $S$ is unique, up to isomorphism.
We call $S$ the spinor module.
There is, up to a multiplicative constant, a unique Hermitian form on $S$ such that the operators $\clif_S(v)$ are skew-adjoint.

For example if $V=\R^2$, with standard inner product, then $S$ is $\C\oplus \C$, with Clifford action
$$\clif(xe_1+ye_2)= \left(\begin{array}{cc}
                  0 & -(x-iy) \\
                  x+iy& 0 \\
                \end{array}
              \right).$$

Let $\spin(V)$ be the spinor group. We denote by $\eta: \spin(V)\to SO(V)$ the covering map. The kernel of $\tau$ has two elements.
The spinor group $\spin(V)$   acts in $S$, preserving $S^+,S^-$. We denote by $\rho$ the spin representation of
 $\spin(V)$. It satisfie $\rho(g) \clif(v)\rho(g^{-1})=\clif(\eta(g)v)$.

\bigskip

Here is a construction  of the spinor module. Consider $V\otimes\C$ and extend the form $(\,,)$ by complex bilinearity to $V\otimes\C$.
 Choose a decomposition $V\otimes\C=L_1\oplus L_2$ in two totally isotropic subspaces for the form $(\,,)$. Thus $L_1,L_2$ are of dimension $n/2$.
We consider on the $\Z_2$-graded complex vector space $\bigwedge L_1$ the
following odd operators: the exterior multiplication $m(v_1)$ by an element of $L_1$,
and the contraction $\iota(v_2)$ by an element of $L_2$: $\iota(v_2)$ is
the odd derivation of $\bigwedge L_1$ such that $\iota(v_2)v_1=(v_2,v_1)$ for $v_1\in L_1$.
The Clifford action of $V_\C$ on $\bigwedge L_1$ is defined by the
formula
\begin{equation}\label{eq:cliffordaction}
 \clif(v_1+v_2)=\sqrt{2}(m(v_1) - \iota(v_2)).
\end{equation} Then $\clif(v)$ is an odd operator on $\bigwedge L_1$
such that $\clif(v)^2=-(v,v)\Id$.
It is easy to see that $\bigwedge L_1$ is an irreducible Clifford module. Thus $\bigwedge L_1$ with action $\clif(v)$ as above is ``the" spinor module.

Let $J$ be an endomorphism of $V$ such that $J^2=-1$ ( $J$ defines a complex structure on $V$) and such that $(Jv,Jv)=(v,v)$.
Equivalently $h(v,w)=g(v,w)-i g(Jv,w)$ is a Hermitian structure on the complex vector space $(V,J)$.
Then $J$ defines also an orientation $o(J)$ of $V$. If $e_1,e_2,\ldots, e_{n/2}$ is a basis of $V$ as a complex space, then we take as oriented basis of $V$ the basis $e_1,Je_1,e_2,Je_2,\ldots, e_{n/2}, Je_{n/2}$.
We denote by $V_J$ the space $V$ considered as a complex vector space
and we may also denote by $\bigwedge_J V=\bigwedge V_J$ the exterior space of the complex space $V_J$.

We denote by
$$
V^{+}=\{v\in V\otimes\C\,\vert\, Jv=iv\},\qquad V^{-}=\{v\in V\otimes\C\,\vert\, Jv=-iv\}.
$$
These are two totally isotropic spaces for $(\,,)$ and we can take $S=\bigwedge V^{+}$ as the irreducible Clifford module .
In complex geometry, it is more customary to take $S=\bigwedge (V^-)^*$, but the space $(V^-)^*$ is isomorphic to $V^+$ by the duality $g(\,,)$.

The map $v\to v-i Jv$ from $V$ to $V^+$ gives an isomorphism of the complex vector spaces $V_J$ with $V^+$.
Then $\bigwedge V_J$ is isomorphic to $\bigwedge V^+$, and we can realize $S$ as $\bigwedge V_J$.
The grading of $S$ induced by the orientation $o(J)$ of $V_J$ is
then $S^+=\bigwedge^{even} V_J$  and $S^-=\bigwedge^{odd} V_J$.

If $H$ is a Hermitian vector space, in some situation (see Section \ref{sec:atiyah-symbol}), it is convient to choose
as spinor space $\bigwedge\overline H$,
where $\overline{H}$ denote the space $H$, with complex structure $-J$. Then the grading
$\bigwedge^{even}\overline{H}\oplus \bigwedge^{odd}\overline{H}$ corresponds to the orientation $o(-J)=(-1)^{\dim H}o(J)$.

Consider the subgroup $U(V)$ of $SO(V)$ consisting  of complex transformations of $(V,J)$. An element
of $U(V)$ leaves stable the space $V^{+}$, and we obtain a natural action $\lambda$ of $U(V)$ in $\bigwedge V^{+}$.
Let us compare the natural action $\lambda$   and the restriction $\rho$ of the spin representation of $\spin(V)$
to $\tau^{-1}(U(V))$.

\begin{lem}\label{lem: lambdarho}
For $g\in \tau^{-1}(U(V))\subset \spin(V)$,  $\lambda(\tau(g))=\delta(g) \rho(g)$, where $\delta(g)$ is the character of $\tau^{-1}(U(V))$ such that
 $\delta(g)^2=\det_{V^{+}}(g)$.
\end{lem}

\medskip

We often will use complex structures on real vector spaces defined by the following procedure.
\begin{defi}\label{defi:Jbeta}
Let $N$ be a real vector space and $\beta$ a transformation of $N$, such that $-\beta^2$ is diagonalizable
 with non negative eigenvalues. Define

$\bullet$ the transformation $|\beta|$ of $N$ by $|\beta|=\sqrt{-\beta^2}$,

$\bullet$ and the complex structure $J_\beta$ on $N/\ker(\beta)$ by
$J_\beta=\beta |\beta|^{-1}.$
\end{defi}

For example if $V$ is an Euclidean space and $\beta$ a skew-symmetric transformation of $V$, then $-\beta^2$ is diagonalizable with non negative eigenvalues. By definition of $J_\beta$, the transformation $\beta$ induces on $N/\ker(\beta)\simeq \mathrm{Image}(\beta)$ a complex diagonalizable transformation, and  the list of its complex eigenvalues is
$[i a_1,\ldots, i a_\ell]$ where the $a_k$ are strictly positive real numbers.

Similarly if $\Ncal\to B$ is a  Euclidean vector bundle with a fibrewise bijective linear endomorphism $\Lcal(\beta)$ which is anti-symmetric relatively to the metric,
we denote by $\Ncal_{\JJbeta}$ the vector bundle $\Ncal$ over $B$  equipped with the complex structure
$J_\beta$. Then $\Ncal_\JJbeta$ is a Hermitian vector bundle over $B$, and we say that $\Ncal_\JJbeta$ is the complex vector bundle
$\Ncal$ polarized by $\beta$.

\bigskip

The following proposition follows from the Schur lemma.
It is  used  to ``divide"  a Clifford module by the irreducible Clifford module $S$.

\begin{prop}\label{prop: E=SW}
When $\dim V$ is even, any $\mathrm{Cl}(V)$-module $E$ has the following decomposition
\begin{equation}\label{eq:iso-graded-clifford-module}
E\simeq S\otimes \hom_{\mathrm{Cl}(V)}(S,E)
\end{equation}
where $\hom_{\mathrm{Cl}(V)}(S,E)$ is the vector space
of  $\mathrm{Cl}(V)$-complex linear maps from $S$ to $E$.  The grading of $E$, and a grading on $S$ coming
from an orientation on $V$  induces a grading on $\hom_{\mathrm{Cl}(V)}(S,E)$ so that
the isomorphism (\ref{eq:iso-graded-clifford-module}) is an isomorphism of graded Clifford module.
\end{prop}

The space $W=\hom_{\mathrm{Cl}(V)}(S,E)$ is called the twisting space.
We write $E=S\otimes W$
and the Clifford action is just on the spinor module $S$:
$$\clif_E(v)= \clif_S(v)\otimes {\rm Id}_W.$$

\begin{exam}
$\bullet$ If $J_1,J_2$ are two complex structures on $V$, we may compare the spinor bundle $S_k=\bigwedge V_{J_k}$. We have
$S_1\simeq\epsilon \,S_2$ where $\epsilon\in \{\pm 1\}$ is the ratio between the orientations $o(J_1)$ and $o(J_2)$.

$\bullet$ If $V$ is even dimensional, the twisting space $W$ for the Clifford module $\bigwedge_\R V\otimes\C$ is
$-S$.
\end{exam}

We implicitly always choose Hermitian products on $E$ obtained by tensor product of the Hermitian structure on $S$ and a Hermitian structure on $W$. Then the operators $\clif(v)$ are skew-adjoint.

If $V=V_1\oplus V_2$ is an orthogonal decomposition, we have $\mathrm{Cl}(V)\simeq \mathrm{Cl}(V_1)\otimes \mathrm{Cl}(V_2)$.
So, if $E_1,E_2$ are graded  Clifford modules for $V_1,V_2$, then  $E=E_1\otimes E_2$ is a graded Clifford module for $V$,
with Clifford action $\clif(v_1\oplus v_2)=\clif(v_1)\boxtimes \clif(v_2)$.

Reciprocally, if $V_1$ is even dimensional and oriented, any  graded Clifford module $E$ on $V=V_1\oplus V_2$
has the decomposition $E\simeq S_1\otimes E_2$, where $S_1$ is the spinor module for $\mathrm{Cl}(V_1)$ with its
canonical grading, and  $E_2=\hom_{\mathrm{Cl}(V_1)}(S_1,E)$ is a graded $\mathrm{Cl}(V_2)$-module.

\subsection{Euclidean vector bundles and Clifford bundles}

Consider now the case of an Euclidean vector bundle $\Vcal\to M$ over a manifold $M$. Let $\mathrm{Cl}(\Vcal)\to M$ be the associated Clifford algebra bundle. A complex $\Z/2\Z$-graded vector bundle $\Ecal\to M$ is a $\mathrm{Cl}(\Vcal)$-bundle if there is a bundle algebra odd morphism
$\mathrm{c}_\Ecal : \mathrm{Cl}(\Vcal)\longrightarrow \End(\Ecal)$.
Thus, for any $x\in M, v\in \Vcal\vert_x$,
$\clif_{\Ecal_x}(v): \Ecal\vert_x^{\pm}\to \Ecal\vert_x^{\mp}$ is
such that $\clif_{\Ecal_x}(v)^2=-\|v\|^2 \mathrm{Id}_{\Ecal\vert_x}$.

Assume that $M$ is a $K$-manifold. Let $\Vcal$  be a $K$-equivariant Euclidean vector bundle on $M$, and $\Ecal$ a $\mathrm{Cl}(\Vcal)$-bundle. Then we will always assume that $\Ecal^{\pm}\to M$ are  $K$-equivariant  vector bundles and that
$\mathrm{c}_\Ecal$ is a morphism of $K$-equivariant vector bundles.
We will say that $\Ecal$ is a $K$-equivariant Clifford bundle.

If $\Wcal$ is any $K$-equivariant  graded complex vector bundle on $M$, and $\Ecal$ a graded Clifford bundle, then
$\Ecal\otimes \Wcal$ with Clifford action
$\clif_\Ecal\otimes {\rm Id}_{\Wcal}$ will be called the twisted Clifford bundle  of $\Ecal$ by $\Wcal$.

\begin{defi}
A $\spinc$ bundle on $\Vcal$ is a $\mathrm{Cl}(\Vcal)$-bundle $\Scal$ such that, for any $x\in M$, $\Scal\vert_x$ is an irreducible Clifford module
for $\Vcal\vert_x$. An orientation on the vector bundle $\Vcal$ induces a (canonical) grading on $\Scal$.
\end{defi}

If $\Scal$ is a $\spinc$-bundle on $\Vcal$ equipped with its canonical grading, we can, as in the linear case, divide  any $\mathrm{Cl}(\Vcal)$-bundle $\Fcal$ by $\Scal$ : we write  $\Fcal=\Scal\otimes \Wcal$, and the twisting bundle $\Wcal=\hom_{\mathrm{Cl}(\Vcal)}(\Scal,\Fcal)$  inherits a graduation.

Let $K$ be a compact Lie group acting on  a  manifold $M$. Consider a $K$-equivariant Riemannian metric on $M$ and let
$\nu\in \T^*M \mapsto \tilde{\nu}\in \T M$ be the corresponding bundle isomorphism.

\begin{defi}
A Clifford bundle on $M$ is a  graded Clifford bundle $\Ecal=\Ecal^+\oplus \Ecal^-$ for the Clifford algebra  $\mathrm{Cl}(\T M)$.

Let $\Ecal$ be  a Clifford bundle on $M$.
We define the symbol
$\sigma(M,\Ecal)\in \Gamma(\T^* M, \hom(\Ecal^+,\Ecal^-))$ by
$$\sigma(M,\Ecal)(x,\nu)=\clif_{\Ecal_x}(\tilde{\nu}): \Ecal\vert_x^+\to \Ecal\vert_x^-.$$
\end{defi}

We always choose $K$-invariant Hermitian structures
on $\Ecal^{\pm}$ so that $c_x(\tilde{\nu})$ is skew-adjoint.
Thus the adjoint of the symbol
$\sigma(M,\Ecal)$ is the symbol $-\sigma(M,\Ecal^{op})$ (in $\Ecal^{op}$,  we have interchanged the parity on $\Ecal^+$ and $\Ecal^-$).
We have
$$\sigma(M,\Ecal)(x,\tilde{\nu}) \sigma(M,\Ecal^{op})(x,\tilde{\nu})=-\|\tilde{\nu}\|^2_x$$ so that,  if $M$  {\bf is compact},
$\sigma(M,\Ecal)$ is {\bf an elliptic symbol}.

\begin{rem}

By Bott periodicity theorem, any $K$-equivariant elliptic symbol on  a compact even dimensional oriented manifold $M$  is equal  to a symbol $\sigma(M,\Ecal)$,  for some equivariant Clifford bundle $\Ecal$, in the equivariant $K$-theory of $\T M$.
We will prove a localization formula for
the symbols $\sigma(M,\Ecal)$.
They could be rephrased for any elliptic symbol, but this would require more notational framework.
\end{rem}

\subsection{Choice of metrics}\label{sub:choice}

Let us consider the role of the metric in our setting. Let $(M,\Vcal,\Ecal)$ be a $K$-equivariant Clifford data : $\Vcal\to M$ is an Euclidean
vector bundle and $\Ecal\to M$ is a graded $\mathrm{Cl}(\Vcal)$-bundle. The metric $g$ on the Euclidean vector bundle $\Vcal\to M$
determines a bundle map $g_\sharp:\Vcal\to\Vcal^*$. If we take another metric $g'$ on $\Vcal\to M$ the bundle map
$A:=(g_\sharp)^{-1}\circ (g')_\sharp$ is symmetric positive definite for $g$. Hence, $C=A^{1/2}$ is a well-defined bundle map,
symmetric with respect to $g$. It defines an isometry $C:(\Vcal,g)\to (\Vcal,g')$, and then a canonical isomorphism
$\mathrm{Cl}(\Vcal,g)\simeq \mathrm{Cl}(\Vcal,g')$. So we will often speak of a Clifford bundle $\Ecal$ on $\Vcal$ without specifying the metric on $\Vcal$.

In particular if $\pi:\Hcal\to B$ is a Hermitian vector bundle  on  a Riemannian $K$-manifold $B$, using a $K$-invariant connection, we can choose a metric on $\T\Hcal$ such that $\T\Hcal$ is isomorphic to the orthogonal direct sum  $\pi^*\T B\oplus \pi^*\Hcal$ of Euclidean vector bundles.
Then any Clifford module for $\T\Hcal$ is of the form
$$\pi^*\Ecal\vert_B\otimes \pi^*\bigwedge \overline{\Hcal},$$ where
$\Ecal\vert_B$ is a Clifford module for $\T B$,
 and, for $\pi(n)=x$,  the Clifford map $\clif_n$ is equal to
$\clif_x^1\boxtimes\clif_x^2$, where
$$\clif^1_x : \T_x B\longrightarrow \End(\Ecal\vert_B\vert_x),
$$
$$
\clif^2_x : \Hcal\vert_x\longrightarrow \End(\bigwedge\overline{\Hcal}\vert_x).
$$

\subsection{The Bott symbol and direct image}\label{sec:directimage}

Let $N$ be an Euclidean vector space. The dual space $N^*$ inherits an Euclidean structure through the identification $N^*\to N, \eta\to\tilde{\eta}$ given by the relation :
$\langle \eta,n\rangle = (\tilde{\eta},n)$.
 The vector space $N\oplus N^*$ is an Euclidean space of even dimension that we identify to $N_\C$ by $n\oplus \eta\to n\oplus i \tilde{\eta}$.
 The group $O(N)$ acting diagonally on $N\oplus N^*$   becomes a subgroup of $U(N_\C)$, and acts on $S_N:=\bigwedge N_\C$ via
 the natural representation of $U(N_\C)$ in  $\bigwedge N_\C$.

Recall the definition of the Bott symbol.
We denote by $\clif_{N_\C}$ the Clifford representation
of $N_\C$ in $S_N$. Then $\clif_{N_\C}(n\oplus i n')$  satisfies $\clif_{N_\C}(n\oplus in')^2=-(\|n\|^2+ \|n'\|^2) {\rm Id}$.

 Consider the trivial bundle $\Scal=N\times S_N$.
We define the symbol $\mathrm{Bott}(N_\C)\in \Gamma(\T^* N,\hom(\Scal^+,\Scal^-))$ by
$$
\mathrm{Bott}(N_\C)(n,\eta)=\clif_{N_\C}(n\oplus i\tilde{\eta}): S_N^+\to S^-_N.
$$
 Here $n\in N$, $\eta\in N^*$. The characteristic support of $\mathrm{Bott}(N_\C)$ is $\{0\}$, so  $\mathrm{Bott}(N_\C)$ is an elliptic symbol on $N$ called the Bott symbol. It is equivariant for the action of $O(N)$ and $\indice_{O(N)}^N(\mathrm{Bott}(N_\C))$ is the trivial representation of $O(N)$ \cite{Atiyah-Singer-1}.

Let $\Ncal\to \Xcal$ be an Euclidean vector bundle. The orthogonal frame bundle $P$ of $\Ncal$ is a
$\mathrm{O}(n)$-principal bundle such that $\Ncal=P\times_{\mathrm{O}(n)}N$, with $N=\R^n$.
Using the notion of fiber product  introduced in Section \ref{sec:fiber-product} with the group $S=\{1\}$, we
define the pushforward morphism $i_!:\Ko_K(\T^*_{K} \Xcal)\longrightarrow \Ko_K(\T^*_{K} \Ncal)$ by the relation
$$
i_!(\sigma):=\sigma\lozenge \mathrm{Bott}(N_\C),
$$
if $\sigma$ is a transversally elliptic symbol on $\Xcal$.

The following theorem is a consequence of Lemma  \ref{lem:indice-lozenge}.
\begin{theo}[Invariance of the index by push-forward]\label{theo:direct}
$$
\indice_K^\Xcal(\sigma)=\indice_K^\Ncal(i_!(\sigma)).
$$
\end{theo}

A  symbol  representing $i_!(\sigma)$ is as follows.
 We write an element $n\in  \Ncal$ as $(x,v)$ with $x\in \Xcal$ and $v\in \Ncal\vert_x$.
We choose a $K$-invariant connection on the bundle $\Ncal\to \Xcal$, which allows us to parameterize
 $\T^* \Ncal=\{(n,(\xi,\eta))\}$, with
$\xi\in \T^*_x \Xcal$, $\eta\in \Ncal^*\vert_x$ : here $\xi$ is lifted to an horizontal  element of $\T_n^* \Ncal$ using the connection, so
$(\xi,\eta)$ parameterize any element in $\T_n^*\Ncal$.
Let
$\sigma \in \Gamma(\T^* \Xcal,\hom(\Ecal^+,\Ecal^-))$
be a symbol on $\Xcal$. Then the symbol
\begin{equation}\label{eq:bott}
((x,v),(\xi,\eta))\mapsto \sigma(x,\xi) \boxtimes \clif_{x}(v\oplus i \tilde{\eta})
\end{equation}
represents  $i_!(\sigma)$.
Here $\clif_x$ is the Clifford representation of $\Ncal\vert_x\otimes \C$ on $\bigwedge (\Ncal\vert_x\otimes\C)$, so
$(\clif_{x}(v\oplus i \tilde{\eta}))^2=-(\|v\|^2+\|\eta\|^2)\mathrm{Id}$
and $\Char(i!(\sigma))=\Char(\sigma)$.

\subsection{Dirac operators and equivariant indices}\label{sec:equivariant index}

Let $\Ecal\to M$  be a $K$-equivariant graded Clifford bundle on a compact
Riemannian $K$-manifold  $M$ of even dimension.

\begin{defi}\label{defiQK}
We denote by $\Qcal_K(M,\Ecal)\in R(K)$ the  index of the symbol
$\sigma(M,\Ecal)$.
\end{defi}

If $K$ is the trivial group, we denote $\Qcal_K(M,\Ecal)$ simply by
$\Qcal(M,\Ecal).$
This is an integer.

We can realize
 $\Qcal_K(M,\Ecal)$ as the index of a Dirac operator.
We  choose  $K$-invariant connections  on $\Ecal^{\pm}$  such that the Dirac operator
$\Dcal_\Ecal^{\pm}: \Gamma(M,\Ecal^\pm)\to \Gamma(M,\Ecal^{\mp})$ \cite{B-G-V}
verifies
$(\Dcal_\Ecal^{+})^*=\Dcal_\Ecal^{-}$.  The principal symbol of  $\Dcal_\Ecal^{+}$ is $\sigma(M,\Ecal)$.
Thus  $\Qcal_K(M,\Ecal)$ can be realized as the difference of the two finite dimensional  representation spaces  $[\ker \Dcal^+_\Ecal]- [ \ker \Dcal^-_\Ecal]$ and we also say that $\Qcal_K(M,\Ecal)$ is the equivariant index of the Dirac  operator $\Dcal_\Ecal^+.$

For example, if $M$ is any $K$-manifold, and $\Ecal=\bigwedge
(\T M\otimes_\R {\C})$, a Dirac operator  with principal symbol $\sigma(M,\Ecal)$ can be taken as $d+d^*$ where $d$ is the de Rham differential. If $K$ is connected, $\Qcal_K(M,\Ecal)$ is a multiple of the trivial representation of $K$, the multiplicity being the Euler characteristic of $M$.

\begin{exam}\label{exam:complex}
An important example is when $M$ is a complex manifold  with complex structure $J\in \Gamma(M,\T M)$.
Then $\T M\otimes_\R \C$ breaks into two pieces, the holomorphic and anti-holomorphic tangent spaces
$$\T M\otimes_{\R} \C=\T^{1,0}M\oplus \T^{0,1}M$$ on which $J$ acts by $i$, and $-i$, respectively.
 We consider a Riemannian metric on $\T M$  such that $g(Jv,Jv)=g(v,v)$.
Let $\Wcal \to M$ be a holomorphic  vector bundle on $M$.
Then the vector bundle of $\Wcal$-valued anti-holomorphic differential forms $$\Ecal= \bigwedge (\T^{0,1}M)^*\otimes \Wcal$$ is a graded Clifford bundle.
This bundle  is isomorphic to  $\bigwedge_J \T M\otimes \Wcal$, with Clifford action described by the formula (\ref{eq:cliffordaction}).
A Dirac operator with principal symbol $\sigma(M,\Ecal)$ is
$ {\overline \partial}_{\Wcal} +({\overline \partial}_{\Wcal})^*$
where ${\overline \partial}_{\Wcal}$ is the Dolbeaut operator.
Assume that $K$ acts by holomorphic transformations on $M,\Wcal$.
Then $\Qcal_K(M,\Ecal)=\sum_{j} (-1)^j H^j(M,\Ocal(\Wcal))$ is the alternate sum of the cohomology groups of the sheaf of holomorphic sections of $\Wcal$.
\end{exam}

\begin{exam}\label{exam:highest-weight}
Let $K$ be a compact connected Lie group, and $T\subset K$ be a maximal torus. We consider the flag manifold
$\mathbb{F}:=K/T$ with base
point $\overline{e}\in K/T$. The choice of a Weyl chamber $\tgot^*_{\geq 0}\subset \tgot^*$ determines a
$T$-invariant complex structure $J$ on
$\T_{\overline{e}}\mathbb{F}\simeq\kgot/\tgot$ such that the complex $T$-module $\T^{1,0}\mathbb{F}$
is equal to  $\oplus_{\alpha>0}\C_\alpha$,
where $\{\alpha>0\}$ denotes the set of positive roots. The complex structure $J$ defines an integrable complex structure
on the flag manifold $\mathbb{F}$ (still denoted $J$). Let $\Lambda\subset \tgot^*$ be the weight lattice. Any weight
$\mu\in \Lambda$ determines an holomorphic
line bundle $[\C_\mu]:=K\times_T\C_{\mu}$ on $\mathbb{F}$. If $\mu$ is dominant, the Borel-Weil Theorem
tells us that $\Qcal_K(\mathbb{F},\bigwedge_J \T \mathbb{F}\otimes [\C_\mu])$ is equal to   the irreducible
representation of $K$ with highest weight $\mu$.
\end{exam}

\subsection{Horizontal Clifford bundle}\label{sec:horizontal}

Let $G$ and $K$ be two compact  Lie groups. Let $P$ be a
manifold provided with an action of $G\times K$. Assume that the action of $K$ is free and let $M=P/K$. We are in the situation described in Subsection \ref{subsec:freeaction}, and we use the same notations.

%
For the tangent bundle, the choice of a $G$-invariant connection on the principal bundle $P\to M$ defines an invariant
decomposition $\T P=\pi^*\T M\oplus P\times\kgot$; if a metric is chosen on $M$,
we define the metric on $P$ by declaring that this decomposition is orthogonal.

Let $\Ecal$  be a $G$-equivariant graded Clifford bundle on $M$.
We obtain from $\Ecal$  a symbol
$$\sigma_{hor}(P,\Ecal)\in \Gamma(\T^* P, \hom(\pi^*\Ecal^+,\pi^*\Ecal^-))$$
by defining
$\sigma_{hor}(P,\Ecal)(p,\eta)=\sigma(M,\Ecal)(x,\tilde{\theta})$
where $(x,\theta):=\bar{\pi}(p,\eta)$.

Assume $P$ compact.
It is clear that
$\sigma_{hor}(P,\Ecal)$ is a $K$-transversally elliptic symbol on $P$.
It will be called the horizontal Dirac symbol.
Furthermore, this symbol is $G$-equivariant, and
 for any irreducible
representation $V_\mu$ of $K$, the  finite dimensional space
$$
\left[\indice_{G\times K}^P(\sigma_{hor}(P,\Ecal))\otimes V_\mu^*\right]^K
$$
carries a  representation of $G$.

We form the $G$-equivariant complex vector bundle $\Vcal_\mu:= P\times_K
V_\mu$ on $M$, and the twisted Clifford bundle $\Ecal\otimes \Vcal_\mu$.
 We consider the elliptic symbol
$\sigma(M,\Ecal\otimes  \Vcal_\mu)$ on $M$.

The following result  follows easily from the Peter-Weyl decomposition of $\Gamma(P,\pi^*\Ecal)$ under the free action of $K$. It is a particular case of the free action property of transversally elliptic symbols.
\begin{theo}
In $\hat{R}(G\times K)$ we have
$$
\indice_{G\times K}^P\left(\sigma_{hor}(P,\Ecal)\right)=\sum_{\mu \in \hat K}
\Qcal_G(M,\Ecal\otimes  \Vcal_\mu^*) \otimes  V_\mu.
$$

In particular $[\indice_{G\times K}^P\left(\sigma_{hor}(P,\Ecal)\right)]^K=\Qcal_G(M,\Ecal)$.
\end{theo}

\medskip

Consider now the case of a $G\times K$ Riemannian manifold $P$, but assume
now that $K$ acts only infinitesimally freely: for any $p\in P$, the stabilizer $K_p$  is a finite subgroup.
So $M=P/K$ is an orbifold, but maybe not a manifold. Consider  $\Hcal$, the subbundle of $\T P$ defined by
the orthogonal decomposition $\T P=\Hcal \oplus P\times\kgot$ and call $\Hcal$ the horizontal tangent bundle.

\begin{defi}\label{defi:Qred}
$\bullet$ If $\Ecal$ is a graded Clifford bundle for the horizontal tangent bundle $\Hcal$, equivariant for the
$G\times K$-action, we  define $\sigma_{hor}(P,\Ecal)(p,\nu)=\clif_{\Ecal_p}(\tilde{\nu}_\Hcal)$ where
$\tilde{\nu}_\Hcal$ is the horizontal component of $\tilde{\nu}\in \T_p P$.

$\bullet$ We define
$\Qcal_G(M,\Ecal)$ by the formula
$$\Qcal_G(M,\Ecal):=\left[\indice_{G\times K}^P(\sigma_{hor}(P,\Ecal))\right]^K.$$
\end{defi}

It can be shown, see \cite{Vergne96}, that this is indeed the equivariant index of a Dirac operator on the $G$-orbifold $M$ in the sense of Kawasaki \cite{Kawasaki81}.

\section{Deformation \`a la Witten of Dirac operators}

We start to discuss the real purpose of this paper. We will study the deformation \`a la Witten of elliptic symbols
$\sigma(M,\Ecal)$ by vector fields associated  to $K$-equivariant maps $\Phi: M\to\kgot^*$.

We recall that we  chose  an invariant scalar product in $\kgot$ providing an identification
$\kgot\simeq\kgot^*$.

\subsection{Kirwan vector field}\label{sec:Kirwan-vector-field}

Let $M$ be a $K$-manifold and let $\Phi: M\to \kgot^*$ be an equivariant map.

\begin{defi}\label{defi:kir}
$\bullet$ The {\em Kirwan vector field} associated to $\Phi$ is defined by
\begin{equation}\label{eq-kappa}
    \kappa_{\Phi}(m)= -\Phi(m)\cdot m, \quad m\in M.
\end{equation}

$\bullet$ We denote by $Z_\Phi$ the set of zeroes of $\kappa_{\Phi}$. Thus $Z_\Phi$ is a $K$-invariant closed subset of $M$.

\end{defi}

The $Z_\Phi$ is not necessarily smooth.
If $0\in \Phi(M)$, then clearly $Z_\Phi$ contains $\Phi^{-1}(0)$.

In order to describe more precisely $Z_\Phi$, we consider the set $\{ (\kgot_i), i\in I\}$  of conjugacy classes of
subalgebras of $\kgot$ representing the infinitesimal stabilizers of the $K$-action on $M$ : if $M$ is compact, the set $I$ is finite.
We denoted by $M_{\kgot_i}$ the submanifold of points $m\in M$ such that $\kgot_i=\kgot_m$. We see that
$Z_\Phi\cap M_{\kgot_i}=M_{\kgot_i}\cap \Phi^{-1}(\kgot^*_i)$, hence
\begin{equation}\label{eq:Z-Phi}
Z_\Phi=\bigsqcup_{i\in I} K\Big(M_{\kgot_i}\cap \Phi^{-1}(\kgot^*_i)\Big),
\end{equation}
where $\kgot_i^*$ is viewed as a subspace of $\kgot^*$ thanks to the scalar product.

\medskip

\begin{defi}
An equivariant map $\Phi: M\to \kgot^*$ is called a {\em moment map} if there exists an invariant (real) $2$-form $\Omega$ on $M$
 satisfying the relations
\begin{equation}\label{eq:hamiltonian-action}
    \iota(X_M)\Omega+d\langle\Phi,X\rangle=0 \quad \mathrm{and} \quad d\Omega=0
\end{equation}
for $X\in\kgot$. The data $(\Omega,\Phi)$ on $M$ is called a weak Hamiltonian structure.
\end{defi}
Classically, a moment map in the sense of Hamiltonian geometry requires that $\Omega$ is non degenerate.
In this text, we will see that some localization results can be proven without this strong hypothesis.

A particularly important example of moment map is provided by the following construction.
\begin{defi}\label{defi:PhiL}
Let $L$ be a $K$-equivariant Hermitian line bundle with Hermitian connection $\nabla$ with curvature $R_L$.
Let $\Omega_L=\frac{1}{i}R_L$. Then $\Omega_L$ is a closed real two form.
This data determines a  moment map
$$
\Phi_{L} : M\to \kgot^*
$$
by the relation $\Lcal(X)-\nabla_{X_M}= i\langle \Phi_{L},X\rangle$, for  all   $X\in\kgot$.
Here $\Lcal(X)$
is the infinitesimal action of $X\in \kgot$ on the sections of $L$.
 The relation $\iota(X_M)\Omega_L+d\langle\Phi_L,X\rangle=0$ follows.
\end{defi}

\begin{rem}\label{rem:proj}
The subset  $\Phi_L^{-1}(0)$ of the zeroes of the associated Kirwan vector field
has some particular importance when $L$ is an ample bundle on a projective manifold provided with an action of $K$.
In this case, if $\Phi_L$ is the moment map associated to the projective embedding,
 $\Phi_L^{-1}(0)/K$ is isomorphic to the geometric invariant theory (GIT) quotient $M/\!\!/K_\C$ of $M$ by
 $K_\C$, and is provided with  structure of projective variety with corresponding line bundle $L/\!\!/K_\C$.
When $M/\!\!/K_\C$ is smooth, Guillemin and Sternberg \cite{Guillemin-Sternberg82} have proved the
first $[Q,R]=0$ identity :
$$
[H^0(M,L)]^K\simeq H^0(M/\!\!/K_\C;L/\!\!/K_\C).
$$
\end{rem}

\begin{rem}\label{rem;Zphi}
If $\Phi$ is a moment map, by  Equation
(\ref{eq:hamiltonian-action}), we have
$$d\|\Phi\|^2=2\iota(\kappa_\Phi)\Omega.$$
Thus the set $Z_\Phi$ is contained in the critical set of the function $\|\Phi\|^2$ : if $\beta\in \Phi(Z_\Phi)$, then $\|\beta\|^2$ is a critical value of the square of the moment map $\Phi$.
\end{rem}

\bigskip

One of the properties that a moment map shares with  the moment map of a Hamiltonian action is the following.

\begin{prop}\label{prop: O regular}

$\bullet$ For any $X\in\kgot$, the function $\langle\Phi,X\rangle$ is constant on each connected component of the manifold $M^X$.

$\bullet$ If $0$ is a regular value of  a moment map $\Phi$, then the action of $K$ is infinitesimally free on the submanifold $\Phi^{-1}(0)$. The reduced space
$M_{red}= \Phi^{-1}(0)/K$ is an orbifold.

$\bullet$ Conversely, if $\Phi$ is the moment map of a Hamiltonian action, and
the action of $K$ is infinitesimally free on the submanifold $\Phi^{-1}(0)$, then $0$ is a regular value of $\Phi$.
\end{prop}

\begin{proof}
Equation (\ref{eq:hamiltonian-action}) tells us that for a point $m\in M$ where $X_M(m)=0$, the differential $\T_m\Phi$ sends $\T_m M$ to
$(\R X)^{\perp}\subset\kgot^*$. Our three points follow from this simple fact.
\end{proof}\bigskip

We note that any invariant $1$-form $\alpha$ defines a moment map $\varphi_\alpha :M\to \kgot^*$ by the relation :
$ \langle\varphi_\alpha,X\rangle=\iota(X_M)\alpha$, $X\in\kgot$: the corresponding $2$-form is $d\alpha$.

We will use the following equivalence relation between moment maps on $M$ :
\begin{equation}\label{eq:equivalence-Phi}
\Phi\sim \Phi'\ \Longleftrightarrow \exists\ \mathrm{an \ invariant}\ 1\mathrm{-form}\ \alpha\ \mathrm{such\ that}\
\Phi-\Phi'=\varphi_\alpha.
\end{equation}

\medskip

Let us describe $\Phi(Z_\Phi)$ when $\Phi$ is a moment map.
As before, let  $\{ \kgot_i, i\in I\}$  be a set of representatives of the infinitesimal stabilizers of the $K$-action on $M$,
and let $M_{\kgot_i}$ be the submanifold of points $m\in M$ such that $\kgot_i=\kgot_m$.
If $\Xcal$ is a connected component  of $M_{\kgot_i}$, Relations (\ref{eq:hamiltonian-action}) imply that the image $\Phi(\Xcal)$ is
contained in an affine space $A(\Xcal,\Phi)\subset \kgot^*$ with direction $\kgot_i^{\perp}$. The orthogonal projection $\beta(\Xcal,\Phi)$ of
$0$ on the affine space of $A(\Xcal,\Phi)$ belongs to $\kgot^*_i$, viewed as a subspace of $\kgot^*$ thanks to the scalar product.

\begin{lem}\label{lem:PhiZPhi}
The set $\Phi(Z_\Phi)$ is contained in
$$
C_\Phi:=\bigcup_{i\in I} \bigcup_{\Xcal\subset M_{\kgot_i}} K\cdot \beta(\Xcal,\Phi).
$$
More precisely, the orbit $K\cdot \beta(\Xcal,\Phi)$ is contained in $\Phi(Z_\Phi)$ if and only if
$\beta(\Xcal,\Phi)$ belongs to $\Phi(\Xcal)$.
\end{lem}
\begin{proof}The proof is left to the reader.
\end{proof}

Concerning the set $C_\Phi$, we have the useful remark.

\begin{prop} Let $\Phi,\Phi'$ be two moment maps on the $K$-manifold $M$.
\begin{itemize}
\item If $\Phi\sim \Phi'$ then $C_\Phi=C_{\Phi'}$.
\item If $M$ is compact,  $C_\Phi$ is  a finite union of coadjoint orbits.
\item If $Z_\Phi$ is compact, then $\Phi(Z_\Phi)$ is  a finite union of coadjoint orbits.
\end{itemize}
\end{prop}
\begin{proof}
The first point is due to the fact that $A(\Xcal,\Phi)=A(\Xcal,\Phi+\varphi_\alpha)$ for any invariant
$1$-form $\alpha$ and any connected component  $\Xcal$ of $M_{\kgot_i}$. If $M$ is compact, the
set $I$ is finite and each submanifold $M_{\kgot_i}$
has a finite number of connected components, so there is a finite number of terms $\beta(\Xcal,\Phi)$.
If $Z_\Phi$ is compact, we consider an invariant neighborhood $\Ucal$ of $Z_\Phi$ such that $\overline{\Ucal}$
is compact. The subset $I_\Ucal\subset I$ formed by the element $i\in I$ such that
$M_{\kgot_i}\cap \Ucal\neq\emptyset$ is finite, and we see that $\Phi(Z_\Phi)$ is contained in
$$
C_\Phi^\Ucal:=\bigcup_{i\in I_\Ucal} \bigcup_{\Xcal\subset M_{\kgot_i}\cap\Ucal} K\cdot \beta(\Xcal,\Phi).
$$
Since $C_\Phi^\Ucal$ is  a finite union of coadjoint orbits, the last assertion is proved.
\end{proof}

\medskip

For any moment map $\Phi$, if $Z_\Phi$ is compact, it admits a finite decomposition
$Z_\Phi=\coprod_{\Ocal} Z_\Ocal$ where $Z_\Ocal= Z_\Phi\cap \Phi^{-1}(\Ocal)$.

\begin{defi}\label{def:Bphi}
We denote by $\Bcal(\Phi)$ the set
of coadjoint orbits contained in $\Phi(Z_\Phi)$.
 \end{defi}

 In practice, we
parameterize the set $\Bcal(\Phi)$ by a finite set $\Bcal$  contained in $\kgot^*$:
$\Phi(Z_\Phi)=\coprod_{\beta\in\Bcal} K\beta$ and then
\begin{equation}\label{eq=Z-Phi-beta}
Z_\Phi=\coprod_{\beta\in\Bcal} Z_\beta
\end{equation}
where $Z_\beta= K(M^\beta\cap\Phi^{-1}(\beta))$.

\subsection{Deforming a  symbol}

Let $M$ be a $K$-manifold. We do not necessarily assume  $M$ to be compact, and the compact Lie group $K$  nor $M$ are
 necessarily connected. Let $\Phi: M\to \kgot^*$ be
an equivariant map, and let $\Ecal$ be a graded Clifford bundle on $M$.

\begin{defi}\label{def:pushed-sigma}
The symbol  $\sigma(M,\Ecal,\Phi)$ pushed by the vector field $\kappa_{\Phi}$ is the
symbol on $M$ defined by
$$
\sigma(M,\Ecal,\Phi)(m,\nu)=\sigma(M,\Ecal)(m,\tilde{\nu}-\kappa_\Phi(m))
=\clif_{\Ecal_m}(\tilde{\nu}-\kappa_\Phi(m))$$
for any $(m,\nu)\in\T^* M$.
\end{defi}

Note that $\sigma(M,\Ecal,\Phi)(m,\nu)$ is invertible except if
$\tilde{\nu}=\kappa_\Phi(m)$. If furthermore $(m,\nu)$ belongs to the subset $\T_K^* M$
of cotangent vectors orthogonal to the $K$-orbits, then $\nu=0$ and
$\kappa_\Phi(m)=0$.  Indeed, as
$\kappa_\Phi(m)$ is tangent to $K\cdot m$,
the two equations
$\tilde{\nu}=\kappa_\Phi(m)$ and $\langle \nu,\kappa_\Phi(m)\rangle=0$ imply
$\tilde{\nu}=\kappa_\Phi(m)=0$.
So we note that $(m,\nu)\in \Char(\sigma(M,\Ecal,\Phi))\cap \T_K^* M$ if and only if $\nu=0$ and $\kappa_\Phi(m)=0$.
Hence $\sigma(M,\Ecal,\Phi)$ is transversally elliptic whenever $Z_\Phi$ is compact.

\begin{defi}
If $Z_\Phi$ is compact, we define $\Qcal_K(M,\Ecal,\Phi)\ \in\ \hat{R}(K)$ as the equivariant index of the transversally elliptic symbol
$\sigma(M,\Ecal,\Phi)\in \Ko_K(\T_K^* M)$.
\end{defi}

When $M$ is compact, it is clear that the classes of the symbols $\sigma(M,\Ecal,\Phi)$ and
$\sigma(M,\Ecal)$ are equal in $\K_{K}^0(\T_{K}^*M)$, hence the equivariant indices $\Qcal_K(M,\Ecal)$ and $\Qcal_K(M,\Ecal,\Phi)$
are equal.

For any $K$-invariant, {\em relatively compact}, open subset $\Ucal\subset M$ such that $\Ucal\cap Z_\Phi$ is closed\footnote{$\Ucal\cap Z_\Phi$ is then compact.} in $M$, we see that the restriction
$\sigma(M,\Ecal,\Phi)\vert_\Ucal$ is a transversally elliptic symbol on $\Ucal$, and so its equivariant index is a well defined element in
$\hat{R}(K)$, independent of the choice of such a $\Ucal$.

\begin{defi}\label{def:indice-localise}
$\bullet$ A closed $K$ invariant subset $Z\subset Z_\Phi$ is called a {\em\bf component} if it is a union of connected components of $Z_\Phi$.

$\bullet$ For a compact component $Z$ of $Z_\Phi$, we denote by
$$
\sigma(M,\Ecal,Z,\Phi)\in \Ko_K(\T_K^* M)
$$
the image of $\sigma(M,\Ecal,\Phi)\vert_\Ucal\in \Ko_K(\T_K^* \Ucal)$ by the pushforward morphism \break
$\Ko_K(\T_K^* \Ucal)\to \Ko_K(\T_K^* M)$ attached to a relatively compact invariant neighborhood $\Ucal$ of $Z$ satisfying
$\Ucal\cap Z_\Phi=Z$.

$\bullet$ For a compact component $Z$ of $Z_\Phi$, we denote by
$$
\Qcal_K(M,\Ecal,Z,\Phi)\ \in\ \hat{R}(K)
$$
the equivariant index of $\sigma(M,\Ecal,Z,\Phi)$. By definition, if $Z=\emptyset$, then \break $\Qcal_K(M,\Ecal,\emptyset,\Phi)=0$.
\end{defi}

In particular, if $\Phi: M\to \kgot^*$ is a proper moment map, we can define
$$
\Qcal_K(M,\Ecal,\Phi^{-1}(0),\Phi)\in \hat{R}(K).
$$
By taking its invariant part $[\Qcal_K(M,\Ecal,\Phi^{-1}(0),\Phi)]^K\in\Z$, we will  define in a natural way invariants on the reduced ``manifold''
$M_{red}=\Phi^{-1}(0)/K$.

\medskip

The simplest example of a pushed symbol $\sigma(M,\Ecal,\Phi)$ is when the map $\Phi$ is constant, equal to a $K$-invariant element $\beta\in\kgot^*\simeq\kgot$.
The vector field $m\mapsto \beta\cdot m$ is then $K$-equivariant, and we define
$$
\sigma(M,\Ecal,\beta)(m,\nu)=\sigma(M,\Ecal)(m,\tilde{\nu}+\beta\cdot m).
$$

The first example of pushed symbol $\sigma(M,\Ecal,\beta)$ is the Atiyah symbol that we will study in the next section.

\medskip

We end this section by considering the case where $\Phi$ is a moment map such that $Z_\Phi$ is compact. For any
$\Ocal\subset C_\Phi$, the set $Z_\Ocal=Z_\Phi\cap \Phi^{-1}(\Ocal)$ is a compact component of $Z_\Phi$,
hence we may consider the localized index $\Qcal_K(M,\Ecal,Z_\Ocal,\Phi)$. This generalized character is non-zero only if
$Z_\Ocal\neq\emptyset$, and this happens only for a finite number of orbits (those included in $\Phi(Z_\Phi)$).

Thanks to the excision property, we
get the first form of the localization theorem :
\begin{equation}
\label{eq:localization-Phi}
\Qcal_K(M,\Ecal,\Phi)=\sum_{\Ocal\subset C_\Phi} \Qcal_K(M,\Ecal, Z_\Ocal,\Phi).
\end{equation}

\medskip

The aim of Section \ref{sec:ASS-deformed} is  the  description by induction of the representations
$\Qcal_K(M,\Ecal, Z_\Ocal,\Phi)$.

\medskip

Let us show that, when $M$ is compact, the character $\Qcal_K(M,\Ecal, Z_\Ocal,\Phi)$ depends only of
the class of $\Phi$ relative to the equivalence relation (\ref{eq:equivalence-Phi}).

\begin{prop}
Let $\Phi\sim \Phi'$ be two moment maps on a  compact $K$-manifold. Let $Z_\Ocal\subset Z_\Phi$ and $Z'_\Ocal=Z_{\Phi'}\cap (\Phi')^{-1}(\Ocal)$
be the components attached to $\Ocal\in C_{\Phi}=C_{\Phi'}$. Then we have
$$
\Qcal_K(M,\Ecal, Z_\Ocal,\Phi)=\Qcal_K(M,\Ecal, Z'_\Ocal,\Phi').
$$
\end{prop}

\begin{proof}
Let $\alpha$ be the invariant $1$-form such that $\Phi'-\Phi=\varphi_\alpha$. We consider the family of moment maps
$\Phi^t=\Phi+t\varphi_\alpha,\ t\in \R$. We denote simply by $\sigma^t$ the symbol $\sigma(M,\Ecal, \Phi^t)$. For any $t$, consider the component $Z^t_\Ocal=Z_{\Phi^t}\cap (\Phi^t)^{-1}(\Ocal)$ and the generalized character
$$
\Qcal^t:=\Qcal_K(M,\Ecal, Z^t_\Ocal,\Phi^t)
$$
which is the equivariant index of $\sigma^t\vert_{U_t}$ where $U_t$ is a neighborhood of $(\Phi^t)^{-1}(\Ocal)$
such that $U_t\cap Z_{\Phi^t}= Z^t_\Ocal$.

Let us prove that the multiplicity $\Qcal^t$ is independent of $t$. It is sufficient to prove that $t\to \Qcal^t$ is
locally constant : let us show that it is constant in a neighborhood of $0$.

Note that the set $C_{\Phi^t}$ does not depend of $t$: let us simply denote it by $C$. We remark that for any invariant subset $U\subset M$, we have
$$
U\cap Z_{\Phi^t}\subset \bigcup_{\Pcal\subset \Phi^t(U)\cap C} Z^t_\Pcal
$$

By an obvious continuity argument, we see that if $U_0$ is a neighborhood of
$(\Phi^0)^{-1}(\Ocal)$ such that $\Phi^0(U_0)\cap C=\Ocal$, then we have
$$
\Phi^t(U_0)\cap C=\Ocal\quad\mathrm{and}\quad U_0\cap Z_{\Phi^t}=Z^t_\Ocal
$$
for small $t$ (says $|t|\leq \epsilon$). As the family $\sigma^t\vert_{U_0}$, $t\in [0,\epsilon]$ defines an homotopy of transversally elliptic symbols,
they have the same equivariant index.
Finally, for small $t$, we have
\begin{eqnarray*}
\Qcal^0=\indice_{K}(\sigma^0\vert_{U_0}) &=&\indice_{K}(\sigma^t\vert_{U_0})\\
&=& \Qcal^t.
\end{eqnarray*}
\end{proof}

\subsection{Atiyah symbol}\label{sec:atiyah-symbol}

Let $N$ be a {\bf  Hermitian  vector space}. We denote by $\mathrm{U}$ the unitary group of transformations of $N$, and by $U(1)\subset \mathrm{U}$ the subgroup  formed the homotheties $x\mapsto e^{i\theta} x$.   Consider $S=\bigwedge \overline{N}$,
the irreducible  Clifford module for $N$ with Clifford action $\clif(n)$.
If we identify $\T^* N=N\oplus N^*$ with $N\oplus N$, the $\mathrm{U}$-equivariant symbol $\sigma(N,S)$  is defined by
$$
\sigma(N,S)(x,\nu)=\clif(\tilde{\nu}): \bigwedge^{even} \overline{N}\longrightarrow \bigwedge^{odd} \overline{N}, \qquad (x,\nu)\in\T^*N.
$$

The matrix $\beta=i \mathrm{Id}$ is an invariant element of $\mathrm{Lie}(\mathrm{U})$. We have $\beta\cdot x=Jx$ where
$J$ is the multiplication by $i$ on $N$. We can then define the pushed symbol $\sigma(N,S,\beta)$.
\begin{defi}
The Atiyah symbol $\at(N) \in \Gamma(\T^* N, \hom(\Scal^+,\Scal^-))$  is defined by
$$
\at(N)(x,\nu)=\clif(\tilde{\nu}+Jx): \bigwedge^{even} \overline{N}\longrightarrow \bigwedge^{odd} \overline{N}, \qquad (x,\nu)\in\T^*N.
$$
\end{defi}

The symbol $\at(N)$ is not elliptic since
$\supp(\at(N))=\{(x,\nu), \tilde{\nu}+Jx=0\}$ is not
compact.
But remark that $x\to Jx$ is the tangent vector  produced by  the action of $U(1)$ on $N$.
Thus $\at$ is a $U(1)$-transversally elliptic symbol.
Indeed the two equations $\tilde{\nu}+Jx=0$ and $\langle \nu,Jx\rangle =0$ imply $\tilde{\nu}=x=0$.
So a fortiori the symbol $\at(N)$ is $\mathrm{U}$-transversally elliptic.
Let us consider $\Sym(N)=\oplus_{k=0}^{\infty} {\rm Sym}^k(N)$
where ${\rm Sym}^k(N)$ is the subspace of $\otimes^k N$ formed of symmetric tensors.

The following proposition is proved in \cite{Atiyah74} (see also \cite{B-V.inventiones.96.1}).

\begin{prop} \label{prop-indice-atiyah}
We have $\indice_{\mathrm{U}}^N(\at(N))=\Sym (N)$ in $\hat{R}(\mathrm{U})$.
\end{prop}

\medskip

More generally, consider the case of a real vector space $N$ equipped with a linear action of a compact Lie group $K$.
Let $\beta\in\kgot$ be a $K$-invariant element such that $\Lcal(\beta):N\to N$ is invertible. We denote $\Tbb_\beta\subset K$ the closed torus contained in the center equal to the closure of $\exp(\R\beta)$.

We denote by $N_{J_\beta}$ the Hermitian vector space $N$ equipped with the invariant complex structure $J_\beta$.
We consider the $K$-equivariant $\Tbb_\beta$-transversally elliptic symbol $\at_\beta(N)$ defined by
$$
\at_\beta(N)(v,\xi)=\clif(\tilde{\xi}+\beta\cdot v) : \bigwedge^{even} \overline{N_{J_\beta}}\longrightarrow\bigwedge^{odd} \overline{N_{J_\beta}}.
$$
Its equivariant index satisfies
$$
\indice_{K}^N(\at_\beta(N))=\Sym (N_{J_\beta}).
$$

The meaning of this equation is as follows.
We write  $\Sym(N_{J_\beta})$ as the sum of the   finite dimensional spaces ${\rm Sym}^k(N_{J_\beta})$.
We consider ${\rm Sym}^k(N_{J_\beta})$ as a representation space for $K$, write
${\rm Sym}^k(N_{J_\beta})=\sum_{\mu \in \hat{K}}\mm(\mu,k)V_\mu$.
 For each irreducible representation  $V_\mu\in \hat{K}$,
  the infinitesimal action of $\beta$ on $V_\mu$ is by a scalar $i z_\mu {\rm Id}_{V_\mu}$, as $\beta$ is fixed by $K$.
On the other hand, the action of $\beta$ in ${\rm Sym}^k(N_{J_\beta})$ is diagonalizable with eigenvalues $\sum_{J} i a_j$ with $J$ a subset of cardinal $k$ of   $[1,2,\ldots, \dim_\C N]$ : here $a_j>0$ are the eigenvalue of $\frac{\Lcal(\beta)}{i}$ on $N_{J_\beta}$.
As $\sum_{J}  a_j$ is larger than $k \min(a_j)$,
  we see that   $k$ has to be small enough in order that $\mm(\mu,k)\neq 0$.
Thus the sum $\sum_k \sum_{\mu \in \hat K}m(\mu,k)V_\mu$ is a well defined element of $\hat R(K)$.

\bigskip

We now consider a $K$-equivariant Euclidean  vector bundle $\Ncal\to \Xcal$ on a connected $K$-manifold $\Xcal$.
Let us assume that $\beta\in \kgot$  is an element fixed by $K$ and acting fiberwise  on $\Ncal$ by  an invertible
transformation : in other words $\Ncal^\beta=\Xcal$. The eigenvalues of $\beta^2$ on the fibers are strictly negative and
constant over each connected component of $\Xcal$.
Using Definition \ref{defi:Jbeta} of $J_\beta$, we define  on $\Ncal$ a Hermitian structure invariant by $K$.
We denote by $\Ncal_\JJbeta$ the vector bundle $\Ncal$ considered as a complex vector bundle with complex structure $J_\beta$.
By definition of $J_\beta$, the list of eigenvalues of $\beta$ on $\Ncal_m$ for the complex structure $J_\beta$ are
$[ia_1,ia_2,\ldots, ia_R]$ where the $a_j$ are strictly positive.

We denote by $\bigwedge\Ncal_\JJbeta$ the  complex  exterior bundle of the complex vector bundle $\Ncal_\JJbeta$. We denote by
$\Sym(\Ncal_\JJbeta)$ the $\Z$-graded symmetric vector bundle of $\Ncal_\JJbeta$.

Let $N$ be the fiber $\Ncal\vert_{x_o}$ at some point $x_o\in \Xcal$, equipped with the linear action of $\beta$, the complex structure $J_\beta$ and its Euclidean structure. Let $\mathrm{U}$ be the subgroup of $\mathrm{U}(N)$ of elements that commutes with $\Lcal(\beta)$. Let $P_U\to \Xcal$ be the $\mathrm{U}$-unitary framed bundle of $\Ncal$ : for $x\in \Xcal$, we take
$$
P_U\vert_x:=\left\{ f :N\to \Ncal\vert_x, \ \mathrm{unitary\ and}\ \Lcal(\beta)-\mathrm{equivariant}\right\}.
$$
Hence the vector bundle $\Ncal$ is equal to $P\times_{\mathrm{U}} N$. The Atiyah symbol on $N$ defines a class
$$
\at_\beta(N)\in \Ko_{\mathrm{U}\times \Tbb_\beta}(\T^*_{\Tbb_{\beta}} N).
$$

Let us use the notion of fiber product  introduced in Section \ref{sec:fiber-product}: here we work with the group $S:=\Tbb_\beta$.
We can define the morphism $\at_\beta:\Ko_K(\T^*_{K} \Xcal)\longrightarrow \Ko_K(\T^*_{K} \Ncal)$ by the relation
\begin{equation}\label{eq:morphismAtbeta}
\at_\beta(\sigma):=\sigma\lozenge \at_\beta(N)
\end{equation}
where $\sigma$ is a $K$-transversally elliptic symbol on $\Xcal$.

Lemma \ref{lem:indice-lozenge} gives
\begin{eqnarray}\label{eq:indice-Atbeta}
\indice_K^\Ncal\left(\at_\beta(\sigma)\right)&=&\indice_K^\Xcal\left( \sigma \otimes\Sym (\Ncal_\JJbeta)\right)\\
&=& \sum_{k=0}^{\infty} \indice_K^\Xcal\left(\sigma \otimes {\rm Sym}^k(\Ncal_\JJbeta)\right).
\end{eqnarray}

A  symbol representing $\at_\beta(\sigma)$ is as follows.
 We write an element $n\in  \Ncal$ as $(x,v)$ with $x\in \Xcal$ and $v\in \Ncal\vert_x$.
We choose a $K$-invariant connection on the bundle $\Ncal\to \Xcal$, which allows us to parameterize
 $\T^* \Ncal=\{(n,(\xi,\nu))\}$, with
$\xi\in \T^*_x\Xcal$, $\nu\in \Ncal^*\vert_x$ (here $\xi$ is lifted to an element of $\T_n^* \Ncal$ using the connection).
Let
$\sigma(x,\xi) \in \Gamma(\T^* \Xcal,\hom(\Ecal^+,\Ecal^-))$
be a symbol on $\Xcal$.
Then
\begin{equation}
\at_\beta(\sigma)((x,v),(\xi,\nu))=\sigma(x,\xi) \boxtimes \clif_{x}(\tilde{\nu}+\beta v).
\end{equation}
Here $\clif_{x}$ is the Clifford action on $\bigwedge\overline \Ncal_\JJbeta\vert_x$, so
$\clif_{x}(\tilde{\nu}+\beta v)^2=-\|\tilde{\nu}+\beta v\|^2\mathrm{Id}$.
We see that
 $\Char(\at_\beta(\sigma))\cap \T^*_K\Ncal=\Char(\sigma)\cap \T^*_K \Xcal$.

\section{Abelian Localization formula}

\subsection{Atiyah-Segal-Singer localization formula}\label{sec:loc-Atiyah-Segal}

Let $M$ be a compact $K$-manifold. Let $\beta\in \kgot$ be a central element, and let $M^\beta$ be the submanifold of points fixed by
the infinitesimal action of $\beta$. Let $\Ncal$ be the normal bundle of $M^\beta$ in $M$. Then $\beta$ acts by a fiberwise invertible
transformation on $\Ncal$. Denote by  $\Ncal_\JJbeta$ the (complex) vector bundle $\Ncal$ polarized by $\beta$.

Let $\Ucal$ be a tubular neighborhood of $M^\beta$ in $M$ that is diffeomorphic with the normal bundle $\Ncal$. We denote by
$\varphi: \Ucal\to \Ncal$ the $K$-diffeomorphism and by $i:\Ucal\croc M$ the inclusion.  Let $\varphi^* : \Ko_K(\T^*_{K} \Ncal)\to \Ko_K(\T^*_{K} \Ucal)$ and $i_*:\Ko_K(\T^*_{K} \Ucal)\to \Ko_K(\T^*_{K} M)$ the corresponding maps in $\K$-theory.

We denote by
$$
\At_\beta:\Ko_K(\T^*_{K} M^\beta)\longrightarrow \Ko_K(\T^*_{K} M)
$$
the map $i_*\circ \varphi^*\circ\at_\beta$. For the equivariant index, we have the equality
$$
\indice_K^M\left(\At_\beta(\sigma)\right)=\indice_K^{M^\beta}\left(\sigma\otimes \Sym (\Ncal_\JJbeta)\right)
$$
for any $\sigma\in \Ko_K(\T^*_{K} M^\beta)$.

\medskip

Let $\Ecal\to M$ be a (graded) Clifford module on $M$. The bundle $\bigwedge \overline{\Ncal_\JJbeta}$ is a $\spinc$ bundle  for the Euclidean bundle $\Ncal$.

\begin{defi}\label{defi:q-beta}
There exists a unique graded Clifford bundle $\mathbb{d}_\beta(\Ecal)$ over the tangent bundle $\T M^\beta$, such that, when we restrict $\Ecal$ to $M^\beta$, we get the isomorphism
\begin{equation}\label{eq:Ebeta}
\Ecal\vert_{M^\beta}\simeq\mathbb{d}_\beta(\Ecal)\otimes  \bigwedge \overline{\Ncal_\JJbeta}
\end{equation}
of graded Clifford bundles relative to the decomposition $\T M\vert_{M^\beta}:= \T M^\beta \oplus \Ncal$.
\end{defi}

We have two elliptic symbols : $\sigma(M,\Ecal)\in \Ko_K(\T^*M)$ and $\sigma(M^\beta,\mathbb{d}_\beta(\Ecal))\in \Ko_K(\T^*M^\beta)$.

\begin{prop}[Atiyah-Segal-Singer localization formula]\label{prop:atiyah-segal}
The following relation holds in $\Ko_K(\T^*_{K} M)$ :
$$
\At_\beta\left(\sigma(M^\beta,\mathbb{d}_\beta(\Ecal))\right)=\sigma(M,\Ecal).
$$
In particular $\Qcal_K(M,\Ecal)=\Qcal_K\left(M^\beta,\mathbb{d}_\beta(\Ecal)\otimes \Sym(\Ncal_\JJbeta)\right)$.
\end{prop}

\begin{rem}
In Atiyah-Segal, this equation is proved in the tensor product of  $\Ko_K(\T^* M)$  with the field of fractions of $R(K)$.
In  \cite{Atiyah74}, this more precise relation is proved in
$\Ko_K(\T^*_{K} M)$.
\end{rem}

\medskip

\begin{proof} We deform the symbol $\sigma(M,\Ecal)$ by the vector field $\beta_M$. So we apply the deformation introduced in Definition \ref{def:pushed-sigma} with the constant map $\Phi=\beta$. The pushed-symbol $\sigma(M,\Ecal,\beta)$ is
$$
\sigma(M,\Ecal,\beta)(m,\nu):= \clif_{\Ecal_m}(\tilde{\nu} +\beta\cdot m) : \Ecal_m^+\to\Ecal_m^-.
$$
The excision property tells us that $\sigma(M,\Ecal)= i_*(\sigma(M,\Ecal,\beta)\vert_\Ucal)$ in $\Ko_K(\T^*_K M)$. We now analyze the symbol
$\tilde{\sigma}:=(\varphi^{-1})^*(\sigma(M,\Ecal,\beta)\vert_\Ucal)$ on $\Ncal$. Let $\pi :\Ncal\to M^\beta$ be the projection. Over $\Ncal$, the Clifford bundle
$\tilde{\Ecal}:=(\varphi^{-1})^*(\Ecal\vert_\Ucal)$ is isomorphic to $\pi^*(\mathbb{d}_\beta(\Ecal))\otimes  \pi^*(\bigwedge \overline{\Ncal_\JJbeta})$.
With notations of Subsection \ref{sub:choice},   we may assume that
 for   $(n,\nu)\in\T^*\Ncal$,  and $n=(x,v)$, the map $\tilde{\sigma}(n,(\xi, \nu))$ is the exterior product of the map
$$
\clif^1_{x} (\tilde{\xi}):\mathbb{d}_\beta(\Ecal)\vert_x^+\longmapsto \mathbb{d}_\beta(\Ecal)\vert_x^-
$$
with the map
$$
\clif^2_x(\tilde{\nu}+\beta\cdot v) :\bigwedge^{even} \overline{\Ncal_\JJbeta}\vert_x\longmapsto\bigwedge^{odd} \overline{\Ncal_\JJbeta}\vert_x.
$$
Here $\xi$ is viewed as a horizontal vector in $\T_n^* \Ncal$ and $\tilde{\nu},\beta\cdot v$ are vertical vectors in  $\T_n\Ncal$.
 By definition, this means that $\tilde{\sigma}$ is equal to $\sigma(M^\beta,\mathbb{d}_\beta(\Ecal))\lozenge \at_\beta(N)$.

\end{proof}

\subsection{Deformed Atiyah-Segal-Singer localization formula}\label{sec:ASS-deformed}
We now prove a slightly more general version of the preceding Atiyah-Segal-Singer localization formula.

 Let $\Ncal\to \Xcal$ be a $K$-equivariant vector bundle over $\Xcal$, and let $\beta$ be a central element in
 $\kgot$ such that $\Xcal=\Ncal^{\beta}$. Thus $\beta$ acts invertibly on the vector bundle $\Ncal$.
Let $\Phi: \Ncal \to \kgot^*$ be a $K$-invariant map and $Z_\Phi\subset \Ncal$ be the set of zeroes of the associated
Kirwan vector field $\kappa_\Phi$. We do not assume the basis $\Xcal$ compact, but we assume that
$\Phi_\Xcal:=\Phi\vert_\Xcal: \Xcal\to \kgot^*$ is a proper map.

The intersection of the set $Z_\Phi$ with $\Xcal$ is equal to $Z_{\Phi_\Xcal}$.

\begin{lem}
Assume that $Z_\Phi\cap \Xcal=\Phi^{-1}_\Xcal(\beta)$. Then $\Phi^{-1}_\Xcal(\beta)$ is a compact component of $Z_\Phi$: there exists a
$K$-invariant neighborhood $\Ucal$ of
$\Phi^{-1}_\Xcal(\beta)$ in $\Ncal$ such that $Z_\Phi\cap \Ucal=\Phi^{-1}_\Xcal(\beta)$.
\end{lem}

\begin{proof}
Consider a $K$-invariant decomposition of the tangent bundle $\T\Ncal=H\oplus V$ in vertical and horizontal tangent bundles. Here $V=\pi^*\Ncal$.
As $\Ncal$ is a $K$-equivariant vector bundle,
for $X\in \kgot$, and $n=(x,v)\in \Ncal$, the vertical projection
$[X_\Ncal]^V(x,v)\in \Ncal\vert_x$  of the vector field $X_\Ncal$ depends linearly of $v$.
In other words, there exists $\mu(X)\in \Gamma(\Xcal,\End(\Ncal))$ such that
 $[X_\Ncal]^V(x,v)= \mu(X)_x \cdot v\in \Ncal\vert_x$,
where $x\in \Xcal$, $v\in \Ncal_x$.

For $X=\beta$, $-\mu(\beta)$ is equal to the linear action $\Lcal(\beta)$ of $\beta$ on $\Ncal_x$. If $\mathcal{K}\subset \Xcal$ is a compact subset, there exists
$c_\mathcal{K}>0$ such that $\|[X_\Ncal]^V(x,v)\|\leq c_\mathcal{K}\|X\|\,\|v\|$ for all $(x,v)\in \pi^{-1}(\mathcal{K})$.

The Kirwan vector field $\kappa_{\Phi}(n)=-\Phi(n)\cdot n$ admits then the decomposition $\kappa_{\Phi}=[\kappa_{\Phi}]^V\oplus [\kappa_{\Phi}]^H$ with
 $[\kappa_{\Phi}]^V(n)= - \mu(\Phi(n))_x \cdot v\in \Ncal\vert_x$.

Choose a relatively compact open neighborhood $\Vcal$ of $\Phi_\Xcal^{-1}(\beta)$ in $\Xcal$, and consider
$\Ucal_\epsilon:=\{n\in \pi^{-1}(\Vcal)\,;\, \|\Phi(n)-\beta\|<\epsilon\}$ which is an open neighborhood of $\Phi_\Xcal^{-1}(\beta)$ in $\Ncal$.

  If $R(n)=\Phi(n)-\beta$,  we have
\begin{equation}\label{eq:kappa-vertical}
\mu(\Phi(n))_x \cdot v= \beta\cdot v + \mu(R(n))_x\cdot v.
\end{equation}
  We choose $\epsilon$ sufficiently small so that
  $\|\mu(R(n))_x\cdot v\|< \|\beta\cdot v\|$ for any $n=(x,v)\in \Ucal_\epsilon$.
  If  $\kappa_\Phi(n)=0$, taking the scalar product with the vertical vector $\beta\cdot v$, we obtain
$$
0=-\langle \kappa_{\Phi}(n),\beta \cdot v\rangle=
-\langle [\kappa_{\Phi}]^V(n),\beta \cdot v\rangle=\|\beta\cdot v\|^2+ \langle \mu(R(n))_x\cdot v,\beta \cdot v\rangle
$$
and this implies $v=0$. Thus the set of zeroes of $\kappa_\Phi$ on $\Ucal_\epsilon$ is the set of zeroes of $\kappa_\Phi$ on $\Xcal$.
\end{proof}\bigskip

Let $\Ecal\to \Ncal$ be a graded Clifford bundle on the manifold $\Ncal$. We can then define the $K$-transversally elliptic class
$\sigma(\Ncal,\Ecal,Z,\Phi)\in\Ko_K(\T^*_K\Ncal)$ attached to the compact component $Z=\Phi^{-1}_\Xcal(\beta)$.

Let $\mathbb{d}_\beta(\Ecal)$ be the graded Clifford bundle on $\Xcal$ induced by $\Ecal$
(see Definition \ref{defi:q-beta}). We can also define the $K$-transversally elliptic  class \break
$\sigma(\Xcal,\mathbb{d}_\beta(\Ecal),Z,\Phi_\Xcal)\in\Ko_K(\T^*_K \Xcal)$.

\begin{prop}[Variation on Atiyah-Segal localization formula]

\label{prop:variation}

The following relation holds in $\Ko_K(\T^*_{K} \Ncal)$ :
\begin{equation}
\label{eq:localization}
\At_\beta\left(\sigma(\Xcal,\mathbb{d}_\beta(\Ecal),Z,\Phi_\Xcal)\right)=\sigma(\Ncal,\Ecal,Z,\Phi).
\end{equation}
In particular $\Qcal_K(\Ncal,\Ecal,Z,\Phi)=\Qcal_K\left(\Xcal,\mathbb{d}_\beta(\Ecal)\otimes \Sym(\Ncal_\JJbeta),Z,\Phi_\Xcal\right)$.
\end{prop}
\begin{proof}
The proof is a variation of the proof of Proposition \ref{prop:atiyah-segal}.
We consider  the vector field $\kappa^0(n)=-\beta \cdot n \oplus \kappa_\Xcal(x)$
where $\kappa_\Xcal(x)\in T_x \Xcal$ is lifted horizontally in $\T \Ncal=\pi^* \T \Xcal\oplus \pi^*\Ncal$.
Then $\kappa^t=t \kappa^0+(1-t) \kappa_\Phi$
has vertical component
$-\beta\cdot v+(1-t) \mu(R(n))_x\cdot v$.
The restriction to $\Xcal$ of $\kappa^t$ is the vector field $\kappa_{\Phi_\Xcal}$ for any $t\in [0,1]$.
We then consider
$$\sigma^t(n,\nu)=\clif_{n} (\tilde{\nu}-\kappa^{t}(n))
$$
for  $t\in [0,1]$.
We see, as before,  that  $\Char(\sigma^t)$ intersected with
$\T^*_K(\Ncal)$ remains equal to $Z$.
Thus the class of $\sigma^t$ is constant in $\Ko_K(\T^*_K \Ncal)$, and $\sigma^0$ is equal to
$\At_\beta\left(\sigma(\Xcal,\mathbb{d}_\beta(\Ecal),Z,\Phi_\Xcal)\right)$.
\end{proof}

\section{Non abelian localization formula}\label{sec:non-abelian-localization}

In this section, given an equivariant map $\Phi: M\to \kgot^*$,  we establish a general localization formula for the equivariant index of a Dirac operator.

\subsection{Localization on  $\Phi^{-1}(0)$}\label{subsec:loc at zero}

Let $\Phi: M\to \kgot^*$ be an equivariant proper map. We describe what is the representation
$\Qcal_K(M,\Ecal,\Phi^{-1}(0),\Phi)$, in the special case where $0$ is a regular value of $\Phi$, and
where $K$ acts infinitesimally freely on the compact submanifold $P=\Phi^{-1}(0)$.
Let $M_{red} :=P/K$ be the corresponding ``reduced'' space, and let $\pi: P\to M_{red}$ be the projection map.

\begin{rem}
By Proposition \ref{prop: O regular}, the hypothesis that $0$  is a regular value implies that $K$ acts infinitesimally freely on $P$ in the case where $\Phi$ is a moment map.
\end{rem}


On $P$, we obtain an exact sequence
$0\longrightarrow \T P\longrightarrow \T M\vert_P \stackrel{\T\Phi}{\longrightarrow} [\kgot^*]\to 0$,
where $[\kgot^*]$ is the trivial bundle $P\times\kgot^*$. We have also an orthogonal decomposition
$\T P= \T_{K} P \oplus [\kgot]$ where $[\kgot]$ is the subbundle identified to $P\times\kgot$ through the map
$(p,X)\mapsto X\cdot p$. So $\T M\vert_P$ admits the orthogonal decomposition
$\T M\vert_P \simeq \T_{K} P\oplus [\kgot] \oplus  [\kgot^*]$. We rewrite this as
\begin{equation}\label{eq:tangent-P}
\T M\vert_P \simeq \T_{K} P\oplus [\kgot_\C]
\end{equation}
with the convention $[\kgot]=P\times(\kgot\otimes\R) $ and $[\kgot^*] = P\times (\kgot\otimes i\R)$.
Note that the bundle $\T_{K} P$ is naturally identified with $\pi^*(\T M_{red})$.

Following Definition \ref{defi:q-beta}, we can divide the graded Clifford bundle $\Ecal\vert_P$ by the $\spinc$-bundle
$\bigwedge \kgot_\C$ for the vector space $\kgot_\C$.

\begin{defi}\label{def:E-red}
Let $\Ecal_{red}$ be the graded Clifford  bundle on the vector bundle $\T_{K} P\to P$ such that
$$
\Ecal\vert_P \simeq\Ecal_{red}\otimes [\bigwedge\kgot_\C]
$$
is an isomorphism of graded Clifford  bundles on $\T M\vert_P$. We still denote by $\Ecal_{red}$ the induced
Clifford (orbi-)bundle on $\T M_{red}$.
\end{defi}

For any representation space $V_\mu$ of $K$, we can form the twisted Clifford  bundle $\Ecal_{red}\otimes \Vcal^*_\mu$ of the Clifford bundle
$\Ecal_{red}$ by the equivariant (orbi)-bundle $\Vcal^*_\mu=P\times_K V^*_\mu$.

Consider  the horizontal Dirac symbol, $\sigma_{hor}(P,\Ecal_{red})$ on $P$  defined in Section \ref{sec:horizontal}.
This is a $K$ transversally elliptic symbol on $P$. We now prove
\begin{theo}\label{theo:localisation-Phi-0}
If $0$ is a regular value of $\Phi$ and $K$ acts locally freely on $\Phi^{-1}(0)$, then $\Phi^{-1}(0)$ is a component of $Z_\Phi$, and  we have
\begin{eqnarray*}
\Qcal_K(M,\Ecal,\Phi^{-1}(0),\Phi)&=&\indice_K^P(\sigma_{hor}(P,\Ecal_{red}))\\
&=&\sum_{\mu\in \hat K} \Qcal\left(M_{red},\Ecal_{red}\otimes \Vcal_{\mu}^*\right) V_\mu.
\end{eqnarray*}
In particular we have $\left[\Qcal_K(M,\Ecal,\Phi^{-1}(0),\Phi)\right]^K=\Qcal\left(M_{red},\Ecal_{red}\right)$.
\end{theo}

\begin{proof}  Let us study the pushed symbol
$\sigma(M,\Ecal, \Phi)(m,\nu)=\sigma_{\Ecal_m}(\tilde{\nu}-\kappa_\Phi(m))$ when $m$ is in a neighborhood
of $P=\Phi^{-1}(0)$. So we can assume that  $M=P\times\kgot^*$. For a point $m=(x,\xi)\in P\times \kgot^*$,
the tangent vector $\tilde{\nu}\in \T_m M$ is represented by $\eta\oplus a \oplus b$ : here  $\eta\in \T_K P$, $a\in\kgot$
represents the tangent vector $a\cdot x\in \T_x P$, and $b=\T_m\Phi(\tilde{\nu})\in\kgot^*$.
For $X\in \kgot$, the vector field $X_M(x,\xi)$
associated to the action of $K$ on $P\times \kgot^*$ is represented by $0\oplus -X \oplus [\xi,X]$:
in particular the Kirwan vector field $\kappa_\Phi(x,\xi)$ is equal to
$0\oplus -\xi\oplus 0$. For the pushed symbol we obtain
$$
\sigma(M,\Ecal, \Phi)(m,\nu)=
\clif_{\Ecal_{red}}(\eta)\boxtimes \clif_{\wedge \kgot_\C} ((a +\xi)\oplus ib).
$$

Let us see that the support of the pushed symbol intersected with $\T^*_K M$  is just equal to $P$.
Indeed if  $(x,\xi;\eta\oplus a \oplus b)\in \T^*_K M$, we have $([\xi,X], b)-(X,a)=0$ for all $X\in\kgot$, e.g. $a=[b,\xi]$.
Then, if $\clif_{\Ecal_x}(\eta)\boxtimes \clif_{\wedge \kgot_\C} ((a +\xi)\oplus ib)$ is not invertible, we obtain
$\eta=0$, $b=0$ and $a+\xi=0$. If furthermore  $(x,\xi;\eta\oplus a \oplus b)$ is in $\T^*_K M$, we obtain $a=\xi=0$.
 Thus we see that $\Phi^{-1}(0)=P$ is a component of $Z_\Phi$, and
$\sigma(M,\Ecal, \Phi)$  restricted to a neighborhood of $P$ defines a transversally elliptic symbol on $M$. We can then define its index
$\Qcal_K(M,\Ecal,\Phi^{-1}(0),\Phi).$

Consider now the injection $i:P\to P\times \kgot^*=M$ and the horizontal symbol
$\sigma_{hor}(P,\Ecal_{red})$ on $P$.
By definition (see Section \ref{sec:directimage}), the direct image
$i_!\sigma_{hor}(P,\Ecal_{red})$ is defined to be the restriction to $\T_K^* M$ of the symbol
$$
\clif_{\Ecal_{red}}(\eta)\boxtimes \clif_{\wedge \kgot_\C} (\xi \oplus ib).
$$

We see that for $t\in [0,1]$, the support of the symbol
$$\sigma^t(x,\xi,\nu)=\clif_{\Ecal_{red}}(\eta)\boxtimes \clif_{\wedge \kgot_\C} ((ta +\xi)\oplus ib)
$$
intersected with $\T^*_KM$ remains equal to $P$.

Thus  $\sigma(M,\Ecal,\Phi)$  defines the same class in $K$-theory than the symbol
$i_!\sigma_{hor}(P,\Ecal_{red})$. Thus by the direct image property of the symbol, we obtain our result.

\end{proof}

\bigskip

If furthermore we had an action of $G\times K$ on $M$,
we would obtain an action of $G$ in each space
$\Qcal(M_{red},\Ecal_{red}\otimes \Vcal^*_\mu)$
and the $G\times K$ decomposition
\begin{equation}\label{eq:localisation-0-K-L}
\Qcal_{G\times K}(M,\Ecal,\Phi^{-1}(0),\Phi)=
\sum_{\mu\in \hat K}\Qcal_G(M_{red},\Ecal_{red}\otimes \Vcal_\mu^*)\otimes V_\mu.
\end{equation}

\subsection{Coadjoint orbits and slices}

Let $\Phi: M\to \kgot^*$ be our $K$-equivariant map (with $K$ not necessarily connected).

Consider the coadjoint action of  $K$ on $\kgot^*$.
Let $\xi\in \kgot^*$.
Let us consider the $K_\xi$-invariant decomposition $\kgot=\kgot_\xi\oplus \qgot$, with dual decomposition
$\kgot^*=\kgot_\xi^*\oplus \qgot^*$.

\begin{prop}\label{prop:slice}
Let $\xi\in \kgot^*$, and let $B$ be a sufficiently small $K_\xi$-invariant
neighborhood of $\xi$ in $\kgot_\xi^*$.
Then

\begin{itemize}
\item  $Y=\Phi^{-1}(B)$ is a $K_\xi$-invariant submanifold of $M$ (perhaps empty).
\item  $\Phi^{-1}(KB)$ is an open neighborhood of $\Phi^{-1}(K\xi)$, and is  diffeomorphic to $K\times_{K_\xi} Y$.
\end{itemize}
\end{prop}

We will call the manifold $Y$ a slice of $M$ at $\xi$.

\begin{proof}
Consider the coadjoint orbit $K\xi$.
The map $X\to X\cdot \xi$ is  a surjective map from $\kgot$ to  $\qgot^*$ and $\kgot_\xi^*$
identifies to the normal space to the coadjoint orbit $K\xi$ at $\xi$.
The tubular neighborhood theorem  implies that if $B$ is sufficiently small, then  $KB$ is a neighborhood of $K\xi$ isomorphic to $K\times_{K_\xi}B$.

Write $\Phi_\qgot$ for the component of $\Phi$ in $\qgot^*$.
The differential map  $T_y \Phi_\qgot: \T_y M  \to \qgot^*$ is surjective, if  $\Phi(y)=\xi$.
Thus it is surjective at any point   $y\in Y$, if $B$ is sufficiently small.
The rest of the proof is standard.
\end{proof}

\bigskip

We also consider $\xi$ as an element of $\kgot_\xi$, using our $K$-invariant identification $\kgot=\kgot^*$.
So $\xi$ induces an invertible skew-symmetric transformation on $\qgot$.
Denote by
$\qgot_{J_\xi}$  the vector space $\qgot=\kgot/\kgot_{\xi}$ equipped with the complex structure $J_\xi$,

Suppose now that $M$ carries a graded Clifford  bundle $\Ecal$. The tangent bundle $\T M$, when restricted to the slice $Y$, decomposes as $\T M\vert_{Y}\simeq  \T Y \oplus [\kgot/\kgot_\xi]$, where $[\kgot/\kgot_\xi]$ denotes the trivial bundle $\kgot/\kgot_\xi\times Y$.

As in Definition \ref{defi:q-beta},
we can divide the graded Clifford bundle $\Ecal\vert_{Y}$ by the $\spinc$-bundle
$\bigwedge \qgot_{J_\xi}$  for the Euclidean vector space $\qgot$.
  There exists a unique graded Clifford bundle
$\Ecal_{\ddY}$ over the tangent bundle $\T Y$, such that, when we restrict $\Ecal$ to $Y$, we get the isomorphism
\begin{equation}\label{eq:E-slice-xi}
\Ecal\vert_{Y}\simeq\Ecal_{\ddY}\otimes  \left[\bigwedge \qgot_{J_\xi}\right]
\end{equation}
of graded Clifford bundles relative to the decomposition $\T M\vert_{Y}\simeq  \T Y \oplus [\kgot/\kgot_\xi]$.

\subsection{Localization on $Z_\beta$}

We consider our equivariant map $\Phi:M\to \kgot^*$, and we assume that
$$
Z_\beta= K(M^\beta\cap \Phi^{-1}(\beta))
$$
is a compact component of $Z_\Phi$.
Here $\beta$ is a non zero element in $\kgot^*$.

Let $\Qcal_K(M,\Ecal,Z_\beta,\Phi)\in \hat{R}(K)$ be the localized equivariant index attached to an equivariant graded
Clifford bundle $\Ecal$. Let $\mathbb{d}_\beta(\Ecal)$ be the induced graded Clifford bundle on $M^\beta$ (see Section \ref{sec:loc-Atiyah-Segal}).
The restriction of $\Phi$ to $M^\beta$ is a map that takes value in $\kgot^*_\beta$ : for simplicity we still denote
$\Phi: M^\beta\to \kgot^*_\beta$.  The set $M^\beta\cap \Phi^{-1}(\beta)$ is a compact component of
$Z_{\Phi}\cap M^\beta$.
Let $\Ncal_\JJbeta$ be the (complex) normal bundle of $M^\beta$ in $M$ polarized by $\beta$. The aim of this subsection is to prove the following localization formula.

\begin{theo}\label{theo:localization-Z-beta}
We have the following relation in $\hat{R}(K)$:
\begin{eqnarray*}
\lefteqn{\Qcal_K(M,\Ecal,Z_\beta,\Phi)=}\\
& & \mathrm{Ind}^K_{K_{\beta}}
\left(\Qcal_{K_{\beta}}(M^\beta,\mathbb{d}_\beta(\Ecal)\otimes \Sym(\Ncal_\JJbeta),
M^\beta\cap \Phi^{-1}(\beta),\Phi)\otimes \bigwedge (\kgot/\kgot_{\beta}\otimes\C)\right).
\end{eqnarray*}
\end{theo}

\begin{proof} The proof will have two steps.

{\bf First Step : } Let $B\subset \kgot^*$ be a small $K_\beta$-invariant ball around $\beta$  in $\kgot_{\beta}^*$ and let $Y=\Phi^{-1}(B)$ be
the corresponding slice. A neighborhood of $Z_\beta$ in $M$ is then diffeomorphic with $K\times_{K_{\beta}}Y$. We denote by
$\Phi_Y: Y\to \kgot^*_\beta$ the restriction of $\Phi$  to the submanifold $Y$. Let $\Ecal_{\ddY}$ be the induced Clifford bundle on
$\T Y$ (see (\ref{eq:E-slice-xi})).

The restriction of the symbol $\sigma(M,\Ecal,\Phi)$ to $\T^* Y$ is equal to
$\sigma(Y,\Ecal_{\ddY},\Phi_Y)\otimes \bigwedge \qgot_\JJbeta$, hence the induction formula of Proposition \ref{Induction formula} gives that
$$
\Qcal_K(M,\Ecal,Z_\beta,\Phi)=
\mathrm{Ind}^K_{K_{\beta}}
\left(\Qcal_{K_{\beta}}(Y,\Ecal_{\ddY}, M^\beta\cap \Phi^{-1}(\beta),\Phi_Y)\otimes \bigwedge \qgot_\JJbeta\right)
$$

(as $\Phi^{-1}(\beta)$ is contained in $Y$,  $M^\beta\cap \Phi^{-1}(\beta)=Y^\beta\cap \Phi_Y^{-1}(\beta)$).
\medskip

{\bf Second Step : } Now we use the variation of the Atiyah-Segal-Singer localization formula for computing
$\Qcal_{K_{\beta}}(Y,\Ecal_{\ddY}, M^\beta\cap \Phi^{-1}(\beta),\Phi_Y)$.
As we are working on a neighborhood of $Y^\beta\cap \Phi_Y^{-1}(\beta)$, we may assume that $Y$ is the
normal bundle $\Ncal_Y$ of $Y^{\beta}$ in $Y$.

Let $\mathbb{d}_\beta(\Ecal_{\ddY})$ be the Clifford bundle on the submanifold
$Y^\beta$ determined by $\Ecal_{\ddY}$ (see Section \ref{sec:loc-Atiyah-Segal}). Let $(\Ncal_Y)_{J_\beta}$ be the normal bundle of $Y^\beta$ in $Y$ polarized by $\beta$.
We  apply the variation of Atiyah-Segal-Singer localisation formula to $\Ncal_Y\to Y^{\beta}$ (the group $K$ being now $K_{\beta}$) and we obtain the following formula:
\begin{eqnarray*}
\lefteqn{\Qcal_{K_{\beta}}(Y,\Ecal_{\ddY}, M^\beta\cap \Phi^{-1}(\beta),\Phi_Y)=}\\
&&\Qcal_{K_{\beta}}(Y^\beta,\mathbb{d}_\beta(\Ecal_{\ddY})\otimes \Sym((\Ncal_Y)_{J_\beta}), M^\beta\cap \Phi^{-1}(\beta),\Phi_Y).
\end{eqnarray*}

The manifold $Y^\beta$ is an open subset of $M^\beta$, and for the normal bundles $\Ncal$ of $M^{\beta}$ in $M$ versus $\Ncal_Y$ (the normal bundle of $Y^{\beta}$ in $Y$), we have
$\Ncal=\Ncal_Y\oplus [\kgot/\kgot_{\beta}]$. For the polarized ones, we get $\Ncal_\JJbeta=(\Ncal_Y)_\JJbeta\oplus [\qgot_\JJbeta]$.
If we use the relation $\Sym(\qgot_\JJbeta)\otimes \bigwedge \qgot_\JJbeta=1$ in $\hat{R}(K_{\beta})$, we see that
\begin{eqnarray*}
\Sym(\Ncal_\JJbeta)\otimes \left[\bigwedge (\kgot/\kgot_{\beta}\otimes\C)\right]&=&
\Sym((\Ncal_Y)_\JJbeta)\otimes\Sym(\qgot_\JJbeta)\otimes \left[\bigwedge \qgot_\JJbeta\right]
\otimes \left[\bigwedge \overline{\qgot_\JJbeta}\right]\\
&=& \Sym((\Ncal_Y)_\JJbeta)\otimes\left[\bigwedge \overline{\qgot_\JJbeta}\right].
\end{eqnarray*}
So we are left to check that $\mathbb{d}_\beta(\Ecal_{\ddY})\otimes \left[\bigwedge \qgot_\JJbeta\right]$ and
$\mathbb{d}_\beta(\Ecal)\otimes \left[\bigwedge \overline{\qgot_\JJbeta}\right]$ are
equal\footnote{In fact we can check that $\mathbb{d}_\beta(\Ecal_{\ddY})=(-1)^{\dim \qgot_\JJbeta}\det(\qgot_\JJbeta)^{-1}\mathbb{d}_\beta(\Ecal)$.} as Clifford module over $\T Y^\beta$.

We have the identity $\Ecal\vert_Y= [\bigwedge \qgot_\JJbeta]\otimes \Ecal_{\ddY}$ as Clifford modules over
$ \kgot/\kgot_{\beta} \oplus \T Y$. We have also
$(\Ecal_{\ddY})\vert_{Y^\beta}=\bigwedge \overline{(\Ncal_Y)_\JJbeta} \otimes \mathbb{d}_\beta(\Ecal_{\ddY})$ as an equality of
Clifford modules over $ \Ncal_Y\oplus\T Y^\beta $. Finally we get
\begin{equation}\label{eq:Ecal-beta-1}
\Ecal\vert_{Y^\beta}= \bigwedge \overline{(\Ncal_Y)_\JJbeta} \otimes\left[\bigwedge \qgot_\JJbeta\right] \otimes \mathbb{d}_\beta(\Ecal_{\ddY})
\end{equation}
as an equality of Clifford modules over $\Ncal_Y \oplus  [\kgot/\kgot_{\beta}]\oplus \T Y^\beta$. By definition, we have also
$\Ecal\vert_{M^\beta}= \bigwedge \overline{\Ncal_\JJbeta} \otimes \mathbb{d}_\beta(\Ecal)$ as an equality of modules over
$ \Ncal\oplus\T M^\beta $. If we restrict the previous identity to the open subset $Y^\beta$, we get
\begin{equation}\label{eq:Ecal-beta-2}
\Ecal\vert_{Y^\beta}= \bigwedge \overline{(\Ncal_Y)_\JJbeta} \otimes \left[\bigwedge \overline{\qgot_\JJbeta}\right]\otimes\mathbb{d}_\beta(\Ecal).
\end{equation}
If we use (\ref{eq:Ecal-beta-1}) and (\ref{eq:Ecal-beta-2}), we get the desired identity.

\end{proof}

\subsection{The non abelian localization theorem}\label{subsec:nonabelianlocalization}

Let $\Ecal\to M$ be a equivariant graded Clifford module over an even dimensional  $K$-manifold $M$ ($M$ is not necessarily compact, the group $K$ is compact but not necessarily connected).
Let $\Phi$ be a moment map, such that $Z_\Phi$ is compact.
We can then construct $\Qcal_K(M,\Ecal, \Phi)\in \hat R(K)$.
If $M$ is compact, $\Qcal_K(M,\Ecal,\Phi)$ is just equal to $\Qcal_K(M,\Ecal)$.

The set of zeroes $Z_\Phi$ is a disjoint union of its components $Z_\beta,\beta\in\Bcal$, with $\Bcal$ finite (see (\ref{eq=Z-Phi-beta})).
The following decomposition of  $\Qcal_K(M,\Ecal,\Phi)$ is the $\K$-theoretical analogue of Witten non abelian
localization theorem in equivariant cohomology.
This reduces the study of $\Qcal_K(M,\Ecal,\Phi)$ by induction to the neighborhood of the zeroes of the moment map.

\begin{theo}[Non abelian localization theorem]
\label{theo;nonabelian}
We have
$$
\Qcal_K(M,\Ecal,\Phi)=\sum_{\beta\in\Bcal}
\Qcal_K(M,\Ecal,Z_\beta,\Phi).$$

Furthermore, for $\beta\neq 0$, the generalized character $\Qcal_K(M,\Ecal,Z_\beta,\Phi)$ is equal to
$$
\mathrm{Ind}^K_{K_{\beta}}
\left(\Qcal_{K_{\beta}}(M^\beta,\mathbb{d}_\beta(\Ecal)\otimes \Sym(\Ncal_\JJbeta),\Phi^{-1}(\beta)\cap M^\beta,\Phi)\otimes
\bigwedge \qgot_\JJbeta\otimes \bigwedge  {\overline\qgot_\JJbeta}\right).
$$
\end{theo}

Here, $\Bcal$ is a set of representatives of the orbits of $K$ in $\Phi(Z_\Phi)$.
For $\beta \in \Bcal$, $M^\beta$ is the submanifold of $M$ stable by $\beta\in \kgot$, $Z_\beta$ is $K(\Phi^{-1}(\beta)\cap M^{\beta})$,
  $\Ncal_\JJbeta$ is the normal bundle on $M^{\beta}$ in $M$ polarized by $\beta$,  $\mathbb{d}_\beta(\Ecal)$ is the $K_{\beta}$ equivariant Clifford bundle on $M^{\beta}$ defined by
$$\Ecal|_{M^{\beta}}=\mathbb{d}_\beta(\Ecal)\otimes \bigwedge {\overline \Ncal_\JJbeta},$$ and
$\qgot_\JJbeta$ denotes the vector space $\qgot=\kgot/\kgot_{\beta}$ equipped with the complex structure $J_\beta$.

As $\beta$ is a central element in $\kgot_{\beta}$,
we are reduced by induction to the study of a localized character at a central element
for the lower dimensional manifold $M^{\beta}$ acted on by the subgroup $K_{\beta}$.

When $0$ is a regular value of the moment map,
the term
$\Qcal_K(M,\Ecal,\Phi^{-1}(0),\Phi)$ is given explicitly:
$$
\Qcal_K(M,\Ecal,\Phi^{-1}(0),\Phi)=\sum_{\lambda\in \hat K} \Qcal\left(M_{red},\Ecal_{red}
\otimes \Vcal_{\lambda}^*\right) V_\lambda.
$$

It follows that  the multiplicities
$\mm(\lambda)$ of the irreducible representations $V_\lambda$
in $\Qcal_K(M,\Ecal,\Phi^{-1}(0),\Phi)$ are given explicitly as the equivariant index of a orbi-bundle over the orbifold
$M_{red}=\Phi^{-1}(0)/K$.

When $K$ is connected, we may  parameterize $\hat K$ by the lattice of weights intersected with a positive Weyl chamber. In this case Kawasaki formula implies that $\lambda \to \mm(\lambda)$ is a quasi-polynomial function of $\lambda$.
We will return to this theme in Subsection \ref{subsec:quasipol}.

\subsection{Example of $\T^*\tilde{K}$}\label{subsec:cotangent}

We  consider in this subsection the paradigmatic  example of the cotangent space
$\T^*\tilde{K}$ where $\tilde{K}$ is a compact Lie group. This is a
 $\tilde{K}$-Hamiltonian manifold with moment map $\Phi$,
and in the philosophy of quantization, its quantization is
$L^2(\tilde K)$.
We here study the space
$\Qcal_{\tilde{K}}(\T^*\tilde{K},\Ecal,\Phi)$,
and verify that indeed
this is the space of  $\tilde K$-finite functions on $\tilde K$ for a particular Clifford module $\Ecal$.
This example will not be used in the rest of this article.

Let $\tilde{K}$ be a compact Lie group and let $K$  be a closed subgroup.
Let $i$ denote the inclusion of $K$ into $\tilde{K}$, $i : \kgot \to \tilde{\kgot}$ the induced embedding of Lie algebras,
and $\pi : \tilde{\kgot}^* \to \kgot^*$ the dual projection.

We consider the following action of $\tilde{K}\times K$ on $\tilde{K}$ : $(\tilde{k},k)\cdot a= \tilde{k} a k^{-1}$.

The tangent bundle $\T \tilde{K}$ is identified with $\tilde{K}\times\tilde{\kgot}$ through the left translations: to
$(a,X)\in \tilde{K}\times\tilde{\kgot}$ we associate $\frac{d}{dt} e^{tX}a\vert_{t=0}\in \T_a \tilde{K}$.
The action of  $\tilde{K}\times K$ on the cotangent bundle $\T^*\tilde{K}\simeq \tilde{K}\times \tilde{\kgot}^*$ is then
$$
(\tilde{k},k)\cdot( a,\xi)= (\tilde{k}ak^{-1},\tilde{k}\xi).
$$

The Liouville form  $\lambda$ of  $\T^*\tilde{K}$ is  thus given by
$\lambda_{a,\xi}(X,\eta)=\langle\xi,X\rangle$ for $(a,\xi)\in \T^*\tilde{K}\simeq \tilde{K}\times\tilde{\kgot}$ and
$(X,\eta)\in\T_a\tilde{K}\times\T_\xi\tilde{\kgot}$.
The symplectic form on $\T^*\tilde{K}$ is $\Omega:=d\lambda$.

The corresponding moment map relative to the $\tilde{K}\times K$-action is the
map $\Phi=\Phi_{\tilde{K}} \oplus \Phi_K : \T^*\tilde{K} \to \tilde{\kgot}^*\oplus \kgot^*$ defined by
\begin{equation}\label{eq:momentcotangent}
\Phi_{\tilde{K}}(a,\xi)=-\xi,\quad \Phi_K(a,\xi)=\pi(a^{-1}\cdot\xi).
\end{equation}

\begin{lem}\label{lem:Z-phi-t}
The set $Z_\Phi\subset \T^*\tilde{K}$ is equal to $\tilde{K}\times\{0\}$. More generally, if we consider the equivariant map
$\Phi_t=\Phi_{\tilde{K}}\oplus t \Phi_K$ for any $t\geq 0$, the set  $Z_{\Phi_t}\subset \T^*\tilde{K}$ is still equal to
$\tilde{K}\times\{0\}$.
\end{lem}

\begin{proof}Let $\kappa_t$ be the Kirwan vector field on $\T^*\tilde{K}$ attached to the map $\Phi_t$. A direct computation
gives that for $(a,\xi)\in \T^*\tilde{K}$, the vector $\kappa_t(a,\xi)\in \T_{(a,\xi)}(\tilde{K}\times\tilde{\kgot}^*)=
\tilde{\kgot}\times \tilde{\kgot}^*$ is equal to
$(\kappa_t^1(a,\xi),0)$ with
$$
\kappa^1_t(a,\xi)= \xi+ t\,a\,\pi(a^{-1}\xi).
$$
For $t\geq 0$, we see that $\kappa^1_t(a,\xi)=0$ if and only if $\xi=0$.
\end{proof}

\medskip

\medskip

Let $\Ecal$ be a $\tilde{K}\times K$-equivariant graded Clifford bundle on the manifold $\T^*\tilde{K}$. Thanks to Lemma
\ref{lem:Z-phi-t}, we know that $Z_\Phi$ is compact. Hence we can consider the generalized index
$$
\Qcal_{\tilde{K}\times K}(\T^*\tilde{K},\Ecal,\Phi)\in \hat{R}(\tilde{K}\times K).
$$

The restriction of
$\Ecal$ at $(1,0)\in \T^*\tilde{K}$ defines a graded Clifford module $\Ecal\vert_{(1,0)}$ for the Euclidean
vector space $\T(\T^*\tilde{K})\vert_{(1,0)}\simeq \tilde{\kgot}\times\tilde{\kgot}^*$. We use the identification
$\tilde{\kgot}\times\tilde{\kgot}^*\simeq \tilde{\kgot}_\C, (X,\xi)\mapsto X\oplus i \xi$.
Thus $\bigwedge \tilde{\kgot}_\C$ is an irreducible Clifford module for $\tilde{\kgot}\times\tilde{\kgot}^*$.

Following Definition \ref{defi:q-beta}, we can ``divide'' the graded Clifford bundle $\Ecal\vert_{(1,0)}$ by the
irreducible Clifford module $\bigwedge \kgot_\C$.

\begin{defi}\label{def:E-0}
Let $E_{\tilde{K}\times K}$ be the graded $\tilde{K}\times K$-representation space such that
$$
\Ecal\vert_{(1,0)}\simeq E_{\tilde{K}\times K}\otimes \bigwedge\kgot_\C.
$$
is an isomorphism of graded Clifford modules on $\tilde{\kgot}\times\tilde{\kgot}^*$.
\end{defi}

\begin{prop}\label{prop:Q-TK}
We have the following relation in $\hat{R}(\tilde{K}\times K)$ :
$$
\Qcal_{\tilde{K}\times K}(\T^*\tilde{K},\Ecal,\Phi)=\sum_{V\in \widehat{\tilde{K}}}
(V^*\otimes V\vert_K)\otimes E_{\tilde{K}\times K}.
$$
\end{prop}

\begin{rem}
If $E_{\tilde{K}\times K}=\C$,  we obtain
that $\Qcal_{\tilde{K}\times K}(\T^*\tilde{K},\Ecal,\Phi)$ is $L^2(\tilde K)$ after completion,
as expected.
\end{rem}

\begin{proof}
Thanks to Lemma \ref{lem:Z-phi-t}, we know that  $t\in [0,1]\mapsto\sigma(\T^*\tilde{K},\Ecal,\Phi_t)$ is an
homotopy of transversally elliptic symbols. It follows that \break $\Qcal_{\tilde{K}\times K}(\T^*\tilde{K},\Ecal,\Phi)=
\Qcal_{\tilde{K}\times K}(\T^*\tilde{K},\Ecal,\Phi_{\tilde{K}})$. We note that
the computation of $\Qcal_{\tilde{K}\times K}(\T^*K,\Ecal,\Phi_{\tilde{K}})$ follows from (\ref{eq:localisation-0-K-L}), since
$Z_{\Phi_{\tilde{K}}}=\Phi_{\tilde{K}}^{-1}(0)$, $0$ is a regular value of $\Phi_{\tilde{K}}$ and $\Phi_{\tilde{K}}^{-1}(0)/\tilde{K}$
is reduced to a point.
\end{proof}

\section{$[Q,R]=0$ Theorem}\label{sec:QR-theorem}

Let $M$ be an  even dimensional $K$-manifold equipped with an equivariant graded Clifford bundle $\Ecal$.

When $M$ is compact, we consider the representation $\Qcal_K(M,\Ecal)$. In this section,  we study $[\Qcal_K(M,\Ecal)]^K$,
with the help of a moment map $\Phi$, and the non abelian localization formula, in particular,  we will find sufficient conditions
on the couple $(\Ecal,\Phi)$ in order that the only component of $Z_\Phi$ contributing to
$[\Qcal_K(M,\Ecal)]^K$ is $\Phi^{-1}(0)$. We will be considering in more detail the case where $M$ is an almost complex manifold.

Our discussion extends to the case where $M$ is non-compact. We will study the generalized character
$\Qcal_K(M,\Ecal,\Phi)$ attached to a moment map $\Phi$ such that
$Z_\Phi$ is compact.

\subsection{A preliminary formula}
Let $M$ be an  even dimensional  $K$-manifold equipped with an equivariant graded Clifford bundle $\Ecal$ and
a moment map $\Phi:M\to \kgot^*$ such that $Z_\Phi$ is compact. We study
$[\Qcal_K(M,\Ecal,\Phi)]^K$ via the non abelian localization formula.
The notations in this subsection are as in  Subsection \ref{subsec:nonabelianlocalization}.
From Theorem \ref{theo;nonabelian}, we obtain
$$
\left[\Qcal_K(M,\Ecal,\Phi)\right]^K=\sum_{\beta\in\Bcal}
\left[\Qcal_K(M,\Ecal,Z_\beta,\Phi)\right]^K
$$
and if $\beta\neq 0$, the integer $\left[\Qcal_K(M,\Ecal,Z_\beta,\Phi)\right]^K$ is equal to
$$
\left[\Qcal_{K_{\beta}}(M^\beta,\mathbb{d}_\beta(\Ecal)\otimes \Sym(\Ncal_\JJbeta),\Phi^{-1}(\beta)\cap M^\beta,\Phi)
\otimes \bigwedge \qgot_\JJbeta\otimes \bigwedge  {\overline\qgot_\JJbeta}
\right]^{K_{\beta}}.
$$

Using Lemma \ref{lem:triv}, the vanishing of this last expression  is automatic
if the linear action  (denoted $\Lcal(\beta)$)  of $\beta$ on the fibers of
$\mathbb{d}_\beta(\Ecal)\otimes \Sym(\Ncal_\JJbeta)
\otimes \bigwedge \qgot_\JJbeta\otimes \bigwedge  {\overline\qgot_\JJbeta}$ has only non-zero eigenvalues.
We thus introduce some definitions.

Consider a $K_{\beta}$-equivariant vector bundle $\Fcal\to M^{\beta}$. Let $\Xcal$ be a connected component of $M^\beta$. Let $c\in \R$. If all the eigenvalues of $\Lcal(\beta)$ acting on $\Fcal\vert_\Xcal$ are of the form $i\theta$ with $\theta>c$  (respectively $\theta\geq c$), we say that
$\frac{1}{i}\Lcal(\beta)>c$  on $\Fcal\vert_\Xcal$ (respectively
$\frac{1}{i}\Lcal(\beta)\geq c$).

\medskip

\begin{prop}\label{prop:QR-beta-general}
Let $\beta\neq 0$.

$\bullet$ If $\frac{1}{i}\Lcal(\beta)\geq 0$  on
$\mathbb{d}_\beta(\Ecal)\vert_\Xcal\otimes \det(\qgot_\JJbeta)^{-1}$ for any connected component
$\Xcal\subset M^\beta$ such that $\Xcal\cap \Phi^{-1}(\beta)\neq\emptyset$. Then
$$
\left[\Qcal_K(M,\Ecal,Z_\beta,\Phi)\right]^K=\left[\Qcal_{K_{\beta}}(M^\beta,\Ecal_{[\beta]},\Phi^{-1}(\beta)\cap M^{\beta},\Phi)\right]^{K_{\beta}}
$$
where
\begin{equation}\label{eq:E-beta-0}
\Ecal_{[\beta]}:=(-1)^{\dim \qgot_\JJbeta} [\mathbb{d}_\beta(\Ecal)\otimes \det(\qgot_\JJbeta)^{-1}]^\beta
\end{equation}
is a graded $K_{\beta}$-equivariant Clifford bundle on $M^\beta$.

$\bullet$ If $\frac{1}{i}\Lcal(\beta)>0$  on
$\mathbb{d}_\beta(\Ecal)\vert_\Xcal\otimes \det(\qgot_\JJbeta)^{-1}$ for any connected component
$\Xcal\subset M^\beta$ such that $\Xcal\cap \Phi^{-1}(\beta)\neq\emptyset$,
 we have $\Ecal_{[\beta]}=[0]$ in a neighborhood of $M^\beta\cap\Phi^{-1}(\beta)$ and then
$\left[\Qcal_K(M,\Ecal,Z_\beta,\Phi)\right]^K=0$.
\end{prop}

Remark that
$[\mathbb{d}_\beta(\Ecal)\otimes \det(\qgot_\JJbeta)^{-1}]^\beta$
is isomorphic to the subbundle of
$\mathbb{d}_\beta(\Ecal)$ formed by the vectors satisfying $\Lcal(\beta)v=\tr_{\qgot_\JJbeta}(\beta) v$.

\medskip

\begin{proof}

Introduce  the union $M^\beta_c=\cup \Xcal$ of the connected components $\Xcal$ of $M^\beta$ such
that $\Xcal$ intersects $\Phi^{-1}(\beta)$. Then $M^\beta_c$ is $K_\beta$ invariant and
clearly, we can replace $M^\beta$ by  $M^\beta_c$ in Theorem \ref{theo;nonabelian}.
Using Lemma \ref{lem:triv}, we obtain
$[\Qcal_K(M,\Ecal,Z_\beta,\Phi)]^K$ is equal to
$$
\left[\Qcal_{K_{\beta}}(M^\beta_c,[\Fcal_{M^\beta_c}]^\beta, M^\beta_c\cap\Phi^{-1}(\beta),\Phi)\right]^{K_{\beta}}
$$
where $\Fcal_{M^\beta_c}=\mathbb{d}_\beta(\Ecal)\vert_{M^\beta_c}
\otimes (\det(\qgot_\JJbeta)^{-1}
\otimes \det(\qgot_\JJbeta))\otimes \Sym(\Ncal_{J_\beta})\vert_{M^\beta_c}
\otimes  \bigwedge \qgot_\JJbeta
\otimes \bigwedge \overline{\qgot_\JJbeta}.$

Let $\Xcal$ be a connected component of ${M^\beta_c}$, thus $\Xcal$ intersects $\Phi^{-1}(\beta)$.
On  $\Sym(\Ncal_{J_\beta})\vert_\Xcal\otimes \bigwedge \qgot_\JJbeta$, we have $\frac{1}{i}\Lcal(\beta)\geq 0$ by definition of the complex structures associated to $\beta$, and
$[\Sym(\Ncal_{J_\beta})\vert_\Xcal\otimes \bigwedge \qgot_\JJbeta]^\beta=[\C]$. Similarly $\frac{1}{i}\Lcal(\beta)\geq 0$ on
$[\det(\qgot_\JJbeta)\otimes \bigwedge \overline{\qgot_\JJbeta}]$ and
$[\det(\qgot_\JJbeta)\otimes \bigwedge \overline{\qgot_\JJbeta}]^{\beta}=(-1)^{\dim \qgot_\JJbeta}$. Then
the Clifford bundle $[\Fcal_\Xcal]^\beta$ is reduced to $\Ecal_{[\beta]}$ if
$\frac{1}{i}\Lcal(\beta)\geq 0$ on $\mathbb{d}_\beta(\Ecal)\vert_\Xcal\otimes \det(\qgot_\JJbeta)^{-1}$. This proves the first point.
The second point is an obvious consequence of the first.
\end{proof}

\subsection{$\Phi$-positivity}

\begin{defi} Let $\Phi: M\to \kgot^*$ be an equivariant map.
\begin{itemize}
\item A $K$-equivariant  bundle $\Fcal\to M$ is weakly  $\Phi$-positive (resp. strictly $\Phi$-positive) if
$$
\frac{1}{i}\Lcal(\beta)\geq 0 \ \ (\mathrm{resp.}\ > 0)\quad \mathrm{on}\quad \Fcal\vert_\Xcal
$$
for any couple $(\Xcal,\beta)$ where $\Xcal\subset M^\beta$ is a connected component such that
$\Xcal\cap \Phi^{-1}(\beta)\neq\emptyset$, and $\beta\neq 0$.

\item A $K$-equivariant Clifford bundle $\Ecal\to M$ is weakly  $\Phi$-positive (resp. strictly $\Phi$-positive) if
$$
\frac{1}{i}\Lcal(\beta)\geq 0 \ \ (\mathrm{resp.}\ > 0)\quad \mathrm{on}\quad \mathbb{d}_\beta(\Ecal)\vert_\Xcal\otimes \det(\qgot_\JJbeta)^{-1}
$$
for any couple $(\Xcal,\beta)$ where $\Xcal\subset M^\beta$ is a connected component such that
$\Xcal\cap \Phi^{-1}(\beta)\neq\emptyset$ and $\beta\neq 0$.
\end{itemize}
\end{defi}

An obvious example of
weakly $\Phi$-positive bundle $\Fcal$ on $M$  is the trivial  bundle $M\times \C$.

If $\Ecal$ is a  weakly $\Phi$-positive Clifford bundle
and $\Fcal$  a weakly $\Phi$-positive bundle,
then the  Clifford module $\Ecal\otimes \Fcal$ is weakly $\Phi$-positive.

Another example arise when we have an equivariant map $\varphi: M' \to M$. Suppose that we start with a
weakly $\Phi$-positive bundle $\Fcal$ on $M$. We consider the pull-backs $\Fcal'=\varphi^*\Fcal$ and $\Phi'=\varphi^*\Phi$ on $M'$.
It is obvious that $\varphi(\{\kappa_{\Phi'}=0\})\subset\{\kappa_{\Phi}=0\}$. It is a small exercise to check the following lemma.

\begin{lem}\label{lem:varphi-pullback}
If $\varphi(\{\kappa_{\Phi'}=0\})=\{\kappa_{\Phi}=0\}$, the bundle $\Fcal'$ is weakly $\Phi'$-positive
\end{lem}

A particularly important example of strictly  $\Phi$-positive bundle is the following.
\begin{defi}
A $K$-equivariant line bundle $L$ over M is called a $\Phi$-moment line bundle if, for any $m\in M$ and any $X\in \kgot_m$,
  the action of $\Lcal(X)$  on the fiber $L\vert_m$ is equal to $i\langle\Phi(m),X\rangle{\rm Id}$.
\end{defi}

Such a line bundle  $L$ is strictly $\Phi$-positive,
since for any $m\in \Phi^{-1}(\beta)\cap M^{\beta}$, the action of $\Lcal(\beta)$ on $L\vert_m$ is equal to $i\|\beta\|^2{\rm Id}$.
In particular, let $\Phi_L$ be the moment map associated to an Hermitian connection (Definition \ref{defi:PhiL}).
Clearly $L$ is a $\Phi_L$-line bundle, and  the line bundle $L$ is strictly $\Phi_L$-positive.

\begin{exam}
Let us consider the case studied in Section \ref{subsec:cotangent} : the Hamiltonian action of $\tilde{K}\times K$ on $\T^*\tilde{K}$. The moment map is $\Phi(a,\xi)=-\xi\oplus \pi(a^{-1}\xi)$. In this situation, the trivial line bundle over $\T^*\tilde{K}$ is a $\Phi$-moment line bundle.
\end{exam}

The following theorem follows from Proposition
\ref{prop:QR-beta-general}.
\begin{theo}\label{theo:Phipos}
Let $M$ be a $K$-manifold equipped with a moment map $\Phi$ such that $Z_\Phi$ is compact. Let $\Ecal$ be an equivariant
graded Clifford bundle on $M$.

 $\bullet$ If $\Ecal$ is weakly $\Phi$-positive, then we have
\begin{eqnarray*}
\lefteqn{\left[\Qcal_K(M,\Ecal,\Phi)\right]^K=\left[\Qcal_K(M,\Ecal, \Phi^{-1}(0),\Phi)\right]^K}\\
& &+\sum_{\beta\in \Bcal, \beta\neq 0}\left[\Qcal_{K_{\beta}}(M^\beta,\Ecal_{[\beta]},M^{\beta}
\cap \Phi^{-1}(\beta),\Phi)\right]^{K_{\beta}}
\end{eqnarray*}
where the Clifford bundles $\Ecal_{[\beta]}$ are defined by (\ref{eq:E-beta-0}).

$\bullet$ When the Clifford bundle $\Ecal$ is strictly  $\Phi$-positive, we have
$$\left[\Qcal_K(M,\Ecal,\Phi)\right]^K=\left[\Qcal_K(M,\Ecal, \Phi^{-1}(0),\Phi)\right]^K.$$
If furthermore $0$ is a regular value of $\Phi$, then  $K$ acts locally freely on $\Phi^{-1}(0)$, and we have
$\left[\Qcal_K(M,\Ecal)\right]^K=\Qcal(M_{red},\Ecal_{red})$.
\end{theo}

 The equality
$\left[\Qcal_K(M,\Ecal, \Phi^{-1}(0),\Phi)\right]^K=\Qcal(M_{red},\Ecal_{red})$
is proved in Theorem \ref{theo:localisation-Phi-0}.

Finally let us remark the following.

\begin{prop}\label{prop:defPhi}
Let $\Phi$ be a moment map, and assume that $L$ is a  $\Phi$-moment line bundle.
Let $\Phi_t, t\in [0,1]$ be a smooth  family of moment maps such that $\Phi_0=\Phi$. We assume that there exists a
relatively compact open subset  $\Ucal\subset M$ such that $Z_{\Phi_t}\subset \Ucal$ for all $t\in [0,1]$.
Then there exists $\epsilon>0$ such that $L$ is weakly $\Phi_t$-positive for any $t\in [0,\epsilon]$.
\end{prop}

\begin{proof}
Take the notations of Lemma \ref{lem:PhiZPhi}.
Consider a finite set $\{\kgot_i\}$ of  representatives of the infinitesimal stabilizers of the action of $K$ on $\Ucal$,
and let $\Ccal_i$ be the set of connected components $\Xcal$ of $M_{\kgot_i}\cap\Ucal$.
The action of $\frac{1}{i}\beta(\Xcal,\Phi)\in \kgot_i$ on $L\vert_m$, for any $m\in \Xcal$,  is given by
$\|\beta(\Xcal,\Phi)\|^2{\rm Id}$
as $\beta(\Xcal,\Phi)\in \kgot_i$ and $\Phi(\Xcal)\subset \beta(\Xcal,\Phi)+\kgot_i^{\perp}$.

Let $m\in Z_{\Phi_t}\subset \Ucal$.
Then, up to conjugacy,  there exists an $i$ such that
$m\in \Phi_t^{-1}(\kgot_i)\cap M_{\kgot_i}\cap\Ucal$.
Let $\Xcal$ be the connected component of $M_{\kgot_i}\cap\Ucal$ containing $m$.
Then $\beta_t=\Phi_t(m)=\beta(\Xcal,\Phi_t)\in \kgot_i$ and is close  to the point
$\beta(\Xcal,\Phi)\in \kgot_i$.
Thus if  $\beta(\Xcal,\Phi)\neq 0$, and $t$ small enough,
the action of $\beta_t$ is positive on $L\vert_m$.
If $\beta(\Xcal,\Phi)=0$, then $\Phi(\Xcal)$ is contained in
$\kgot_i^{\perp}$, and so the action of $\beta_t\in \kgot_i$ on $L\vert_m$
is equal to  $i\langle \Phi(m),\beta_t\rangle =0$.
\end{proof}\bigskip

\subsection{$[Q,R]=0$ in the asymptotic sense}

Let $M$ be a compact even dimensional $K$-manifold equipped with an equivariant graded Clifford bundle $\Ecal$.
Let $L$ be a $K$-equivariant Hermitian line bundle with Hermitian connection $\nabla$, and consider its moment map $\Phi_L.$

\begin{theo}
When $k$ is sufficiently large, we have
\begin{equation}\label{eq:QR-Phi-0}
\left[\Qcal_K(M,\Ecal\otimes L^{\otimes k})\right]^K=\left[\Qcal_K(M,\Ecal\otimes L^{\otimes k}, \Phi_L^{-1}(0),\Phi_L)\right]^K.
\end{equation}
If furthermore $0$ is a regular value of $\Phi_L$, then  $K$ acts locally freely on $\Phi_L^{-1}(0)$, and we have
$\left[\Qcal_K(M,\Ecal\otimes L^{\otimes k})\right]^K=\Qcal(M_{red},\Ecal_{red}\otimes L^{\otimes k}_{red})$.
\end{theo}

This theorem follows right away from
Theorem \ref{theo:Phipos}. Indeed, when $\beta\neq 0$,
for any couple $(\Xcal,\beta)$ where $\Xcal\subset M^\beta$ is a connected component such that
$\Xcal\cap \Phi_L^{-1}(\beta)\neq\emptyset$, the action of
 $\beta/i$ on  $L^{\otimes k}|_{\Xcal}$ is given by
$k\|\beta\|^2$, so when $k$ is large, the  Clifford bundle $\Ecal\otimes L^{\otimes k}$ is strictly $\Phi_L$-positive.

So, although the map $\Phi_L$ varies with the choice of connection on $L$ and the topological
space $M_{red}$ may vary dramatically, then for $k$ large,
the quantity
$\Qcal(M_{red},\Ecal_{red}\otimes L^{\otimes k}_{red})$
is independent of the choice of the connection on $L$.

Theorem \ref{eq:QR-Phi-0} plays an essential role when one studies the asymptotic behaviour of branching law coefficients (see \cite{pep-stability}).

\section{$[Q,R]=0$ for almost complex  manifolds}\label{sec:almost-complex}

Let us assume that our $K$-manifold $M$ is provided with a $K$-invariant almost complex structure $J$.
Then, as in the complex case (see Example \ref{exam:complex}),
 a natural Clifford bundle on $M$ is $\bigwedge_J \T M$ and any Clifford module on
$M$ is a twisted Clifford bundle $\bigwedge_J \T M\otimes \Fcal$ where $\Fcal$ is a $K$-equivariant complex
vector bundle.
\begin{defi}
$\bullet$ When $M$ is compact, we define the Riemann-Roch character
$$
RR^J_K(M,\Fcal):=\Qcal_K(M,\bigwedge_J \T M\otimes \Fcal).
$$

$\bullet$ When $M$ is not necessarily compact, we can attach the localized
Riemann-Roch character
$$
RR^J_K(M,\Fcal,\Phi):=\Qcal_K(M,\bigwedge_J \T M\otimes \Fcal,\Phi)
$$
to any  equivariant map $\Phi: M\to\kgot^*$ such that $Z_\Phi$  is compact.

$\bullet$ When $Z$ is a compact component of $Z_\Phi$, we define the localized
Riemann-Roch character by
$$
RR^J_K(M,\Fcal,Z,\Phi):=\Qcal_K(M,\bigwedge_J \T M\otimes \Fcal,Z,\Phi).
$$
\end{defi}

An important example of almost complex manifold is the case of symplectic manifold $(M,\Omega)$.
In this case, we construct $J$ as follows:  choose a $K$-invariant Riemannian metric $g$ on $M$, and write
$\Omega_m(v,w)=g(v, A_m w)$ where $A_m$ is a invertible antisymmetric matrix, and choose $J_m=\frac{A_m}{\sqrt{-A_mA_m^*}}$.

Our aim is to compute $[RR^J_K(M,\Fcal,\Phi)]^K$ for a well chosen moment map $\Phi$.

\subsection{Statement of the results}

We start with a definition.
\begin{defi}
Consider the data $(\Omega,J)$ where $\Omega$ is a two form and $J$ is an almost complex structure on $M$. We say that
$(\Omega,J)$ is adapted if, for any $m\in M$,
$\Omega(Jv,Jw)=\Omega(v,w)$ and
the quadratic form $v\in \T_m M\mapsto \Omega(v,Jv)$ is semi-positive.
\end{defi}

We have the fundamental fact that we will prove in next subsection.

\begin{prop}\label{prop:J-Phi-positive}
If $(\Omega,\Phi)$ is a weakly Hamiltonian structure on $M$ such that $(\Omega,J)$ is adapted, then the Clifford
bundle $\bigwedge_J \T M$ is weakly $\Phi$-positive.
\end{prop}

For the rest of this section, we work with a weakly Hamiltonian structure $(\Omega,\Phi)$ on $M$ such that $(\Omega,J)$ is adapted.


Consider $\beta \in \kgot$.
For the tangent bundle, we have the decomposition
$\T M\vert_{M^\beta}=\T M^\beta\oplus \Ncal$, where $\T M^\beta=\ker(\Lcal(\beta))$ and $\Ncal=\mathrm{Image}(\Lcal(\beta))$.
Since the complex structure $J$ on $M$ commutes with $\Lcal(\beta)$, the vector sub-bundles $\T M^\beta$ and $\Ncal$
are stable under $J$. On $\Ncal$, as we have another complex structure $J_\beta$ that commutes with $J$, we get the decomposition
$\Ncal=\Ncal_{\JJbeta, +}\oplus\Ncal_{\JJbeta,-}$ where $\Ncal_{\JJbeta, \pm}=\{v\in \Ncal\,\vert\, J_\beta v=\pm Jv\}$.

Also this lemma will be proved in the next subsection.
\begin{lem}\label{lem:rank-N-beta}
For any non-zero element $\beta\in\kgot$, and any $m\in M^\beta\cap\Phi^{-1}(\beta)$, the real dimension of $\Ncal_{\JJbeta,+}|_m$ is
greater or equal than $\dim_\R(\kgot/\kgot_{\beta})$.
\end{lem}

Recall that the critical set $Z_\Phi$ is parameterized by a finite set $\Bcal$. Let us fix a non-zero element $\beta\in\Bcal$.

\begin{defi}\label{defi;Zmin}
For any non-zero element $\beta\in\Bcal$, define $M^{\beta}_o\subset M^\beta$ as the union
$\cup\Xcal$ of the connected components $\Xcal$ of $M^\beta$ intersecting $\Phi^{-1}(\beta)$ and
such that the real rank of the vector bundle
$\Ncal_{J_\beta,+}\vert_\Xcal$ is equal to $\dim_\R (\kgot/\kgot_{\beta})$.
\end{defi}

We notice that $M^{\beta}_o$ is a submanifold of $M^\beta$ stable under the action of the group $K_\beta$.

The main result of this subsection is the following theorem.
\begin{theo}\label{th:QR-Phi-J}
Let $\Fcal$ be a $K$-equivariant bundle on  an almost complex manifold $(M,J)$.
Let $(\Omega,\Phi)$ be a weakly Hamiltonian structure such that $(\Omega,J)$ is adapted and $Z_\Phi$ is compact.

$\bullet$ If the vector bundle $\Fcal$ is weakly $\Phi$-positive,  we have
\begin{eqnarray*}
\lefteqn{\Big[RR^J_K(M,\Fcal,\Phi)\Big]^K =
\Big[RR_K(M, \Fcal,\Phi^{-1}(0),\Phi)\Big]^K +}\\
& &  \sum_{\beta\in \Bcal\setminus\{0\}}\
\Big[RR_{K_{\beta}}(M^{\beta}_o,[\Fcal\vert_{M^{\beta}_o}]^{\beta},M^{\beta}_o
\cap \Phi^{-1}(\beta),\Phi)\Big]^{K_{\beta}}.
\end{eqnarray*}

$\bullet$ If the vector bundle $\Fcal$ is strictly $\Phi$-positive, then
$$
\Big[RR^J_K(M,\Fcal,\Phi)\Big]^K=
\Big[RR_K(M, \Fcal,\Phi^{-1}(0),\Phi)\Big]^K.$$

$\bullet$ If $0$ is a regular value of $\Phi$, we have
$$
\Big[RR_K(M, \Fcal,\Phi^{-1}(0),\Phi)\Big]^K =
\Qcal(M_{red},\Scal_{red}\otimes \Fcal_{red}).
$$
Here $\Fcal_{red}=\Fcal\vert_{\Phi^{-1}(0)}/K$ is a orbibundle over $M_{red}=\Phi^{-1}(0)/K$, and $\Scal_{red}$ is the
$\spinc$-bundle on $M_{red}$ induced by $\bigwedge_J \T M$.
\end{theo}

\medskip

Suppose now that $(\Omega,\Phi)$ is an Hamiltonian structure : the closed $2$-form $\Omega$ is non-degenerate. When $M$ is non-compact, we assume furthermore that $\Phi$ is proper and that the set $Z_\Phi$ is compact.
Theorem \ref{th:QR-Phi-J} can be improved as follows.

\begin{theo}\label{th:QR-J-hamiltonian}
Let $(M,J)$ be an almost complex manifold, and let $\Fcal$ be a $K$-equivariant vector bundle.
Let $(\Omega,\Phi)$ be a Hamiltonian structure on $M$ such that :
$(\Omega,J)$ is adapted, the moment map $\Phi$ is proper and $Z_\Phi$ is compact.

If the vector bundle $\Fcal$ is weakly $\Phi$-positive, {\bf and} $0$ is in the image of the moment map, then:
$$
\Big[RR^J_K(M,\Fcal,\Phi)\Big]^K=[RR_K^J(M, \Fcal,\Phi^{-1}(0),\Phi)\Big]^K.
$$
\end{theo}

Our manifold $M$ is not necessarily connected, but, considering the action of $K$ on connected components of $M$,
 our study reduces easily to the case where $M=K M^0$, where $M^0$ is connected.
Then $M^0$ is connected and stable by the connected component $K^0$ of $K$. We now assume that we are in this situation.
In  the Hamiltonian setting,  we  then have the following lemma.

\begin{lem}\label{lem:uniquebeta}
The set of points in $\Phi(M)$  at minimum distance of the origin is a  $K$-orbit $K{\beta_{\rm min}}$.
 Furthermore $\Phi^{-1}(K{\beta_{\rm min}})\subset M^{{\beta_{\rm min}}}$.
 \end{lem}

 We denote by $M'=M^{{\beta_{\rm min}}}$, $K'=K_{{\beta_{\rm min}}}$,
  $\Phi'=\Phi|_{M'}-{\beta_{\rm min}}$ , $J'$ the restriction of $J$ to $M'$, and $\Fcal'$ the restriction of
 $\Fcal$ to $M'$.
 We have then the following theorem, which allows us to always reduce the study of
$\Big[RR^J_K(M,\Fcal,\Phi)\Big]^K$when $\Fcal$  is weakly $\Phi$-positive, to the
study of a  Hamiltonian manifold with $0$ in the image of its moment map.

\begin{theo}\label{th:QR-J-hamiltonianbis}
Let $(M,J)$ be an almost complex manifold, and let $\Fcal$ be a $K$-equivariant vector bundle.
Let $(\Omega,\Phi)$ be a Hamiltonian structure on $M$ such that :
$(\Omega,J)$ is adapted, the moment map $\Phi$ is proper and $Z_\Phi$ is compact.
If the vector bundle $\Fcal$ is weakly $\Phi$-positive, then $\Fcal'$ is weakly $\Phi'$-positive, $Z_\Phi'$ is compact  and
$$
\Big[RR^J_K(M,\Fcal,\Phi)\Big]^K= \Big[RR^{J'}_{K'}(M',\Fcal',\Phi')\Big]^{K'}.$$
\end{theo}

\medskip

Let us summarize what we have proved in this section for a line bundle $L\to M$ on an almost complex manifold $(M,J)$. Let
$(\Omega,\Phi)$ be a weak Hamiltonian structure on $M$ such that $(\Omega, J)$ is adapted and $Z_\Phi$ is compact.
Then the relation  $[RR^J_K(M,L,\Phi)]^K=[RR^J_K(M,L,\Phi^{-1}(0),\Phi)]^K$ holds in two  cases:
\begin{enumerate}
\item $L$ is a $\Phi$-moment line bundle,
\item $L$ is weakly $\Phi$-positive, $\Omega$ is non-degenerate and $\Phi^{-1}(0)\neq \emptyset$.
\end{enumerate}

\subsection{Proofs}

We now concentrate  on the proofs of  Lemma \ref{lem:rank-N-beta}, Proposition \ref{prop:J-Phi-positive} and Theorems \ref{th:QR-Phi-J} $-$ \ref{th:QR-J-hamiltonianbis}.

\medskip

Let $\beta\in\kgot\setminus\{0\}$. If $m\in M^\beta\cap \Phi^{-1}(\beta)$, we have $\kgot_m\subset \kgot_\beta$ and then
the map
$X\mapsto X\cdot m$, $\qgot_\beta\to \Ncal\vert_m\subset \T_m M$, is injective, $\R$-linear  and
equivariant relatively to the action of $\Lcal(\beta)$. We consider the map
$$
\mathrm{p}_m:\qgot_\beta\to \Ncal\vert_m^{\beta,+}
$$
which is the composition of the injective map $\qgot_\beta\croc \Ncal\vert_m$ with the projection
$\Ncal\vert_m\to \Ncal\vert_m^{\beta,+}$. The map $\mathrm{p}_m$ is complex linear if both vector spaces
$\qgot_\beta$ and $\Ncal\vert_m$ are endowed with the
complex structure $J_\beta$.

Lemma \ref{lem:rank-N-beta}  is a consequence of
\begin{lem}\label{lem:rankbigger}
The map $\mathrm{p}_m$ is injective for any $m\in \Phi^{-1}(\beta)\cap M^{\beta}$.
\end{lem}

\begin{proof}
Consider the symmetric $\R$-bilinear form
\begin{equation}\label{eq:Hessian-beta}
H_m(v,w)=\Omega_m(v,\Lcal(\beta)w)
\end{equation}
on $\Ncal\vert_m$. By our condition on $\Omega$,  $H_m(v,v)\leq 0$  on $\Ncal_{\JJbeta,-}\vert_m$.
In contrast, on the subspace $\qgot\cdot m\subset \Ncal\vert_m$,  the quadratic form
$H_m$ is  positive definite. Indeed
$H_m(X\cdot m,X\cdot m)= (\Phi(m),[X,[\Phi(m),X]])=\|[\beta, X]\|^2$.
Thus $\qgot\cdot m$ does not intersect $\Ncal_{\JJbeta,-}\vert_m$, and then $\mathrm{p}_m$ is injective.
\end{proof}

\begin{rem}\label{rem:H-definite}
The proof above shows that if $\mathrm{p}_m$ is {\em bijective}, then $\qgot\cdot m$ is a maximal subspace of
$\Ncal\vert_m$ where the quadratic form associate to $H_m$ is definite positive.
\end{rem}

\medskip

The injectivity of the complex linear map $\mathrm{p}_m$ implies that  $\frac{1}{i}\Lcal(\beta)\geq 0$ on $\det(\Ncal_{\JJbeta,+}\vert_m)\otimes \det(\qgot_\JJbeta)^{-1}$,
and for  a connected component $\Xcal$ of $M^{\beta}$ intersecting $\Phi^{-1}(\beta)$ we have
\begin{equation}\label{eq:L-beta-0}
\Xcal\subset M^\beta_o\hspace{5mm}\Longleftrightarrow\hspace{5mm}\Lcal(\beta)= 0\ \ \mathrm{on}\ \
\det(\Ncal_{\JJbeta,+}\vert_\Xcal) \otimes\det(\qgot_\JJbeta)^{-1}.
\end{equation}

Moreover, when $m$ is on a neighborhood $\Phi^{-1}(\beta)\cap M^\beta_o$ of $\Phi^{-1}(\beta)\cap M^{\beta}_o$,
 the map $\mathrm{p}_m$ stay injective and  realize a trivialization of the bundle
$\det(\Ncal_{\JJbeta, +})\otimes \det(\qgot_\JJbeta)^{-1}$.

Let us  determine the Clifford bundle $\mathbb{d}_\beta(\Ecal)$ associated to the $\spinc$-bundle $\Ecal=\bigwedge_J\T M$ on $M$.
By definition, we have $\Ecal\vert_{M^\beta}=\bigwedge\overline{\Ncal_\JJbeta}\otimes \mathbb{d}_\beta(\Ecal)$.

Since $\Ncal_\JJbeta\simeq \Ncal_{\JJbeta, +}\oplus\overline{\Ncal_{\JJbeta,-}}$ as complex vector bundles, we have
\begin{eqnarray*}
\bigwedge\overline{\Ncal_\JJbeta}&\simeq& \bigwedge\overline{\Ncal_{\JJbeta, +}}\otimes\bigwedge\Ncal_{\JJbeta,-}\\
&\simeq & (-1)^{\rang\Ncal_{\JJbeta, +}} \det(\Ncal_{\JJbeta, +})^{-1}\otimes
\bigwedge \Ncal_{\JJbeta, +}\otimes\bigwedge\Ncal_{\JJbeta,-}\\
&\simeq & (-1)^{\rang\Ncal_{\JJbeta, +}} \det(\Ncal_{\JJbeta, +})^{-1}\otimes \bigwedge_J \Ncal\\
\end{eqnarray*}
as graded $\spinc$-bundles on $\Ncal\to M^\beta$. On the other hand we have also
$\bigwedge_J \T M\vert_{M^\beta}\simeq \bigwedge_J \T M^\beta\otimes \bigwedge_J \Ncal$. This shows that
\begin{equation}\label{eq:E-beta-M-beta}
\mathbb{d}_\beta(\Ecal):=(-1)^{\rang\Ncal_{\JJbeta, +}} \bigwedge_J\T M^\beta\otimes \det(\Ncal_{\JJbeta, +}).
\end{equation}
At this stage we have proved Proposition \ref{prop:J-Phi-positive}: The Clifford bundle $\Ecal=\bigwedge_J\T M$ is weakly $\Phi$-positive
since for any   connected component $\Xcal$ of $M^\beta$ intersecting $\Phi^{-1}(\beta)$ we have $\frac{1}{i}\Lcal(\beta)\geq 0$ on
$\mathbb{d}_\beta(\Ecal)\vert_\Xcal\otimes \det(\qgot_\JJbeta)^{-1}$ as
$\frac{1}{i}\Lcal(\beta)\geq 0$ on $\det(\Ncal_{\JJbeta, +})\otimes \det(\qgot_\JJbeta)^{-1}$ restricted to $\Xcal$.

Moreover we compute that
\begin{eqnarray*}
\Ecal_{[\beta]}&:=&(-1)^{\dim \qgot_\JJbeta} [\mathbb{d}_\beta(\Ecal)\otimes \det(\qgot_\JJbeta)^{-1}]^\beta\\
&=& (-1)^{n_\beta} \mathbb{D}_\beta\otimes\bigwedge_J\T M^\beta
\end{eqnarray*}
where $n_\beta=\dim_\C \qgot_\JJbeta -\rang_\C\Ncal_{\JJbeta, +}$ and
$$
\mathbb{D}_\beta=\left[\det(\Ncal_{\JJbeta, +})\otimes\det(\qgot_\JJbeta)^{-1}\right]^\beta.
$$
Two cases hold for a connected component $\Xcal$ of $M^\beta$ intersecting $\Phi^{-1}(\beta)$. Thanks to the discussion
above (see (\ref{eq:L-beta-0})), we see that either $\Xcal\nsubseteq M_o^\beta$ and then $\mathbb{D}_\beta\vert_\Xcal=[0]$,
or $\Xcal\subset M_o^\beta$ and then $n_\beta=0$ on $\Xcal$ and
$\mathbb{D}_\beta\vert_\Ucal\simeq \Ucal\times\C$ on a neighborhood $\Ucal$ of $\Xcal\cap\Phi^{-1}(\beta)$ in $\Xcal$.

Finally we see that Theorem \ref{th:QR-Phi-J} is a consequence of Theorem \ref{theo:Phipos}. If we work with the Clifford bundle
$\Ecal'=\bigwedge_J \T M\otimes \Fcal$, we have $\Ecal'_{[\beta]}\vert_\Xcal=[0]$ unless $\Xcal$ is contained in $M^{\beta}_o$,
and we have $\Ecal'_{[\beta]}\vert_{M^{\beta}_o}\simeq \bigwedge_J \T M^{\beta}_o\otimes [\Fcal\vert_{M^{\beta}_o}]^\beta$
in  a neighborhood of
$M^{\beta}_o\cap\Phi^{-1}(\beta)$.
So we can conclude.

\bigskip

We now concentrate ourselves  on the proof of Theorem \ref{th:QR-J-hamiltonian} and
Theorem \ref{th:QR-J-hamiltonianbis}.
First we prove the preliminary lemma, Lemma \ref{lem:uniquebeta}.
Let $K^0$ be the connected component of $K$. We assume $M=KM^0$ where $M^0$ is connected and stable by $K^0$.
Let $T$ be the maximal torus of $K^0$ with Lie algebra $\tgot$. We make the choice of a Weyl chamber
$\tgot^*_{\geq 0}$. The Convexity Theorem \cite{Atiyah82,Guillemin-Sternberg82.bis,Kirwan.84.bis,L-M-T-W} is telling us that the set $\Delta(M)=\Phi(M)\cap \tgot^*_{\geq 0}$
is a convex polytope, called the Kirwan polytope. So
there is a unique point $\beta_{\mathrm{min}}\in\Delta(M)$ at minimum distance of the origin.
So $K\beta_{\mathrm{min}}$ is the unique orbit at minimal distance in $\Phi(M)$.

Let ${\beta_{\rm min}}\in \Phi(M)$ at minimum distance of the origin.
Equation $d\|\Phi\|^2=2\iota(\kappa_\Phi)\Omega$ shows that $\Phi^{-1}({\beta_{\rm min}})$ is
contained in $M^{\beta_{\rm min}}=M'$. Let $K'$ be $K_{{\beta_{\rm min}}}$.
We notice that $M'$ is a $K'$-Hamiltonian manifold with moment map $\Phi_{M'}$, and $J$ induces an invariant
almost complex structure $J'$ on $M'$. We have the following basic properties whose proof is left to the reader.
\begin{lem}
$\bullet$  The data $(\Omega',J')$ is adapted.
$\bullet$  The map $\Phi_{M'}$ is proper and $Z_{\Phi_{M'}}$ is compact.
$\bullet$  If $\Fcal$ is  weakly $\Phi$-positive, then $\Fcal_{M'}$ is weakly $\Phi_{M'}$-positive.
\end{lem}

Let us go back to the proofs of
Theorem \ref{th:QR-J-hamiltonian} and
Theorem \ref{th:QR-J-hamiltonianbis}.

 Using Theorem \ref{th:QR-Phi-J},  and Lemma \ref{lem:uniquebeta} we
are left to prove the following

\begin{lem} Let $\beta\in\Bcal\setminus \{0\}$. The submanifold $M^{\beta}_o$ is  non-empty if and only if
$\beta=\beta_{\mathrm{min}}$.
\end{lem}

Thus, let $\beta \in \Bcal$ and consider
$m\in \Phi^{-1}(\beta)\cap M^{\beta}$ such that
the real rank of  $\Ncal_{\JJbeta,+}$ is equal to $\dim_\R \qgot$.
We want to prove that $\beta=\beta_{\mathrm{min}}$.
For that, we can only consider the action of the connected group $K^0$
on a connected component $M^0$ of $M$. Changing locally of notations for the aim of this proof,
we now assume $K$ and  $M$ connected.
Let us use the notations of the proof of Lemma \ref{lem:rankbigger}. As $m\in M^{\beta}$, $m$ is a critical point of
the map $(\Phi,\beta): M\to \R$. Its Hessian $H_m$ at $m$ is the symmetric bilinear form defined in (\ref{eq:Hessian-beta}).
The fact that the $2$-form $\Omega$ is symplectic implies that $H_m$ is non-degenerate on $\Ncal\vert_m$.

Let $U$ be the connected component of the set of elements $\xi$ in $\kgot_\beta$ such that $\kgot_\xi\subset \kgot_\beta$.
This is an open set in $\kgot_\beta$ containing $\beta$. The subset $Y=\Phi^{-1}(U)$ is connected, and is a
$K_{\beta}$-symplectic  submanifold of $M$ such that $KY\subset M$ is a dense open subset of $M$.

Let $\Ncal^Y$ be the normal bundle of $Y^\beta$ in $Y$. At the level of tangent spaces we have the decompositions
\begin{eqnarray*}
\T_m M &= &\T_m M^\beta\oplus \Ncal\vert_m\\
&=& \T_m Y^\beta\oplus\qgot\cdot m\oplus \Ncal^Y\vert_m
\end{eqnarray*}
that are orthogonal with respect to $\Omega_m$ (and also under $H_m$).

Now,  if the map $\mathrm{p}_m$ is an isomorphism,  $\qgot\cdot m$ is a maximal subspace of $\Ncal\vert_m$
where the quadratic form associate to $H_m$ is definite positive (see Remark \ref{rem:H-definite}). As $H_m$ is non-degenerate on $\Ncal^Y\vert_m$, necessarily $H_m(v,v) < 0$ for any non-zero element $v\in\Ncal^Y\vert_m$.
 So,  we have $(\Phi,\beta)\geq \|\beta\|^2$ on a neighborhood of $m$ in $Y$. In the symplectic setting, any function
 of the type $(\Phi,\beta)$ has a unique local minimum, thus $(\Phi,\beta)\geq \|\beta\|^2$ on $Y$, and  consequently
 $\|\Phi\|^2\geq\|\beta\|^2$ on $Y$.  By $K$-invariance it implies that $\|\Phi\|^2\geq \|\beta\|^2$ on $M$.
 So $\beta=\beta_{\mathrm{min}}$ as desired.

\subsection{$[Q,R]=0$ and semi-ample line bundles}

We define the notion of semi-ample line bundles in the almost complex setting.

\begin{defi}\label{defi: semi-ample}
Let $L$ be a $K$-equivariant line bundle on  an almost complex manifold $(M,J)$. We say that $L$ is {\bf semi-ample}, if there exists
 a $K$-invariant Hermitian connection with curvature two form $R_L=i\Omega_L$ such that  $\Omega_L(v,Jv)\geq 0$ for $v\in \T M$.
 The corresponding moment map $\Phi_L$ will be called adapted.
\end{defi}

The following corollary is clear.
\begin{coro}\label{coro-QR-symplectic-strictly-positive}
Let $L$ be a $K$-equivariant semi-ample line bundle on a compact almost complex manifold $M$,
with adapted moment map $\Phi_L$. Then
$$
\Big[RR^J_K(M,L^{\otimes k})\Big]^K=[RR_K^J(M, L^{\otimes k},\Phi_L^{-1}(0),\Phi_L)\Big]^K
$$
for any $k\geq 1$. In particular, $\Big[RR^J_K(M,L^{\otimes k})\Big]^K\neq 0$ for some $k\geq 1$ only if $0\in\Phi_L(M).$

\end{coro}

This theorem is an analogue of the GIT construction.
We will obtain more precise results  on the description of  $[RR_K(M,\Fcal)]^K$ via a moment map $\Phi$
 if we can  choose  $(\Omega,\Phi)$ such that $\Omega$  is non degenerate,
and $\Fcal$ weakly $\Phi$-positive. This is the object of Section \ref{subsec:multiRiemannRoch}.

\begin{exam}\label{exam:semi-ample} We follow the notations of Example \ref{exam:highest-weight}. Let $\mathbb{F}:=K/T$ be the flag manifold equipped with an integrable complex structure. The line bundle
$[\C_\mu]:=K\times_T\C_{\mu}$ associated to a weight $\mu\in \Lambda$ is semi-ample if and only if $\mu$ is dominant, and $[\C_\mu]$
is ample if and only if $\mu$ is regular dominant.
\end{exam}

\begin{exam}\label{exam:line-cotangent-bundle} The cotangent bundle $\T^*\tilde{K}$ of a compact Lie group $\tilde{K}$
is equipped with the symplectic form $\Omega=d\lambda$ where $\lambda$ is the Liouville $1$-form (see section \ref{subsec:cotangent}).
Let $J$ be a almost complex structure on $\T^*\tilde K$ that is compatible with $\Omega$.  Let us consider the trivial line bundle
$[\C]=\T^*\tilde{K}\times\C$ with the connection $\nabla:= d+i\lambda$. We see then that the two form
$\Omega_{[\C]}:=\frac{1}{i}\nabla^2$ is equal to $\Omega$.
\end{exam}

\section{A slice theorem  for deformed symbol}\label{sec:slice-induction}

\subsection{An induction formula}\label{sec:induction-cotangent}
Let $H$ be a closed subgroup of $K$, and consider a $H$-invariant decomposition
$\kgot=\hgot\oplus \qgot$. Let $N'$ be a $H$-manifold, and
consider the $K$-manifold
$$
N=K\times_H (B_{\qgot^*} \times N')
$$
where $B_{\qgot^*}\subset \qgot^*$ is a $H$-invariant ball centered at $0$. Then $N'$ is a submanifold of $N$, and the normal bundle of $N'$ in $N$ is
isomorphic to the trivial bundle with fiber $\qgot\oplus \qgot^*$ : for $n'\in N'$, the element $X\in\qgot$ defines the vector
$\frac{d}{dt}\vert_{t=0} e^{tX}\cdot [1,0,n']\in \T_{[1,n',0]}N$.

 The vector space $\qgot\oplus \qgot^*$ is  identified to $\qgot_\C$ by $X\oplus \xi\to X\oplus i \tilde{\xi}$~: here $\xi\in \qgot^*\to\tilde{\xi}\in\qgot$ is the isomorphism associated to an invariant scalar product on $\kgot$. Let $S_\qgot=\bigwedge \qgot_\C$ be the spinor space for $\qgot_\C$.
Thus if $\Ecal$ is a $K$-equivariant graded Clifford bundle on $N$, there exists a $H$-equivariant graded Clifford bundle $\Ecal'$ on $N'$ such that
$$\Ecal\vert_{N'}=S_\qgot \otimes \Ecal'.$$

Let $\Phi': N'\to \hgot^*$ be a $H$-equivariant map, and let $\Phi: N\to \kgot^*$ be a $K$-equivariant map. We assume that these maps are linked
by the following relations  :
\begin{equation}\label{eq:phi-positif}
\begin{cases}
\ \       \Phi\vert_{N'}=\Phi',     \\
 \ \    \Phi([1;\xi,n'])\in \hgot^*\Longleftrightarrow \xi=0,\\
 \ \    (\Phi([1;\xi,n']),\xi)\geq 0,
\end{cases}
\end{equation}
for $(\xi,n')\in B_{\qgot^*} \times N'$.

The following induction result will be useful.

\begin{prop}\label{prop:induction-cotangent}
$\bullet$ The critical sets $Z_\Phi\subset N$ and $Z_{\Phi'}\subset N'$ are related by
$Z_\Phi=K\times_H (\{0\}\times Z_{\Phi'})$.

$\bullet$ Let $Z$ be a compact component of $Z_\Phi$ and $Z'$ its intersection with $N'$. Then $Z'$ is a compact component of $Z_{\Phi'}$ and
$$\Qcal_K(N,\Ecal,Z,\Phi)={\rm Ind}_H^K \left(\Qcal_H(N',\Ecal',Z',\Phi')\right).$$
This leads to the relation $\left[\Qcal_K(N,\Ecal,Z, \Phi)\right]^K=\left[\Qcal_H(N',\Ecal', Z',\Phi')\right]^H$.
\end{prop}

\begin{proof}
Let $(\xi,n')\in B_{\qgot^*}\times N'$. We identify $\T_{[1;\xi,n']}N$ to $\qgot\oplus \qgot^* \oplus \T_{n'} N'$.
Let us write $\Phi([1;\xi,n'])=\Phi_{\qgot}(\xi,n') \oplus \Phi_\hgot(\xi,n')$ in the decomposition $\kgot=\hgot\oplus \qgot$ (and we have identified $\kgot^*$ and $\kgot$).
We have
 \begin{equation}\label{kappa'}
 \kappa_\Phi ([1;\xi,n'])= -\Phi_{\qgot}(\xi,n')\oplus [\xi,\Phi_\hgot(\xi,n')]\oplus -\Phi_\hgot(\xi,n')\cdot n'.
\end{equation}
From (\ref{eq:phi-positif}), we have that $\Phi_{\qgot}(\xi,n')=0$ only if $\xi=0$. Hence the vector field $\kappa_{\Phi}$ vanishes at
$[1;X,n']$ if and only if $\xi=0$ and $\kappa_{\Phi'}(n')=-\Phi_\hgot(0,n')\cdot n'=0$. So if $Z$ a component of $Z_\Phi$, its intersection
$Z'$ with $\qgot\times N'$ is a component of $Z_{\Phi'}$ such that $Z=K\times_H (\{0\}\times Z')$.

\medskip

Consider a compact component $Z$ of $Z_\Phi$. Let $\Ucal'$ be a relatively compact neighborhood of $Z'$ in $N'$, so that $\Vcal=B_\qgot \times \Ucal'$ is a relatively compact neighborhood of $Z'$ in $\qgot\times M'$. Let $\tau$ be the restriction of $\sigma(M,\Ecal,\Phi)$ to $\T^* \Vcal$.
We are working on $K\times_H \Vcal$, a small neighborhood of $Z=KZ'$ in $M$, and
by the induction formula of Proposition \ref{Induction formula},  we have
$$
\Qcal_K(M,\Ecal,Z,\Phi)= \ind_H^K\left(\indice_H^\Vcal(\tau)\right).
$$
Let us consider the injection $i:\Ucal'\to \Vcal$. The symbol $\sigma'=\sigma(\Ucal',\Ecal',\Phi')$ is a $H$-equivariant transversally elliptic symbol on $\Ucal'$. Its index is by definition $\Qcal_H(N',\Ecal',Z',\Phi')$.
We can construct the push-forward $i_!(\sigma')$ of the symbol
$\sigma'\in \Ko_H(\T^*_H \Ucal')$ in $\Ko_H(\T^*_H\Vcal)$ as defined in Section \ref{sec:directimage}.
Our proposition \ref{prop:induction-cotangent} will follow from the following formula and Theorem \ref{theo:direct}.
\begin{lem}
$$\tau=i_!(\sigma').$$
\end{lem}
\begin{proof}
By definition the symbol $i_!(\sigma')$ is equal to $\sigma'\boxtimes \mathrm{Bott}(\qgot_\C)$. For $(\xi,n')\in \qgot^*\times N'$,
and  $(q,\nu)\in \qgot\times \T^*_{n'} N'$, we have
$$
i_!(\sigma')_{(\xi,n')}(q,\nu)=\clif_{\Ecal'_{n'}}(\tilde{\nu}-\kappa_{\Phi'}(n'))\otimes \clif_{S_\qgot}(\tilde{\xi}\oplus i q).
$$

We compute $\tau_{(X,n')}(q,\nu)$. Using Formula \ref{kappa'} we have
$$
\tau_{(\xi,n')}(q,\nu)= \clif_{\Ecal'_{n'}}(\tilde{\nu}+\Phi_\hgot(\xi,n')\cdot n')\boxtimes \clif_{S_\qgot}\Big(\Phi_{\qgot}(\xi,n')\oplus
i(q-[\xi,\Phi_\hgot(\xi,n')])\Big).
$$
For any $t\in [0,1]$, we define
\begin{eqnarray*}
 A^t(\xi,n',q)&=& t\Phi_{\qgot}(\xi,n') +(1-t) \xi\oplus
i(q- t[\xi,\Phi_\hgot(\xi,n')])\ \ \in \ \ \qgot_\C\\
B^t(\xi,n')&=& -t \Phi_\hgot(\xi,n')\cdot n'+ (1-t)\kappa_{\Phi'}(n')\ \ \in \ \ \T_{n'} N',
\end{eqnarray*}
and the symbol
$$
\tau^t_{(\xi,n')}(q,\nu)= \clif_{\Ecal'_{n'}}(\tilde{\nu}-B^t(\xi,n'))\boxtimes \clif_{S_\qgot}\left(A^t(\xi,n',q)\right).
$$

We have $\|A^t(\xi,n',q)\|^2=\|t\Phi_{\qgot}(\xi,n') +(1-t) \xi\|^2+
\|\tilde{q}- t[\xi,\Phi_\hgot(\xi,n')]\|^2$.

If $((\xi,n'),(q,\nu))$ is in the support of $\tau^t$, necessarily  $A^t(\xi,n',q)=0$.
We use the conditions (\ref{eq:phi-positif}).
As $(\xi,\Phi_{\qgot}(\xi,n'))= (\Phi([1;\xi,n']),\xi)\geq 0$, we see that $\|t\Phi_{\qgot}(\xi,n') +(1-t) \xi\|^2\geq  t^2 \|\xi\|^2+(1-t)^2\|\Phi_{\qgot}(\xi,n')\|^2$
and this last expression vanishes only when $\xi=0$.
Thus, we see that, when $(\xi,n')\in B_\qgot\times N'$, the vector  $A^t(\xi,n',q)$ vanishes if and
only if $\xi=0$ and $q=0$. And for any $t\in[0,1], B^t(0,n')=\kappa_{\Phi'}(n')$.

Thus $\tau^t$ is an homotopy of transversally elliptic symbol on $B_\qgot\times N'$.
Hence $\tau^1=\tau$ defines the same $\K$-theory class than $\tau^0= i_!(\sigma')$.

\end{proof}\bigskip
\end{proof}

\medskip

In the next subsection, we give a first  application of this formula.

\subsection{The shifting formula}\label{sec:induction-2}

We consider a $K$-manifold $M$, oriented and of even dimension, that is equipped with a $K$-equivariant Clifford bundle $\Gcal$.
Let $\Phi_M: M\to \kgot^*$ be an equivariant map.
Let $a\in \kgot^*$, and
let $H=K_a$ be the stabilizer subgroup of $a$.
On the coadjoint orbit $Ka$, consider the inclusion $\Phi_a: Ka\croc \kgot^*$.
 We assume that $Ka$ is equipped with a $K$-equivariant Clifford bundle $\Fcal$.
On the product $M\times Ka$, consider the map $\Phi=\Phi_M-\Phi_a$.

In this section we assume that $\Phi^{-1}(0)$ is a compact component of $Z_\Phi$ : this property holds when $\Phi_M$ is a
moment map and $M$ is compact.

Consider
$$
\Qcal_K(M\times Ka,\Gcal\otimes\Fcal,\Phi^{-1}(0),\Phi)
$$
 the equivariant index localized near the component  $\Phi^{-1}(0)\subset M\times Ka$.

We look at a small $H$-invariant open neighborhood $B$
of  $a$ in $\hgot^*$ such that $Y=\Phi_M^{-1}(B)\subset M$ is a  slice at $a\in \kgot^*$ for the map $\Phi_M$.
Then  $Y$ is a $H$-invariant manifold of $M$ such that $\Phi_M(Y)\subset \hgot^*$ and such that $K\times_{H} Y$
is  diffeomorphic to a  invariant open neighborhood of $\Phi_M^{-1}(Ka)$. We denote by  $\Phi_Y$ the restriction of the map $\Phi_M$ to $Y$.

The element $a\in \kgot$ acts by an invertible skew-symmetric transformation on $\kgot/\kgot_a$. We denote
by $\qgot_\JJa$  the vector space $\kgot/\kgot_a$ equipped with the complex structure $ J_a$,
 We write $\Gcal|_Y= \Gcal_{\ddY}\otimes \bigwedge \qgot_\JJa$ where $\Gcal_{\ddY}$ is a $H$-equivariant Clifford bundle on $Y$, and $\Fcal\vert_{\{a\}}=\bigwedge \overline{\qgot_\JJa} \otimes F$ where $F\in R(H)$ (here $\{a\}$ is the base point of $Ka$).

\begin{prop}\label{prop:shifting trick} For any equivariant Clifford bundles $\Gcal\to M$ and
$\Fcal \to Ka$, and any equivariant map $\Phi(k,\xi)=\Phi_M(m)-\xi$ such that $\Phi^{-1}(0)$ is a compact component of $Z_\Phi$,
we have
$$
\Qcal_K(M\times Ka,\Gcal \otimes \Fcal, \Phi^{-1}(0),\Phi)
=\mathrm{Ind}_H^K \big(\Qcal_H(Y,\Gcal_{\ddY},\Phi_Y^{-1}(a),(\Phi_Y-a))\otimes F\big).
$$
\end{prop}

\begin{proof}
We show that we are in the  situation of Subsection \ref{sec:induction-cotangent}.
A neighborhood of $\Phi^{-1}(Ka)$ in $M$ is of the form $K\times_H Y$. Let us consider the
diffeomorphism $\varphi: K\times_H (Ka\times Y) \to  Ka\times (K\times_H Y)$
defined by $\varphi([k;\xi, y])= (k\xi,ky)$.
Through $\varphi$, the map $\Phi$ becomes on $K\times_H (Ka\times Y)$ the map
$\tilde{\Phi}([k;\xi, y])=k\left(\Phi_Y(y)-\xi\right)$, and
the set $\tilde{\Phi}^{-1}(0)$ is equal to $K\times (\{a\}\times \phi^{-1}(a))$.

We are then working on a $H$-invariant neighborhood of $\{a\}$ in $Ka$. This neighborhood can indeed be parameterized by
$\qgot^*$ through the map $\xi\to e^{\tilde{\xi}} a$. So we work with the $K$-manifold $N= K\times_H(B_{\qgot^*}\times Y)$ and the $H$-manifold $N'=Y$.
The equivariant maps are $\Phi([k;\xi,y]):=k( \Phi_Y(y) -e^{\tilde{\xi}} a)$ on $N$ and $\Phi'=\Phi_Y-a$ on $Y$.
We check easily that relations (\ref{eq:phi-positif}) hold if the ball $B_{\qgot^*}$ is small enough. Indeed, $(e^{\tilde{\xi}} a,\tilde{\xi})=0$ for any $\xi\in\qgot^*$, and
for small $\xi$ the term $e^{\tilde{\xi}} a=a+[\tilde{\xi},a]+ O(\|\tilde{\xi}\|^2)$ belongs to $\hgot$ only if $\xi=0$.

At the level of Clifford bundles, if one considers $\Ecal:=\Gcal\otimes\Fcal$, the corresponding $\Ecal'$ on $Y$ is
$\Ecal':=\Gcal_{\ddY}\otimes F$ since the Clifford bundle $S_{\qgot_a}$ is equal to $\bigwedge \qgot_\JJa\otimes \bigwedge \overline{\qgot_\JJa}$.
So we are in the setting of Proposition \ref{prop:induction-cotangent} and we can conclude.
\end{proof}

\section{The Hamiltonian setting }\label{sec:hamiltonian}

Let us assume that our $K$-manifold $M$ is provided with a Hamiltonian structure $(\Phi,\Omega)$ :
$\Omega$ is a non degenerate closed $2$-form on $M$ and the moment map $\Phi:M\to \kgot^*$ and
$\Omega$ are related by the Kostant relations (\ref{eq:hamiltonian-action}). Then we can always choose an
invariant almost complex structure $J$ on $M$ such that  $\Omega(v,Jw), v,w\in \T M$ is a Riemannian metric
on $M$ : such almost complex structure $J$ is called compatible with $\Omega$.

Consider the case where $M$ compact. In this context, the representation
$\Qcal_K(M,\bigwedge_J \T M\otimes \Fcal)$ does not depend on the choice of $J$ since the set of
$\Omega$-compatible almost complex structures is contractible, so, given $\Omega$, we denote it
by $RR_K(M,\Fcal)$ whenever the almost complex structure $J$ is compatible with the given choice of $\Omega$.

We can make use of the result of Theorem \ref{th:QR-J-hamiltonian}, since our hypothesis ``$J$ compatible with $\Omega$'' is
stronger than the hypothesis ``$(\Omega, J)$ is adapted''. In fact we will gain much more, as we will be able to compute
geometrically the multiplicities of $RR_K(M,L)$ when $L$ is a $\Phi$-moment bundle~: it is the heart of the $[Q,R]=0$ theorem
of Meinrenken-Sjamaar that is explained in the next section.

Our proof of the Meinrenken-Sjamaar Theorem relies on the Witten deformation argument presented in
the previous section. We attach a localized Riemann-Roch character
$$
RR_K(M,\Fcal,Z_\beta,\Phi)=\Qcal_K(M,\bigwedge_J \T M\otimes \Fcal, Z_\beta,\Phi)
$$
to each component $Z_\beta$ of $Z_\Phi$ \cite{pep-RR, pep-formal-2,Ma-Zhang14}, where $J$ is
an almost complex structure
compatible with $\Omega$.
In this way, we decompose the character $RR_K(M,\Fcal)$ in sum of (infinite) dimensional representations, while
\break Meinrenken-Sjamaar decomposes  $RR_K(M,\Fcal)$ in sums   of (finite) dimensional representations  by cutting
the Kirwan polytope in small polytopes.
In both of these methods, when $\Fcal$ is weakly $\Phi$-positive,
the study of $[RR_K(M,\Fcal)]^K $ is reduced to the study of $M$ around
$\Phi^{-1}(0)$, and this can be done explicitly, at least in the case where $0$ is a quasi-regular value.
The use of the   compact connected component $\Phi^{-1}(0)$ of the zeroes of
the Kirwan vector field $\kappa_\Phi$ replaces the compactification of a neighborhood of $\Phi^{-1}(0)$ by cutting used in
Meinrenken-Sjamaar.

When $M$ is non-compact, we can also study the localized Riemann-Roch character $RR_K(M,\Fcal,\Phi)$ when the set
$Z_\Phi$ is proper \cite{pep-formal-2,Ma-Zhang14}. In Section \ref{sec:branching} we will concentrate ourselves
to the case of an Hamiltonian action of $\tilde{K}\times K$ on the cotangent bundle $\T^*\tilde{K}$.

\subsection{$[Q,R]=0$ Theorem of Meinrenken-Sjamaar}

Assume in this subsection   that $K$ is a compact connected Lie group, and let $T$ be its maximal torus.
Let $\Lambda\subset \tgot^*$ be the weight lattice of $T$,
and consider the set of dominant weights $\Lambda_{\geq 0}=\Lambda\cap \tgot^*_{\geq 0}$ that parameterizes
the irreducible representations of $K$~: for $\mu\in \Lambda_{\geq 0}$, we denote by  $V^K_\mu$, the irreducible
representation with highest weight $\mu$.

For the remaining part of this section we assume that $M$ is compact and connected,
and we denote $\Delta(M):=\Phi(M)\cap\tgot^*_{\geq 0}$ the Kirwan polytope.

Let $L\to M$ be a $\Phi$-moment bundle. Write
 $$
 RR_K(M,L)=\sum_{\mu\in \Lambda_{\geq 0}} \mm_\mu(L) \, V^K_\mu.
 $$
We will explain how to compute geometrically the multiplicity $\mm_\mu(L)$ of the representation $V^K_\mu$ in
$RR_K(M,L)$ in terms of the fibers of the moment map.

\begin{defi}
For any $a\in \kgot^*$, we define the topological space
 $M_{\rra}:=\Phi^{-1}(Ka)/K.$
\end{defi}
We say that  $M_{\rra}$ is the reduced space of $M$ at $a$.
In particular, the reduced space $M_{red}=\Phi^{-1}(0)/K$ is denoted $M_0$.

Let us recall the notion of quasi-regular value of $\Phi$.

\begin{defi}
An element $a\in\kgot^*$ is a quasi-regular value of $\Phi$ if $a\in \Phi(M)$ and the $K$-orbits in $\Phi^{-1}(Ka)$ all have the same dimension.
\end{defi}

Equivalently, $a$ is a quasi-regular value if there exists a subalgebra $\sgot\subset\kgot$ such that $\Phi^{-1}(Ka)$
is contained in the infinitesimal orbit type stratum $M_{(\sgot)}=\{m\in M,\, (\kgot_m)=(\sgot)\}$.
As we will see, this implies that  $\Phi^{-1}(Ka)$ is a smooth $K$-manifold.
Furthermore, if $a$ is a quasi-regular value, then the orbifold stratification of $M_{\rra}:=\Phi^{-1}(Ka)/K$ consists of one
piece only, and  $M_{\rra}$ is therefore a symplectic orbifold.

Assume that a dominant weight $\mu$ is a quasi-regular value of $\Phi$. We can then consider the orbifold line bundle over $M_\rrmu$,
$$
L_\mu:= \left(L\vert_{\Phi^{-1}(\mu)}\otimes \C_{-\mu}\right)/K_\mu,
$$
and we may define the Riemann-Roch number $RR(M_\rrmu,L_\mu)$.

In general,  we consider quasi-regular values $a$ of $\Phi$ close to $\mu$, and we  can, as we will see,  define the orbifold line bundle
$$
L_{a,\mu}:= \left(L\vert_{\Phi^{-1}(a)}\otimes \C_{-\mu}\right)/K_a,
$$
over the symplectic orbifold $M_{\rra}=\Phi^{-1}(a)/K_a$. The line orbi-bundle $L_{a,\mu}$ is well defined since
for $a$ close enough to $\mu$, the action of $\kgot_m$ on $L\vert_m\otimes \C_{-\mu}$ is trivial for any
$m\in \Phi^{-1}(a)$ (see Lemma \ref{lem:trivial-action-stabilizer}).

We can now state the Meinrenken-Sjamaar Theorem \cite{Meinrenken-Sjamaar}.
\begin{theo}\label{theo:MS}
Let $\mu\in\Lambda_{\geq 0}$.

$\bullet$ If $\mu\notin \Delta(M)$, then $\mm_\mu(L)=0$.

$\bullet$ If $\mu\in \Delta(M)$, then $\mm_\mu(L)= RR(M_{\rra}, L_{a,\mu})$ for any quasi-regular value $a$ of $\Phi$ that is
close enough to $\mu$.
\end{theo}

We can restate this result by defining Riemann-Roch numbers on (singular) symplectic reduced spaces as follows.
\begin{defi}\label{def:Q-M-mu}
Let $L$ be a $\Phi$-moment bundle. For any dominant weight $\mu\in\Lambda_{\geq 0}$, the Riemann-Roch number $RR(M_{\rrmu},L_\mu)$
is defined by the following dichotomy
$$
RR(M_{\rrmu},L_\mu)=
\begin{cases}
   0\qquad\hspace{25mm} {\rm if}\ \mu\notin \  \Delta(M),\\
   RR(M_{\rra},L_{a,\mu})   \hspace{12mm} {\rm if}\ \mu\in \  \Delta(M) \ {\rm  and}\ a \
    \mathrm{is\ a\ quasi \ regular}\\
    \mathrm{value\ close\ enough\ to}\ \mu.
\end{cases}
$$
\end{defi}
\medskip

The $[Q,R]=0$ Theorem of Meinrenken-Sjamaar says  that  for  a $\Phi$-moment bundle $L$, the multiplicity $\mm_\mu(L)$ of the representation $V^K_\mu$ in $RR_K(M, L)$ is equal to $RR(M_{\rrmu},L_\mu)$.

\medskip

\subsection{Localization on $\Phi^{-1}(0)$}

In this section, we consider a $K$-Hamiltonian manifold $(M,\Omega,\Phi)$ not necessarily compact nor connected, and the group $K$ is not supposed connected. We study
the character $RR_K(M,\Fcal, \Phi^{-1}(0),\Phi)$, in the case where $0$ is a quasi-regular value of $\Phi$ and the
manifold $\Phi^{-1}(0)$ is compact.

\subsubsection{Regular case}

We suppose here that $0$ is a regular value of $\Phi$. The group $K$ acts infinitesimally freely on
the compact submanifold $P=\Phi^{-1}(0)$, and $M_{0} :=P/K$ is a compact orbifold equipped with an
induced symplectic form $\Omega_0$. We denote $\pi : P\to P/K=M_{0}$ the projection.

\begin{prop}\label{prop:localisation-Phi-0-hamiltonien}
For any equivariant vector bundle $\Fcal\to M$ we have
\begin{equation}\label{eq:localisation-Phi-0-hamiltonien}
\left[RR_K(M,\Fcal,\Phi^{-1}(0),\Phi)\right]^K=RR(M_{0},\Fcal_{0})
\end{equation}
where $\Fcal_{0}=\Fcal\vert_P/K$.
\end{prop}

\begin{proof}Let $\Ecal:=\bigwedge_J \T M$ be the Clifford bundle associated to a compatible almost complex structure $J$.
We consider the restriction of the tangent bundle $\T M$ on the submanifold $P$. Let $[\kgot]\subset \T M\vert_P$ be the
subbundle which is the image of the map
$(m,X)\in P\times \kgot\mapsto X\cdot m\in \T_m P\subset \T_m M$. We consider the subbundle $[\kgot]\oplus J[\kgot]$ of $\T M\vert_P$ that is canonically isomorphic to $[\kgot_\C]$. Let $E\subset \T M\vert_P$ be the orthogonal of $[\kgot]\oplus J[\kgot]$ relatively to the symplectic form. We see that $E\to P$ is a symplectic bundle isomorphic to the pull-back $\pi^*\T M_0$, and that
$J$ induces an almost complex structure $J_0$ on $E$ compatible with $\pi^*(\Omega_0)$. We see then that $\Ecal\vert_P\simeq \bigwedge \kgot_\C\otimes \Ecal_0$ with $\Ecal_0=\pi^*(\bigwedge_{J_0} \T M_{0})$. Our result follows then from Theorem \ref{theo:localisation-Phi-0}.
\end{proof}

\medskip

\subsubsection{Reduction in stage}

In this section, we explore the case of reduction in stage. Suppose that we have a Hamiltonian action of the compact Lie group
$G\times K$ on the manifold $(M,\Omega)$. Let $\Phi=\Phi_G\oplus\Phi_K :M\to \ggot^*\oplus\kgot^*$ be  the
corresponding moment map. We suppose that
\begin{itemize}
\item  $0$ is a regular value of $\Phi_K$,
\item  $K$ acts freely on $P:=\Phi_K^{-1}(0)$,
\item $Z:=\Phi^{-1}(0)$ is compact.
\end{itemize}

We denote by $\pi: P\to M_0:=P/K$ the corresponding $G$-equivariant principal fibration. We denote by $\phi:M_0\to \ggot^*$ the equivariant map induced by $\Phi_G$ : it is the moment map relative to the Hamiltonian action of $G$ on the symplectic manifold $(M_0,\Omega_0)$.

\begin{prop}\label{prop:localisation-stage}
For any equivariant vector bundle $\Fcal\to M$, we have the following relation
$$
\left[RR_{K\times G}(M,\Fcal,\Phi^{-1}(0),\Phi)\right]^K=
RR_G(M_{0},\Fcal_{0},\phi^{-1}(0),\phi) \qquad \mathrm{in}\quad \hat{R}(G).
$$

Here $\Fcal_0=(\Fcal\vert_P)/K$.
\end{prop}

\begin{proof} We follow the proof of Proposition \ref{prop:localisation-Phi-0-hamiltonien}. We analyze the symbol
$\sigma(M,\Phi)(m,\nu):=\clif_m(\tilde{\nu}-\kappa_\Phi(m)):\bigwedge^{even}_J \T M\to \bigwedge^{odd}_J \T M$ on the model
$\hat{M}=P\times\kgot^*$ in a neighborhood of $Z\times \{0\}$.

The choice of a $G\times K$-invariant connection $1$-form $\theta$ on the $K$-principal bundle $P\to M_0$
defines a projection map $\T_p P\to \kgot$ which associates to a tangent vector its vertical part. This projection
induces a map $j:P\to \hom(\ggot,\kgot)$, where the vertical part of the vector $X\cdot p\in \T_p P$
is equal to $j_p(X)\in\kgot$ for any $(X,p)\in \ggot\times P$.

The symplectic form can be taken as $\hat{\Omega}=\pi^*\Omega_0 +d\langle\theta,\xi\rangle$, and the moment maps
have the following expressions : for $(p,\xi)\in P\times \kgot^*$
\begin{itemize}
\item  $\Phi_K(p,\xi)=\xi$,
\item $\Phi_G(p,\xi)= \phi(\bar{p})+ j_p^*(\xi)$, where $\bar{p}=\pi(p)$.
\end{itemize}

We denote by $\kappa_G$ (resp. $\kappa_\phi$) the Kirwan vector field on $P$ (resp. $M_0$)
associated to the equivariant map $\Phi_G:P\to \ggot^*$ (resp. $\phi$). The Kirwan vector
field $\kappa_\Phi$ has the following expression
$\kappa_\Phi=\kappa_\Phi^{vert}+\kappa_\Phi^{hor}$, where
\begin{itemize}
\item  $-\kappa_\Phi^{vert}(p,\xi)=(Id+j_pj_p^*)\xi+ j_p\phi(\bar{p})$,
\item $-\kappa_\Phi^{hor}(p,\xi)\simeq -\kappa_\phi(\bar{p})+ j_p^*(\xi)\cdot \bar{p}$ through the tangent map
$\T_p\pi$.
\end{itemize}

Let $V\subset U$ be two invariant relatively compact neighborhoods of $Z=\Phi^{-1}_G(0)\cap P$ in $P$ such that
$\pi(\overline{U})\cap \{\kappa_\phi=0\}=\phi^{-1}(0)$, and $\overline{V}\subset U$. The vector field $\kappa_\phi$
does not vanish on the compact set $\pi(\overline{U}\setminus V)$, hence we can choose the connection $1$-form $\theta$
such that the vector field $\kappa_G$ is horizontal on $\overline{U}\setminus V$. In other words, the term $j_p\phi(\bar{p})$
vanishes for $p\in \overline{U}\setminus V$. Let $c >0$ such that $\|j_p\phi(\bar{p})\|\leq c$ on $V$.
We deform $\Phi$ in $\Phi_t(p,\xi)=\xi\oplus\left(\phi(\bar{p}) +t j_p(\xi)\right)$ for $t\in [0,1]$, and we denote
$\kappa_t$ the corresponding vector field. The following lemma is an easy computation

\begin{lem}
For any $t\in [0,1]$, the set $\{(p,\xi)\in U\times \kgot^*,\,\kappa_t(p,\xi)=0\}$ is contained in the compact set
$\overline{V}\times \{\|\xi\|\leq c\}$.
\end{lem}

Thanks to the previous lemma, we see that the symbols $\sigma(M,\Phi)$ and $\sigma(M,\Phi_0)$ define the same
class in $\Ko_{G\times K}(\T_{G\times K}(U\times \kgot^*))$. Let $\sigma_{hor}\in \Ko_{G\times K}(\T_{G\times K}U)$ be the pullback
of $\sigma(M_0,\phi)\vert_{\pi(U)}\in \Ko_{G}(\T_{G}\pi(U))$.

For $m=(p,\xi)\in U\times \kgot^*$ and $\nu:=\eta\oplus a \oplus b\in \T^* (U\times \kgot^*)$, we have
$$
\sigma(M,\Phi_0)(m,\nu)=
\sigma_{hor}(p,\eta)\boxtimes \clif_{\wedge \kgot_\C} ((a +\gamma_p+\xi)\oplus ib)
$$
where $\gamma_p= j_p\phi(\bar{p})$. We see that the support of the symbol $\sigma(M, \Phi_0)$ intersected
with $\T^*_{G\times K} (U\times \kgot^*)$  is just equal to $Z\times\{0\}$.

We see also that for $t\in [0,1]$, the support of the symbol
$$\sigma^t(m,\nu)=\sigma_{hor}(p,\eta)\boxtimes \clif_{\wedge \kgot_\C} ((t(a +\gamma_p)+\xi)\oplus ib)
$$
intersected with $\T^*_{G\times K}(U\times \kgot^*)$ remains equal to $Z\times\{0\}$. Thus we have proved that $\sigma(M,\Phi)$ is equal to $i_!( \sigma_{hor})$ in $\Ko_{G\times K}(\T_{G\times K}(U\times \kgot^*))$.
Our result follows from this fact.

\end{proof}

\subsubsection{Quasi regular case}\label{subsec:quasi-regular}

In this subsection, we study  the character $RR_K(M,\Fcal, \Phi^{-1}(0),\Phi)$, in the case where $0$ is a quasi-regular value of $\Phi$.


In particular,  we prove the first form of the $[Q,R]=0$ theorem.
 \begin{theo}
 If $0$ is a quasi-regular value of the moment map,
and $L$ a $\Phi$-moment bundle,  then
$$[RR_K(M, L)]^K=RR(M_0,L_0).$$
Here $L_0=L\vert_{\Phi^{-1}(0)}/K$.
\end{theo}

 Consider first the basic example of  a Hermitian vector space $W$ with complex structure $J$
 and Hermitian inner product $\mathrm{h}(v,w)$.
Consider  the symplectic form on $W$ given by
   $\Omega=-\mathrm{Im}(\mathrm{h}).$
   Thus $J$ is a complex structure adapted to $\Omega$.
   Let $K$ be a compact group acting by unitary transformations on $W$.
   Then $W$ is a $K$-Hamiltonian manifold with
   moment map  $\langle\Phi(x),X\rangle=\frac{i}{2}\mathrm{h}(X\cdot x,x)$, for $X\in \kgot$, $x\in W$.

 Assume that $\Phi^{-1}(0)=\{0\}$: in this case $0$ is a quasi-regular value. Consider a complex representation space $F$ of $K$, and the $K$-equivariant vector bundle  $[F]=W\times F$.

 \begin{prop}\label{prop:mumford}
  Assume that $\Phi^{-1}(0)=0$.
 Then
 $$RR_K(W,[F],\Phi^{-1}(0),\Phi)=\Sym(W^*)\otimes F.$$
 Furthermore,  $[\Sym(W^*)]^K=\C$.
 \end{prop}

 \begin{rem}
 The second claim follows from  Mumford GIT theory. In spirit, this is the first case of the $[Q,R]=0$ theorem.
 \end{rem}

 \begin{proof}
 Let $x\in W$, and $\kappa_\Phi(x)=-\Phi(x)\cdot x$ be the Kirwan vector field associated to $\Phi$. Then
 $\kappa_\Phi$ vanishes only at $0$ :
 indeed we have $\mathrm{h}( Jx,\kappa_\Phi(x))=2\|\Phi(x)\|^2$.

Our deformed symbol is $\sigma(W,\Phi)(x,\nu)=\clif(\tilde \nu -\kappa_\Phi(x)): \bigwedge^{even} W\longrightarrow \bigwedge^{odd} W$,
while Atiyah symbol $\at(\overline{W})$  is  $\at(\overline{W})(x,\nu)=\clif(\tilde{\nu}-Jx): \bigwedge^{even} W\longrightarrow \bigwedge^{odd} W$.

Let $t\in [0,1]$ and $\sigma^t(x,\nu)=\clif(\tilde{\nu}-(1-t)Jx-t\kappa_\Phi(x))$.
  Let us see that $\sigma^t$ is a transversally elliptic symbol with support $\{(0,0)\}$ in $\T^*_K W$.
  Indeed if $\tilde{\nu}=(1-t)Jx+t\kappa_\Phi(x)$ and is orthogonal to $\kgot \cdot x$, we obtain
  $$
  0=\mathrm{Re}\left(\mathrm{h} (\tilde{\nu}, \kappa_\Phi(x))\right)=2(1-t)\|\Phi(x)\|^2+t\|\kappa_\Phi(x)\|^2,
  $$
  and this implies $x=0$. We obtain our first claim since the equivariant index of $\at(\overline{W})$ is
  $\Sym(W^*)$ (see Proposition \ref{prop-indice-atiyah}).

  Let us now prove that $[\Sym(W^*)]^K=\C$.
  If $P\in \Sym(W^*)$ is $K$-invariant, $P$ is constant on the trajectory of the vector field $\kappa_\Phi$
   which is tangent to the $K$-orbits.
  So the value of $P$ coincide with the value of $P$ on the unique stationary point of $\kappa_\Phi$.
  So $P(x)=P(0)$ and $P$ is constant.
  \end{proof}

\medskip

\begin{rem}\label{rem:comparaison-symbols}
In the proof above we have checked that the transversally elliptic symbols $\sigma(W,\Phi)$ and $\at(\overline{W})$ define
the same class in $\Ko_K(\T^*_K W)$. If $\Tbb$ is the $1$-dimension center of $U(V)$, the same proof shows that
$\sigma(W,\Phi)$ and $\at(\overline{W})$ define the same class in $\Ko_{K\times \Tbb}(\T^*_{\Tbb} W)$
\end{rem}

\medskip

We now consider the general case where $0$ is a quasi-regular value, and we study the character
$RR_K(M,\Fcal,\Phi^{-1}(0),\Phi)$ for any vector bundle $\Fcal\to M$.

Let $\sgot$  be the  subalgebra of $\kgot$  such that $Z:=\Phi^{-1}(0)$ is contained in the submanifold $M_{(\sgot)}$.
Let $M_{\sgot}=\{m\in M, \kgot_m=\sgot\}$. Let $H$ be the normalizer subgroup of $\sgot$ in $K$.
 Thus $M_{(\sgot)}=K\times_H M_\sgot$ and $Z= K\times_H Z_\sgot$ where $Z_\sgot:=\Phi^{-1}(0)\cap M_\sgot$.

 We write  a $H$-invariant decomposition $\kgot=\hgot\oplus \qgot$, with corresponding decomposition
$\kgot^*=\hgot^*\oplus \qgot^*$, and write $\Phi=\Phi_{\hgot}\oplus \Phi_{\qgot}$. Here
$\Phi_\hgot:M\to\hgot^*$ is the moment map relative to the $H$-action on $M$, and the map
$\Phi_\qgot:M\to\qgot^*$ is $H$-equivariant. We start with the basic lemma.

\begin{lem}
The map $\Phi_\qgot$ is equal to zero on $M_\sgot$.
\end{lem}
\begin{proof}
Let $m\in M_\sgot$. Since $\kgot_m=\sgot$, the term $\Phi(m)\in \kgot^*\simeq \kgot$ satisfies $[\sgot,\Phi(m)]=0$ and a fortiori
$\Phi(m)\in\hgot=N_\kgot(\sgot)$.
\end{proof}

\medskip

Thus the restriction $\phi$ of $\Phi$ to $M_\sgot$ takes values in $\hgot^*$: the manifold $M_{\sgot}$ is a
$H$-Hamiltonian manifold, with moment map $\phi$. Let $M_\sgot^0$ be the union of connected components
of $M_\sgot$ that intersect $\phi^{-1}(0)$. Since $S$ acts trivially on $M_\sgot^0$, the map $\phi$ take value in $(\hgot/\sgot)^*$
on $M_\sgot^0$. Since the group $H/S$ acts locally freely on $M_\sgot^0$, the map $\phi: M_\sgot^0\to (\hgot/\sgot)^*$
is a submersion. In particular, $Z_\sgot=\phi^{-1}(0)\cap M_\sgot^0$ is a compact submanifold, so
$Z=K\times_H Z_\sgot$ is a submanifold of $M$, as asserted previously, and the reduced space
$$
M_{0}:=\Phi^{-1}(0)/K= \phi^{-1}(0)\cap M_\sgot^0/ (H/S)
$$
is a compact symplectic orbifold.

\bigskip

Let $\Ncal\to Z$ be the symplectic normal bundle of the submanifold $Z$ in $M$: for $x\in Z$,
$$
\Ncal\vert_x= (\T_x Z)^{\perp}/(\T_x Z)^{\perp}\cap\T_x Z.
$$
Here we have denoted by $(\T_x Z)^{\perp}$ the orthogonal with respect to the symplectic form.

We can equip $\Ncal$ with an $H$-invariant Hermitian structure $\mathrm{h}$ such that the symplectic structure on the fibers
of $\Ncal\to Z$ is equal to $-\mathrm{Im}(\mathrm{h})$.

  The subalgebra $\sgot$ acts fiberwise on the complex vector bundle  $\Ncal\vert_{Z_\sgot}$. We consider the action of $\sgot$ on the fibers of the complex bundle $\Sym(\Ncal^*\vert_{Z_\sgot})$. We will use the following.

\begin{lem}
The subbundle $[\Sym(\Ncal^*\vert_{Z_\sgot})]^\sgot$ is reduced to the trivial bundle $[\C]\to Z_\sgot$.
\end{lem}

\begin{proof}
Let $x\in Z_\sgot$. The vector space $\kgot\cdot x\subset \T_x M$ is totally isotropic, since $ \Omega_x(X\cdot x,Y\cdot x)=
\langle \Phi(x),[X,Y]\rangle=0$, hence we can consider the vector space
$E_x:=(\kgot\cdot x)^\perp/\kgot\cdot x$ that is equipped  with a $K_x$-equivariant symplectic structure $\Omega_x$ : we denote by
$\Phi_{E_x}:E_x\to \sgot^*$ the corresponding moment map.  A local model
for a symplectic neighborhood of $Kx$ is $U:=K\times_{K_x}( (\kgot_x/\sgot)^*\times E_x)$ where the moment map is
$\Phi_x[k;\xi,v]=k(\xi+ \Phi_{E_x}(v))$.
Our hypothesis that $\Phi^{-1}(0)$ is contained in $M_{(\sgot)}$ implies that $\Phi_{E_x}^{-1}(0)= (E_x)^\sgot$. We get then the decompositions
\begin{eqnarray*}
\T_x Z &=&  (\kgot/\sgot)\cdot x \oplus (E_x)^\sgot,\\
\T_x M&=&  (\kgot/\sgot)\cdot x\oplus  (\kgot/\sgot)^* \oplus E_x.
\end{eqnarray*}
Let $S$ be the connected component of $K_x$. We see then that $\Ncal\vert_x$ is equal to the symplectic orthogonal of
$(E_x)^\sgot$ in $E_x$. Hence $\Ncal\vert_x$
is a symplectic vector space with a Hamiltonian action of the compact group $S$ such that $\Phi_{\Ncal\vert_x}^{-1}(0)=
\Phi_{E_x}^{-1}(0)\cap \Ncal\vert_x=0$. As we have seen before, it implies that $[\Sym(\Ncal^*\vert_x)]^\sgot=\C$.
\end{proof}

\bigskip

\begin{defi}
If $\Fcal\to M$ is a $K$- equivariant complex vector bundle, we define on $M_{0}$ the (finite dimensional)  orbi-bundle
$$
\Fcal_{0}:=\left[\Fcal\vert_{Z_\sgot}\otimes \Sym(\Ncal^*\vert_{Z_\sgot})\right]^\sgot/(H/S).
$$
If $\sgot$ acts trivially on the fibers of  $\Fcal\vert_{Z_\sgot}$, the bundle $\Fcal_{0}$ is equal to
$\Fcal\vert_{Z_\sgot}/(H/S)$.
\end{defi}

\medskip

 The  aim of this subsection is the following theorem.
  \begin{theo}\label{theo: Qsemiregular}
Assume that $\Phi^{-1}(0)\subset M_{(\sgot)}$. For any $K$-equivariant complex vector bundle $\Fcal\to M$, we have
 $$
 [RR_K(M,\Fcal, \Phi^{-1}(0),\Phi)]^K=RR(M_{0},\Fcal_{0}).
 $$
\end{theo}

The proof of this theorem will require two steps, inducing and restricting, that are considered in the next subsections.

\subsubsection*{Inducing from $H$ to $K$.}

Let us show that there is a $K$-invariant neighborhood $N$ of  the compact submanifold $Z_\sgot:=\Phi^{-1}(0)\cap M_\sgot$ of $M$
of the form $N=K\times_H (N'\times B_{\qgot^*})$, where $B_{\qgot^*}$ is a small $H$-invariant ball in $\qgot^*$, and $N'$ a $H$-invariant
symplectic submanifold of $M$.

Let $\Vcal \subset M$ be a small $K$-invariant tubular neighborhood of the stratum $M_{(\sgot)}$
with fibration $p: \Vcal\to M_{(\sgot)}$. Let $V=p^{-1}(M_{\sgot})$. Thus $V$ is a $H$-invariant submanifold of
$M$ containing $M_\sgot$, fibered over $M_\sgot$ and $\Vcal=K\times_H V$ is an open subset of $M$.
We consider the map $\Phi_\qgot: V\to \qgot^*$.
\begin{lem}
Let $z\in Z_\sgot$:
\begin{itemize}
\item[(a)] the differential $\T_z\Phi_\qgot : \T_z V\to \qgot^*$ is surjective,
\item[(b)] $\ker(\T_z\Phi_\qgot)\subset \T_z V$ is a symplectic subspace of $\T_z M$.
\end{itemize}
\end{lem}

\begin{proof}For $z\in Z_\sgot$, we have $\T_zM=\qgot\cdot z \oplus \T_z V$ with $\qgot\cdot z\simeq \qgot$ since
$\kgot_z=\sgot$. The subspace $\qgot\cdot z$ is totally isotropic since $\Phi(z)=0$, and by definition the subspace
$\ker(\T_z\Phi_\qgot)\subset \T_z V$ is equal to $(\qgot\cdot z)^\perp\cap \T_z V$. So we have
$(\qgot\cdot z)^\perp= \qgot\cdot z\oplus \ker(\T_z\Phi_\qgot)$. This implies that $\ker(\T_z\Phi_\qgot)\subset \T_z V$ is a symplectic subspace of $\T_z M$ of dimension equal to $\dim(\T_z V)-\dim(\qgot)$: the surjectivity of $\T_z\Phi_\qgot : \T_z V\to \qgot^*$ follows.
\end{proof}

\medskip

Let $\Xcal$ be a small $H$-invariant neighborhood of $Z_\sgot$ in $M_\sgot$, and let $V_\Xcal$ be a small neighborhood of
$\Xcal$ in $p^{-1}(\Xcal)\subset V$ such that for any $z\in V_\Xcal$ the statements $(a)$ and $(b)$ hold.

\medskip

Consider $N':=\{z\in V_\Xcal\,;\, \Phi_\qgot(z)=0\}$. Then  $N'$ is a $H$-Hamiltonian submanifold of $V_\Xcal$
containing $\Xcal$, with symplectic form $\Omega':=\Omega\vert_{N'}$ and moment map $\Phi':N'\to \hgot^*$
equal to the restriction of $\Phi$ to $N'$. If $\Xcal$ is small enough, there exists an invariant neighborhood $U'$ of
$\Xcal$ in $V_\Xcal$ and a diffeomorphism $j: N'\times B_{\qgot^*} \to U'$ such that $\Phi\circ j=\Phi'(n')+\xi$,
for any $(n',\xi)\in N'\times B_{\qgot^*}$.

Finally we have proved that a neighborhood of $Z_\sgot$ in $M$ is diffeomorphic to $N=K\times_H(N'\times B_{\qgot^*})$. Here
$(N',\Omega',\Phi')$ is a $H$-Hamiltonian submanifold of $(M,\Omega,\Phi)$ and the moment map
$\Phi_N: N\to \kgot^*$ is defined by the relation
$$
\Phi_N([k;n',\xi])=k\cdot(\Phi'(n')+\xi).
$$

We return to the study of $RR_K(M,\Fcal,\Phi^{-1}(0),\Phi)$ under the hypothesis that $\Phi^{-1}(0)$ is contained in $M_{(\sgot)}$. We may do the computation in the neighborhood $N$ of $\Phi^{-1}(0)$. Note that $\Phi^{-1}(0)\subset M_{(\sgot)}$ implies that
 $Z':=(\Phi')^{-1}(0)$ is equal to
$Z_\sgot$.
Recall that  $RR_K(N,\Fcal,\Phi^{-1}(0),\Phi)= \Qcal_K(N,\bigwedge_J \T N\otimes \Fcal, \Phi^{-1}(0),\Phi).$
Here  we have chosen  $J_N$, an almost complex structure
that is compatible with the symplectic form $\Omega$. We may then further deform the almost complex structure $J_N$ and we will keep the same index.

 Let $J_{N'}$ be an almost complex structure
on $N'$ that is compatible with $\Omega'$.
If we restrict the Clifford bundle $\Ecal:=\bigwedge_{J_N} \T N$ to the submanifold $N'$, we get
$$
\Ecal\vert_{N'}= S_\qgot\otimes \Ecal'
$$
where $S_\qgot= \bigwedge_{\C} \qgot_\C$.
We need to see that  $\Ecal'=\bigwedge_{J_{N'}} \T N'.$

At each point $y$ of $N'$,
the tangent bundle $\T_y N$ admits a canonical identification with $\qgot\oplus\qgot^*\oplus \T_y N'$.
Through the identification $X\oplus\xi\in\qgot\oplus\qgot^*\mapsto X\oplus i\tilde{\xi}\in \qgot_\C$,
 the vector space $\qgot\oplus\qgot^*$ inherits a complex structure $J_\qgot$ and a Hermitian form.
Take a Riemannian metric $g$ on $\T N$, orthogonal direct sum of the metric
$\Omega'(\cdot, J_{N'}\cdot)$ and of the underlying Riemannian metric of $\qgot_\C$.
 The symplectic form $\Omega_y$ can be written  $g(\cdot,A_y\cdot)$ where $A_y$ is a non degenerate antisymmetric matrix.
 Let $J_y=\frac{A_y}{\sqrt{-A_yA_y^*}}$ be the almost complex structure determined by $A_y$.
 When $y\in Z_\sgot$,  $A_y$ is equal to $J_{N'}\oplus J_\qgot$.
 As $Z_\sgot$ is compact, we see that  if $N'$ is a sufficiently small neighborhhod, $J_y$ is homotopic to $J_{N'}\oplus J_\qgot$,  and we obtain our result.

Proposition \ref{prop:induction-cotangent} tells us then that
$$
\left[RR_K(M,\Fcal, \Phi^{-1}(0),\Phi)\right]^K=\left[RR_H(N',\Fcal,Z',\Phi')\right]^H.
$$
Here $Z'=Z_\sgot$, and we have still denoted by $\Fcal$ the restriction of $\Fcal$ to $N'$.

\subsubsection*{Localization on $M_\sgot$.}

We now analyze
$\left[RR_H(N',\Fcal,Z',\Phi')\right]^H$, where $Z'=Z_\sgot$.

Let us summarize  the situation. We consider a $H$-equivariant decomposition $\hgot=\sgot\oplus\rgot$.
Here $\Xcal$ is a small neighborhood of  $Z_\sgot$ in $M_\sgot$,
$N'$ is a $H$-Hamiltonian  manifold, with symplectic structure $\Omega'$ and moment map $\Phi':N'\to \hgot^*$.
We have $N'_\sgot=\Xcal$. The restriction $\Phi_\Xcal$ of $\Phi'$ to $\Xcal$ takes its values in $\rgot^*$,
and
the zero set $Z'$  of the vector field $\kappa_{\Phi'}$ is the compact submanifold
$\Phi_\Xcal^{-1}(0)=Z_\sgot$ of $\Xcal$.

We denote by $\kappa_\Xcal$ the Kirwan vector field on $\Xcal\subset M_\sgot$ associated to $\Phi_{\Xcal}$.

As we work on a neighborhood of $Z'$,
we may assume that $N'$ is diffeomorphic to $\Ncal'$,
the total space of the normal bundle of $\Xcal$ in $N'$. The restriction of $\Omega'$ on $\Xcal$ induces a symplectic structure on each fibers :
$\Omega_{x}\in \bigwedge^2 \Ncal'\vert_x$ for $x\in\Xcal$.

We choose an $H$-invariant connection on $\Ncal'$. The subspace $\rgot\cdot w$ of $\T_w \Wcal$ projects injectively on
$\rgot\cdot x$, and keep the same dimension $\dim \rgot$. Thus we can assume that the horizontal space contains $\rgot \cdot w$.

If $\eta\in \T_w\Ncal'$, we denote $\eta^{\rm vert}$ its vertical component.
Let $\lambda$ be the $H$-invariant one-form on $\Ncal'$ such that, at a point $w=(x,v)\in \Ncal'$,
 $\lambda_w(\eta)=\Omega_x(v,\eta^{\rm vert})$.
 From the symplectic embedding theorem, we may assume that
 $\Omega'=d\lambda+p^*\Omega_\Xcal$. Here $p$ is the projection $\Ncal'\to \Xcal$.

We  can equip the vector bundle $\Ncal'$
with a  $H$-equivariant Hermitian structure  such that the section $x\in \Xcal\mapsto\Omega_{x}$  is minus  the imaginary part of the Hermitian metric. Thus the
complex structure $J$ on the fiber $\Ncal'\vert_x$ is compatible with $\Omega_{x}$.
Let us choose a $H$-invariant metric on $\Xcal$, and
let $J_\Xcal$ be the corresponding almost complex structure
on $\Xcal$ compatible with $\Omega_\Xcal$.
Thanks to the connection, we can consider on the manifold $\Ncal'$ a $H$-invariant metric
orthogonal direct sum of  the metric on $\Xcal$ (its horizontal lift) and the Hermitian metric on $\Ncal'$, let
 $J_{\Ncal'}$ be the corresponding compatible almost complex structure associated to $\Omega'$ on $\Ncal'$.
We consider also with the help of the connection $J'=J\oplus J_\Xcal$.
On $\Xcal$, we have
$J'=J_{\Ncal'}$.
Thus
 $J_{\Ncal'}$ is a deformation of $J'$ and
 $RR_H(N',\Fcal,Z',\Phi')$ can be computed using the direct sum $J'$.

Remark that we have the additional action of the group $U(1)$ acting by homotheties  $v\mapsto e^{i\theta} v$ on the
fibers of the Hermitian bundle $\Ncal'\to \Xcal$ and commuting with the $H$-action, with infinitesimal generator $J$.
Thus we have a morphism $\at_J:\Ko_H(\T^*_{H} \Xcal)\longrightarrow \Ko_H(\T^*_{H} \Ncal')$ defined
by Equation (\ref{eq:morphismAtbeta}).

\medskip

We compute now the moment map $\Phi'$ associated to the action of $H$ on $(\Ncal',\Omega')$.
We write  $\Phi'=\Phi_\sgot\oplus \Phi_\rgot$. We obtain for $w=(x,v)\in \Ncal'$:
\begin{itemize}
\item $\langle \Phi_\sgot(x,v),X\rangle=\frac{1}{2}\Omega_x(Xv,v)$ if $X\in \sgot$,
\item $\Phi_\rgot(x,v)=\Phi_\Xcal(x)$.
\end{itemize}

Let $\Ncal\to Z$ be the symplectic normal bundle of the submanifold $Z$ in $M$, and let $\Ncal\vert_{Z_\sgot}$
be its restriction on $Z_\sgot$. We will need the following
\begin{lem}
$\bullet$ The restriction of symplectic vector bundle $\Ncal'\to \Xcal$ to $Z_\sgot$ is equal to $\Ncal\vert_{Z_\sgot}$.

$\bullet$ The subbundle $[\Sym((\Ncal')^*)]^\sgot$ is reduced to the trivial bundle $[\C]\to\Xcal$.
\end{lem}

\begin{proof}
Let $x\in Z_\sgot$. Let $\qgot^*_x$ be a subspace of $(\T_x N')^\perp\subset \T_x M$ that is in bijection with $\qgot^*$ through the tangent map $\T_x\Phi_{\qgot}$. Let $(\hgot/\sgot)^*_x$ be a subspace of $\T_x \Xcal$ that is in bijection with $(\hgot/\sgot)^*$ through the tangent map
$\T_x\Phi_{\hgot}$. We have the decompositions
\begin{eqnarray*}
\T_x M&=&  \left(\qgot\cdot x\oplus  \qgot^*_x\right) \stackrel{\perp}{\oplus} \T_x N'\\
\T_x N' &=&  \T_x\Xcal \stackrel{\perp}{\oplus} \Ncal'\vert_x,\\
\T_x\Xcal&=& (\hgot/\sgot)^*_x \oplus \T_x Z_\sgot.
\end{eqnarray*}
that gives $\T_x M=\left(\qgot\cdot x\oplus  \qgot^*_x\right) \stackrel{\perp}{\oplus}\left((\hgot/\sgot)^*_x \oplus \T_x Z_\sgot\right)
\stackrel{\perp}{\oplus} \Ncal'\vert_x$. Now, if we use that $\T_x Z=\qgot\cdot x\oplus\T_x Z_\sgot$, we see that
$(\T_x Z)^\perp= (\T_x Z)^\perp\cap \T_x Z \stackrel{\perp}{\oplus} \Ncal'\vert_x$, and then $\Ncal\vert_x=\Ncal'\vert_x$.
The first point is proved.

The restriction of the bundle $[\Sym((\Ncal')^*)]^\sgot$ to $Z_\sgot$ is equal to $[\Sym(\Ncal\vert^*_{Z_\sgot})]^\sgot$
which is the trivial bundle $[\C]\to Z_\sgot$.  The second point follows.
\end{proof}

\medskip

We are in a very similar situation than Proposition \ref{prop:variation}
: instead of $\beta$ acting fiberwise, we have $\sgot$ (an $H$-stable ideal in $\hgot$)  acting fiberwise.

Let $w=(x,v)\in \Ncal'$, with $x\in \Xcal$, and $v\in \Ncal'\vert_x$.
The horizontal lift of $\kappa_\Xcal$ (still denoted by $\kappa_\Xcal$) is the vector field
$\Phi_\rgot(w)\cdot w$. As in the proof of Proposition \ref{prop:mumford}, we introduce the deformation
$\kappa^t(w)=-\Phi_\rgot(w)\cdot w- t \Phi_\sgot(w)\cdot w+(1-t) Jw$
and
$$
\sigma^t(w,\nu)=\clif_w(\tilde \nu- \kappa^t(w)), \quad (w,\nu)\in \T^*\Ncal'
$$
acting on $\bigwedge_{J'} \T \Ncal'\otimes \Fcal$. Note that the vector $\Phi_\rgot(w)\cdot w$ is horizontal,
while  $\Phi_\sgot(w)\cdot w$ and  $Jw$
are verticals.

By definition, the character $RR_H(N',\Fcal,Z',\Phi')$ is equal to the equivariant index of $\sigma^1\otimes\Fcal$.
Let us prove that the characteristic set of
$\sigma^t$ intersected with $\T^*_H \Ncal'$  stays equal to $Z_\sgot$.
Consider $\eta=\kappa^t(w)$ and orthogonal to $\hgot\cdot w$.
Thus we obtain
$\langle \Phi_\rgot(w)\cdot w+ t \Phi_\sgot(w)\cdot w-(1-t) Jw, \Phi_\rgot(w)\cdot w\rangle =0$
and $\langle \Phi_\rgot(w)\cdot w + t \Phi_\sgot(w)\cdot w-(1-t) Jw, \Phi_\sgot(w)\cdot w\rangle =0.$
Thus considering vertical and orthogonal components, we obtain
$\Phi_\rgot(w)=0$.
Then we have
\begin{eqnarray*}
0&=&\langle t \Phi_\sgot(w)\cdot w-(1-t) Jw, \Phi_\sgot(w)\cdot w\rangle\\
&=&  t \|\Phi_\sgot(w)\cdot w\|^2+ 2(1-t) \|\Phi_\sgot(w)\|^2
\end{eqnarray*}
as $\langle Jw,\Phi_\sgot(w)\cdot w\rangle = - 2\|\Phi_\sgot(w)\|^2$.

Thus we have proved that if $\kappa^t(w)$ is orthogonal to $\hgot\cdot w$, then $\Phi'(w)=0$: this implies
$w=(x,0)$ with $x\in Z_\sgot$.

We know then that $RR_H(N',\Fcal,Z',\Phi')$ is equal to the equivariant index of the transversally elliptic symbol
$\sigma^0\otimes\Fcal$. Since the symbol  $\sigma^0$  is  equal to $\at_J(\sigma(\Xcal,Z_\sgot,\Phi_\Xcal))$,
we have then by Equation (\ref{eq:indice-Atbeta})
$$
\left[RR_H(N',\Fcal,Z',\Phi')\right]^H=\left[RR_H(\Xcal,\Fcal\otimes \Sym((\Ncal')^*),Z_\sgot,\Phi_\Xcal)\right]^H.
$$
We obtain
\begin{eqnarray*}
\lefteqn{\left[RR_H(\Xcal,\Fcal\otimes \Sym((\Ncal')^*),Z_\sgot,\Phi_\Xcal)\right]^H}\\
&=&\left[RR_{H/S}(\Xcal,\left[\Fcal\otimes \Sym((\Ncal')^*)\right]^\sgot,Z_\sgot,\Phi_\Xcal)\right]^{H/S},\\
&=& RR(M_0,\Fcal_0)\qquad (1)
\end{eqnarray*}
where $\Fcal_0:=\left[\Fcal\vert_{Z_\sgot}\otimes \Sym(\Ncal^*\vert_{Z_\sgot})\right]^\sgot/(H/S)$. In $(1)$,
 we  used the regular case considered in Proposition \ref{prop:localisation-Phi-0-hamiltonien}.

\subsection{Riemann-Roch number on symplectic reduction}

Let $(M,\Omega,\Phi)$ be a Hamiltonian $K$-manifold not necessarily compact. Suppose that $a$ is a quasi-regular value of
$\Phi$ such that the fiber $Z:=\Phi^{-1}(a)$ is a {\em compact} submanifold. The quotient
$M_{\rra}:=\Phi^{-1}(a)/K_a$ is a compact symplectic orbifold.

Let $\sgot\subset \kgot_a$ be a subalgebra such that $\Phi^{-1}(Ka)$ is contained in the submanifold $M_{(\sgot)}$. We denote by $S$ the connected subgroup of $K_a$ with Lie algebra $\sgot$, and $H$ the normalizer subgroup of $\sgot$ in $\kgot_a$. The orbifold $M_{\rra}=\Phi^{-1}(a)/K_a$ is also the quotient of the submanifold $Z_\sgot:=\Phi^{-1}(a)\cap M_{\sgot}$ by the group $H/S$.

We consider now the case of an equivariant vector bundle $\Fcal\to M$ and a character $\chi_\lambda: K_a\to \mathrm{U}(1)$ : here $\lambda=\frac{1}{i}d\chi_\lambda$ is a weight, and we denote $\C_\lambda$ the $1$-dimensional representation of $K_a$ defined by $\chi_\lambda$.

We suppose that the action of $S$ on the fibers of the bundle $\Fcal\vert_{Z_\sgot}$ is $t\cdot v=\chi_\lambda(t)v$. Then $S$ acts trivially on $\Fcal\vert_{Z_\sgot}\otimes \C_{-\lambda}$ and we may form the orbibundle
$$
\Fcal_{a,\lambda}:=\left(\Fcal\vert_{Z_\sgot}\otimes \C_{-\lambda}\right)/(H/S)
$$
on $M_{\rra}$. The aim of this section is to show that the Riemann-Roch number $RR(M_{\rra}, \Fcal_{a,\lambda})$ can be computed by a localization procedure.

\medskip

Let $(Ka)^-$ be the coadjoint orbit $Ka$ with the opposite symplectic structure. Let $[\C_{-\lambda}]=K\times_{K_a}\C_{-\lambda}$
be the line bundle on $(Ka)^-$. We consider the symplectic manifold $M\times (Ka)^-$ with the complex bundle
$\Fcal\boxtimes [\C_{-\lambda}]$ on it. The corresponding moment map
$\Phi_a: M\times (Ka)^-\to \kgot^*$ is $\Phi_a(m,\xi)=\Phi(m)-\xi$.
As the set $\Phi_a^{-1}(0)$ is compact, we may define the localized Riemann Roch character
$$
RR_K(M\times (Ka)^-,\Fcal\boxtimes [\C_{-\lambda}],\Phi_a^{-1}(0),\Phi_a)\, \in \, \hat{R}(K).
$$

The main result of this section is the

\begin{theo}\label{theo:RR-M-a-E-alpha}
If $a$ is a quasi-regular value of $\Phi$ we have
\begin{equation}\label{eq:RR-M-a-E-alpha}
\left[RR_K(M\times (Ka)^-,\Fcal\boxtimes [\C_{-\lambda}],\Phi_a^{-1}(0),\Phi_a)\right]^K= RR(M_{\rra}, \Fcal_{a,\lambda}).
\end{equation}
\end{theo}

\begin{proof} Let $B\subset \kgot^*_a$ be a small $K_a$-invariant ball around $a$, and consider the slice
$Y:=\Phi^{-1}(B)$ which is a symplectic submanifold of $M$: the action of $K_a$ on $(Y,\Omega\vert_Y)$
is Hamiltonian with moment map $\Phi_Y:=\Phi\vert_Y$.

The computation of the left hand side of (\ref{eq:RR-M-a-E-alpha}) is done in the manifold $\tilde{M}\times (Ka)^-$ where
$\tilde{M}=K\times_{K_a}Y$ is viewed as an open subset of $M$. Let $J_Y$ be a $K_a$-invariant almost complex
structure on $Y$ compatible with the symplectic form $\Omega\vert_Y$. Let $J_{\tilde{M}}$ be the $K$-invariant complex
structure on $\tilde{M}$ that is equal to $J_a\oplus J_Y$ at the points $[1,y]$, $y\in Y$.  Notice that $J_{\tilde{M}}$ is compatible
with the symplectic structure $\Omega\vert_{\tilde{M}}$.

If we apply the induction formula of Proposition \ref{prop:shifting trick} to the manifold $\tilde{M}\times (Ka)^-$ equipped
with the Clifford bundles $\Gcal:=\bigwedge_J \T M\otimes \Fcal$ on $\tilde{M}$ and
$\Rcal= \bigwedge_{-J_a} \T Ka\otimes  [\C_{-\lambda}]$ on $Ka$, we get

$\bullet$ $\Gcal_{\ddY}=\bigwedge_{J_Y}\T Y\otimes \Fcal\vert_Y$  for the corresponding Clifford bundle on $Y$,

$\bullet$ $R=\C_{-\lambda}$ for the corresponding $K_a$-module.

Hence we have
\begin{eqnarray}\label{eq:reduction-slice}
\lefteqn{\left[RR_K(M\times (Ka)^-,\Fcal\boxtimes [\C_{-\lambda}],\Phi_a^{-1}(0),\Phi_a)\right]^K}&&\nonumber\\
&=&\left[RR_{K_a}(Y, \Fcal_{\lambda},\phi^{-1}(0),\phi)\right]^{K_a}.
\end{eqnarray}
where $\phi:=\Phi_Y-a$ is a moment map on the $K_a$-Hamiltonian manifold $(Y,\Omega_Y)$, $0$ is a quasi regular value of
$\phi$ and $\Fcal_{\lambda}=\Fcal\vert_Y\otimes \C_{-\lambda}$.

We are in the setting of Section \ref{subsec:quasi-regular}. Since $\sgot$ acts trivially on the fibers of the vector bundle
$\Fcal_{\lambda}\vert_{Z_\sgot}$, Theorem \ref{theo: Qsemiregular} says that
$$
\left[RR_{K_a}(Y, \Fcal_{\lambda},\phi^{-1}(0),\phi)\right]^{K_a}= RR(M_{\rra}, \Fcal_{a,\lambda})
$$
with $\Fcal_{a,\lambda}= \Fcal_{\lambda}\vert_{Z_\sgot}/(H/S)$. The proof is completed.

\end{proof}

\subsection{Multiplicities as Riemann-Roch numbers}\label{subsec:multiRiemannRoch}

In this section, we give a proof of the $[Q,R]=0$ Theorem of Meinrenken-Sjamaar.

Let $L$ be a $\Phi$-moment line bundle on the compact Hamiltonian $K$-manifold $(M,\Omega,\Phi)$.  We fix a dominant weight
$\mu$, and we consider the multiplicity $\mm_\mu(L)=[RR_K(M,L)\otimes V_\mu^*]^K$. We consider
elements $a\in\tgot^*$ close enough to $\mu$, so $K_a\subset K_\mu$. Let $(Ka)^-$ be the coadjoint orbit
equipped with the opposite symplectic structure and the line bundle $[\C_{-\mu}]\simeq K\times_{K_a} \C_{-\mu}$.
We have $RR_K((Ka)^-,[\C_{-\mu}])= V_\mu^*$, hence the shifting trick gives
$$
\mm_\mu(L)=[RR_K(M\times (Ka)^-,\Lbb(a))]^K.
$$
Here $\Lbb(a)=L\boxtimes [\C_{-\mu}]$ and
the moment map is $\Phi_a:M\times (Ka)^-\to \kgot^*, (m,\xi)\mapsto \Phi(m)-\xi$. We start with the
\begin{lem}\label{lem:close-enough}
There exist $\epsilon>0$ such that the line bundle $\Lbb(a)$ is weakly $\Phi_a$-positive if $\|a-\mu\|\leq \epsilon$.
\end{lem}

\begin{proof}
We consider the maps $\varphi_a: M\times K/T\to M\times (Ka)^-$ defined by $\varphi(m,[k])=(m,ka)$.
The line bundle $\varphi^*_a\Lbb(a)$ does not depends of $a$ and is denoted $\Lbb'$. Let $\Phi'_a=\varphi_a^*\Phi_a$.

One checks that $\varphi_a(\{\kappa_{\Phi'_a}=0\})=\{\kappa_{\Phi_a}=0\}$, hence we can use Lemma \ref{lem:varphi-pullback} :
the line bundle $\Lbb(a)$ is weakly $\Phi_a$-positive if and only if the line bundle $\Lbb'$ is weakly $\Phi'_a$-positive.

Now we see that for $a=\mu$, the line bundle $\Lbb'$ is a $\Phi'_\mu$-moment line bundle. Thanks to Proposition \ref{prop:defPhi}, we know that $\Lbb'$ is weakly $\Phi'_a$-positive if $a$ is close enough to $\mu$.

\end{proof}

\medskip

Now we take $a\in \Delta(M)$ such that $\|a-\mu\|\leq \epsilon$. Since $\Lbb(a)$ is a weakly $\Phi_a$-positive line bundle on the
Hamiltonian $K$-manifold $M\times (Ka)^-$ and $\Phi_a^{-1}(0)\neq\emptyset$, Theorem \ref{th:QR-J-hamiltonian} says that
$$
\mm_\mu(L)=[RR_K(M\times (Ka)^-,\Lbb(a),\Phi_a^{-1}(0),\Phi_a)]^K.
$$

We suppose furthermore that $a$ is a quasi-regular value of $\Phi$. Let $\sgot$ be the subalgebra of $\kgot_a$ such that
$\Phi^{-1}(Ka)$ is contained in $M_{(\sgot)}$ : we denote $Z_\sgot$ the submanifold $\Phi^{-1}(a)\cap M_\sgot$.
In order to make use of Theorem \ref{theo:RR-M-a-E-alpha} we need to compute the action of $S$ on the fibers of the line
bundle $L\vert_{Z_\sgot}\otimes\C_{-\mu}$.

\begin{lem}\label{lem:trivial-action-stabilizer}
If $a$ is a quasi-regular value close enough to $\mu$, the group $S$ acts trivially on
the fibers of the line bundle $L\vert_{Z_\sgot}\otimes\C_{-\mu}$.
\end{lem}
\begin{proof}
We consider the $K_\mu$-invariant subset $U\subset M$ formed by the points $m\in M$ such that the algebra
$\kgot_m\cap\kgot_\mu$ acts trivially on $L\vert_m\otimes\C_{-\mu}$.
The Kostant formulas shows that $\Phi^{-1}(\mu)\subset U$, and if we use local symplectic models of orbits
$K_\mu\cdot m$, one sees that $U$
is an open subset. Hence $\Phi^{-1}(a)\subset U$ for $a$ close enough to $\mu$. The proof is completed.
\end{proof}

Finally, we can use Theorem \ref{theo:RR-M-a-E-alpha} when $a$ is a quasi-regular value close enough to $\mu$ :  relation (\ref{eq:RR-M-a-E-alpha}) gives
$$
\mm_\mu(L)=[RR_K(M\times (Ka)^-,\Lbb(a),\Phi_a^{-1}(0),\Phi_a)]^K=RR(M_{\rra},L_{\mu,a}).
$$
The proof of Theorem \ref{theo:MS} is then completed.

\subsection{Quasi polynomial behavior}\label{subsec:quasipol}

We work in the same setting than in Subsection \ref{subsec:multiRiemannRoch}.
We fix a $\Phi$-moment line bundle $L$ on $M$.
For any dominant weight $\mu\in \Lambda_{\geq 0}$ and any integer $k\in\N$, we define $\mm(k,\mu)\in \N$
as the multiplicity of the representation $V^K_\mu$ in $RR_K(M,L^{\otimes k})$.

We introduce
\begin{enumerate}
\item the vector subspace $V_M\subset \R\oplus\tgot^*$ generated by $\{1\}\times\Delta(M)$,
\item the cone $C_M=\R_{\geq 0}\left(\{1\}\times\Delta(M)\right)$ contained in $V_M$,
\item the lattice $\Lambda_M=\{\Z\oplus \Lambda\}\cap V_M$,
\item the principal face $\sigma$ for the $K$-manifold : $\sigma$ is the relative interior of a face of the Weyl chamber and
$\sigma\cap \Delta(M)$ is dense in $\Delta(M)$.
\end{enumerate}

Theorem \ref{theo:MS} tells us that $\mm(k,\mu)=0$ if $(k,\mu)\notin C_M$. We will now give some qualitative
properties on the behavior of the multiplicities $\mm(k,\mu)$ on the cone $C_M$.

 If $\xi$ varies in $\sigma$, the subgroup $K_\xi$ is independent of $\xi$. We denote it by $K_\sigma$.
 Let $Y=\Phi^{-1}(\sigma)$ be the slice attached to the principal face. It is a $K_\sigma$-invariant symplectic submanifold of
$M$, and the subgroup $[K_\sigma,K_\sigma]$ acts trivially on $L\vert_Y\to Y$ \cite{L-M-T-W}.
The vector space $\R\sigma$ is naturally identified with
the dual of the Lie algebra of the torus $A_\sigma:= K_\sigma/[K_\sigma,K_\sigma]$. Hence the restriction
$\Phi_Y=\Phi\vert_Y : Y\to\sigma$ corresponds to the moment map relative to the action of $A_\sigma$ on
$(Y,\Omega\vert_Y)$.

Let $I_M\subset \R\sigma$ be the affine subspace generated by $\Delta(M)$ : the orthogonal, denoted $\tgot_M$,
of the direction of $I_M$
in $\mathrm{Lie}(A_\sigma)$ corresponds to the generic infinitesimal stabilizer for the $A_\sigma$-action on $Y$.
Let $\Delta(M)^0\subset \Delta(M)\cap\sigma$ be the open subset formed by the regular values of $\Phi_Y : Y\to I_M$.
Note that any element $\xi\in \Delta(M)^0$ is a quasi-regular value of $\Phi$.
 The reduced space $M_{\xi}$ stays of constant dimension when $\xi\in  \Delta(M)^0$, and this dimension
 is the maximal dimension of the reduced spaces $M_{\xi}$, where $\xi$ varies over the set of quasi-regular values of $\Phi$.

If $\agot\subset \Delta(M)^0$ is a connected component, we choose a point $\xi\in \agot$, and we consider

\begin{itemize}
\item  the closed subcone $C_\agot = \R_{\geq 0}\left(\{1\}\times \overline{\agot}\right)$ of $C_M$,
\item  the symplectic orbifold $M_\agot:= \Phi^{-1}(\xi)/A_\sigma$,
\item  the family of line orbi-bundle $L_\agot^{k,\mu}\to M_\agot$ parameterized by $(k,\mu)\in \Lambda_M$~:
$$
L_\agot^{k,\mu}:=\left(L^{\otimes k}\vert_{\Phi^{-1}(\xi)}\otimes \C_{-\mu}\right)/A_\sigma
$$
\end{itemize}

The line orbi-bundles $L_\agot^{k,\mu}$ are well-defined for any $(k,\mu)\in \Lambda_M$ since
$\Lcal(X)=0$ on $L^{\otimes k}\vert_{\Phi^{-1}(\xi)}\otimes \C_{-\mu}$ for any $X\in\tgot_M$.

\begin{defi}
Let $\agot$ be a connected component of $\Delta(M)^0$. We consider the function
$p_\agot: \Lambda_M\to \Z$ defined by the relation
$$
p_\agot(k,\mu):=RR(M_\agot,L_\agot^{k,\mu})
$$
\end{defi}

The Kawasaki-Riemann-Roch Theorem \cite{Kawasaki81} implies that the function $p_\agot$ have a quasi polynomial
behavior on the lattice $\Lambda_M$.

The function $p_\agot$ does not depend on the choice of the choice of  $\xi$ in $\agot$. It is due to the fact that
$p_\agot(k,\mu)$ is equal to
$$
\left[RR_{A_\sigma}\left(Y, L^{\otimes k}\vert_Y, \Phi_Y^{-1}(\xi),\Phi_Y-\xi\right)\otimes \C_{-\mu}\right]^{A_\sigma}.
$$
See (\ref{eq:RR-M-a-E-alpha}). Hence if we take another $\xi'\in\agot$, we consider a continuous family
$\xi_t\in \agot, t\in [0,1]$ joining $\xi$ and $\xi'$. By an easy homotopy argument, we see that
$\left[RR_{A_\sigma}\left(Y, L^{\otimes k}\vert_Y, \Phi_Y^{-1}(\xi_t),\Phi_Y-\xi_t\right)\otimes \C_{-\mu}\right]^{A_\sigma}$
is independent of $t\in [0,1]$.

We can now state the main result of this section. Notice that $C_M=\bigcup_{\agot} C_\agot$.

\begin{theo}\label{theo:quasipolynomial}
If $(k,\mu)\in C_\agot$, we have $\mm(k,\mu)=p_\agot(k,\mu)$.
\end{theo}

\begin{proof} Let us fix $(k,\mu)\in C_\agot$. We apply Theorem \ref{theo:MS} to the $K$-manifold
$\X:=(M,k\Omega,k\Phi)$ equipped with the moment line bundle $\Lbb:=L^{\otimes k}$. The multiplicity of
$V^K_\mu$ in $RR_K(\X,\Lbb)$ is equal to
$RR(\X_{\rra},\Lbb_{a,\mu})$ where $a$ is any quasi-regular value of $k\Phi$ close enough to $\mu$. We can take $a=k\xi$ with
$\xi\in\agot$ : the reduced space $\X_{\rra}$ is equal to $M_{\xi}$ and the line bundle $\Lbb_{a,\mu}$ coincides with $L_\agot^{k,\mu}$.
 We have proved that $\mm(k,\mu)=RR(\X_{\rra},\Lbb_{a,\mu})=p_\agot(k,\mu)$.
\end{proof}

\subsection{Multiplicities on a face}\label{multiplicity-face}

Let $F$ be a (closed) face of the polytope $\Delta(M)$ that intersects the principal face $\sigma$.
We say that $F$ is a general face of $\Delta(M)$.

We assume in this section that $\mu$ is a dominant weight that belongs to $F$.
Let $T_\sigma\subset T$ be the connected component of the center of the subgroup $K_\sigma$.
The dual of the Lie algebra of $T_\sigma$ is the subspace $\R\sigma\subset \tgot^*$. Let $T_F\subset T_\sigma$ be the subtorus
with Lie algebra equal to the orthogonal of the direction $\overrightarrow{F}\subset \R\sigma$ of the face $F$. We denote $K_F$
the subgroup of $K$ that centralizes $T_F$: we see that $K_\sigma\subset K_F$.

We start with the basic fact.

\begin{lem}\label{lem:phi-F-sigma}
The pullback $\Phi^{-1}(F\cap \sigma)$ is connected and belongs to the submanifold $M^{T_F}$.
\end{lem}

\begin{proof}
The pullback $\Phi^{-1}(F\cap \sigma)$ belongs to the slice $Y=\Phi^{-1}(\sigma)$, and it  is connected as each
fiber $\Phi^{-1}(\xi)$ is connected. Since $F$ is a face of the
polytope $\Delta(M)\subset\R\sigma$, we see that $\Delta(M)$ is contained in a cone
of the form
$$
\bigcap_{j=1}^p \{\xi\in \R\sigma\ |\  \langle\xi,X_j\rangle\geq \langle\mu,X_j\rangle\}
$$
where $X_1,\ldots,X_p$ is a basis of $\mathrm{Lie}(T_F)$. Since $\Delta(M)\cap\sigma=\mathrm{Image}(\Phi\vert_Y)$,
the functions $y\in Y\mapsto \langle\Phi(y),X_j\rangle$
reach their minimum at any point of $\Phi^{-1}(F\cap \sigma)$. So the differential of the function
$y\in Y\mapsto \langle\Phi(y),X_j\rangle$ vanishes at any point of $\Phi^{-1}(F\cap \sigma)$.
This implies that
$\Phi^{-1}(F\cap \sigma)\subset \bigcap_{j=1}^p M^{X_j}= M^{T_F}$.
\end{proof}

\medskip

Let $M_F$ be the connected component of $M^{T_F}$ that contains $\Phi^{-1}(F\cap\sigma)$. It is a $K_F$-Hamiltonian
manifold equipped with the two form $\Omega_F:=\Omega\vert_{M_F}$ and the moment map
$\Phi_F=\Phi\vert_{M_F} : M_F\to \kgot_F^*$. We denote $L\vert_F$
the restriction of the $\Phi$-moment line bundle $L$ on $M_F$.
It is easy to see that $L\vert_F$ is a $\Phi_F$-moment line bundle.

We parameterize the irreducible representations of $K_F$ thanks to the choice of a Weyl chamber $\Ccal_F\subset \tgot^*$
that contains the Weyl chamber $\tgot^*_{\geq 0}$ of $K$. So, we may consider the irreducible representation
$V_\mu^{K_F}$ of the group $K_F$ with highest weight $\mu$.

A direct consequence of the $[Q,R]=0$ Theorem is the following proposition.

\begin{prop}\label{QR-mu-face}
Let $L$ be a $\Phi$-moment bundle on $M$, and let $F$ be a  face of $\Delta(M)$ that intersects the principal face $\sigma$.
For any dominant weight $\mu$ that belongs to $F$, the multiplicity $\mm_\mu(L)$ of the representation
$V^K_\mu$ in $RR_K(M, L)$ is equal to the multiplicity of the representation $V^{K_F}_\mu$ in $RR_{K_F}(M_F, L\vert_F)$.
\end{prop}

\begin{proof} We have $\mm_\mu(L)=RR(M_{\rra}, L_{a,\mu})$ for any quasi-regular value of $\Phi$ close enough to $\mu$.
Since $\mu\in F$, we can choose $a$ in $F\cap\sigma$, and then $\Phi^{-1}(a)=\Phi^{-1}_F(a)\subset M_F$. The proposition
follows then from the fact that  the datas $ L_{a,\mu}\to M_{\rra}$ and $(L\vert_F)_{a,\mu}\to (M_F)_{\rra}$ are equal (the last one
is a reduction relatively to the action of the group $K_F$ on $L\vert_F\to M_F$).

\end{proof}

\section{Branching laws}\label{sec:branching}

Let $\tilde{K}$ be a compact connected Lie group and let $K$  be a closed connected subgroup.
Let $i$ denote the inclusion of $K$ into $\tilde{K}$, $i : \kgot \to \tilde{\kgot}$ the induced embedding of Lie algebras,
and $\pi : \tilde{\kgot}^* \to \kgot^*$ the dual projection.

We consider the following action of $\tilde{K}\times K$ on $\tilde{K}$ : $(\tilde{k},k)\cdot g= \tilde{k} g k^{-1}$.
Then $N=\T^*\tilde{K}$, the cotangent bundle of  $\tilde{K}$ is a $\tilde K\times K$ Hamiltonian manifold.
Recall from Subsection \ref{subsec:cotangent} (Equation \ref{eq:momentcotangent})  that  the corresponding moment map relative to the $\tilde{K}\times K$-action is the
map $\Phi=\Phi_{\tilde{K}} \oplus \Phi_K : \T^*\tilde{K} \to \tilde{\kgot}^*\oplus \kgot^*$ defined by
$$
\Phi_{\tilde{K}}(g,\xi)=-\xi,\quad \Phi_K(g,\xi)=\pi(g^{-1}\xi).
$$

Select maximal tori $T$ in $K$ and $\tilde{T}$ in $\tilde{K}$ such that $T\subset \tilde{T}$, and
Weyl chambers $\tilde{\tgot}^*_{\geq 0}$ in $\tilde{\tgot}^*$ and
$\tgot^*_{\geq 0}$ in  $\tgot^*$, where $\tgot$ and $\tilde{\tgot}$ denote the Lie algebras of $T$, resp. $\tilde{T}$.

Let $\Delta(\T^*\tilde{K})\subset \tilde{\tgot}^*_{\geq 0}\times \tgot^*_{\geq 0}$
be the Kirwan polytope associated to $\Phi$:
$$
\Delta(\T^*\tilde{K})=\left\{(\alpha,\beta)\in\tilde{\tgot}^*_{\geq 0}\times \tgot^*_{\geq 0}\, |\, -\beta\in \pi(\tilde{K}\alpha) \right\}.
$$

For $(\tilde \lambda,\lambda)\in \tilde{\tgot}^*_{\geq 0}\oplus \tgot_{\geq 0}^*$, we denote by $N_{{\tilde \lambda},\lambda}$ the reduced space
of $N:=\T^*\tilde{K}$ at $(\tilde \lambda,\lambda)$. Consider the coadjoint orbit  $\tilde{K}\tilde{\lambda}^*:=-\tilde{K}\tilde{\lambda}$, equipped with the $K$-action
 and moment map $\pi\vert_{\tilde{K}\tilde{\lambda}^*}$. Let $(\tilde{K}\tilde{\lambda}^*)_{\lambda}$ the reduced manifold of $\tilde{K}\tilde{\lambda}^*$ at $\lambda$.
 The following lemma is obvious.

\begin{lem}\label{lem:N-lambda-tilde}
As sets, $(\tilde{K}\tilde{\lambda}^*)_{\lambda}$ is isomorphic to $N_{{\tilde \lambda},\lambda}$.
\end{lem}

\medskip

Let $\tilde{\Lambda}_{\geq 0}\subset\tilde{\tgot}^*_{\geq 0}$, $\Lambda_{\geq 0}\subset\tgot^*_{\geq 0}$ be the
set of dominant weights. For any $(\lambda,\tilde{\lambda})\in \Lambda_{\geq 0}\times\tilde{\Lambda}_{\geq 0}$,
we denote $V_\lambda$, $\tilde{V}_{\tilde{\lambda}}$ the corresponding irreducible representations of $K$ and $\tilde{K}$ , and we
define
\begin{equation}\label{eq:multiplicity-K-tilde}
\mm(\tilde{\lambda},\lambda) \in\N
\end{equation}
as the dimension of the vector space $\left[\tilde{V}_{\tilde{\lambda}}\vert_{K}\otimes V_{\lambda}\right]^K$. In other words
$\mm(\tilde{\lambda},\lambda)$ is the multiplicity of the representation $V_{\lambda}$ in $(\tilde{V}_{\tilde{\lambda}})^*\vert_{K}$.

We consider  the coadjoint orbit $\tilde{K}\tilde{\lambda}^*$, equipped with  the prequantum line bundle $L(\tilde{\lambda}^*):=[\C_{\tilde{\lambda}^*}]$.
The $[Q,R]=0$ theorem \ref{theo:MS} says that $\mm(\tilde{\lambda},\lambda)=RR((\tilde{K}\tilde{\lambda}^*)_{\lambda}, L(\tilde{\lambda}^*)_{\lambda})$, in particular
$\mm(\tilde{\lambda},\lambda)=0$ if $(\tilde{K}\tilde{\lambda}^*)_{\lambda}=\emptyset$.

If we use Lemma \ref{lem:N-lambda-tilde}, we see that the cone $\Delta(\T^*\tilde{K})$ is related to the multiplicity function (\ref{eq:multiplicity-K-tilde}) through the following basic result.
\begin{lem}
We have $\mm(\tilde{\lambda},\lambda)=0$ if $(\tilde{\lambda},\lambda)\notin \Delta(\T^*\tilde{K})$.
\end{lem}

\medskip

Let $(\tilde X,X)$ be an element of  $\tilde \tgot\times \tgot$.
In the following lemma, which follows by direct checking, we describe the manifold $N^{(\tilde X,X)}$ of zeroes of the vector field $(\tilde X,X)$ on $N=\T^*\tilde K$. We denote $\tilde{K}_X\subset \tilde{K}$ the subgroup that fixes $X$. The cotangent bundle $\T^*\tilde{K}_X$ is then a submanifold of $\T^*\tilde{K}$.

\begin{lem}\label{lem:point-fixe-cotangent}

\begin{itemize}

\item If $\tilde X$ is not conjugated in $\tilde \kgot$ to $X$, then
$N^{(\tilde X,X)}=\emptyset$

\item If $\tilde X=aX$ with $a\in \tilde K$, then
$N^{(\tilde X,X)}=a\cdot \T^*\tilde{K}_X$.

\end{itemize}
\end{lem}

We recall from general principle that the Kirwan polyhedron has non zero interior in
 $\tilde{\tgot}^*\times \tgot^*$ if and only if
the generic stabilizer of the $\tilde{K}\times K$-action on $\T^* \tilde{K}$ is finite.

To simplify the exposition, we assume from now on that no nonzero ideal of $\kgot$ is an ideal of $\tilde{\kgot}$. It implies the following

\begin{lem}
The intersection of  $\Delta(\T^*\tilde{K})$ with $\tilde{\tgot}^*_{>0}\times \tgot^*_{>0}$ has a non empty interior, and the generic stabilizer of the $\tilde{K}\times K$-action on $\T^* \tilde{K}$ is finite.
\end{lem}
\begin{proof}
By the preceding computation, we see that the set $N^1$ of elements of $N$ having a non zero
stabilizer is
$N^1={\tilde K}\cdot \cup_{X\in \kgot, X\neq 0}\tilde T^*\tilde{K}_X$.

Using conjugation by $K$, we may assume that $X$ is in $\tgot$.
Thus we obtain
$N^1=(\tilde K\times K) \cdot \cup_{X\in \tgot, X\neq 0}\T^*\tilde{K}_X$.

The set of subgroups $\tilde{K}_X$ when $X$ varies in $\tgot$ is finite.
Consider the fiber product
$F(X)=(\tilde K\times K)\times_{\tilde{K}_X\times K_X}\T^*\tilde{K}_X$.
It surjects on $(\tilde K\times K) \cdot\T^*{\tilde K}_X$.
Thus it is sufficient to prove that
for $X\neq 0$ in $\tgot$, and under our hypothesis, the dimension of
$F(X)$ is strictly less than   $\T^*\tilde K$.

Computing dimensions, if the dimensions are equal, then $\dim (K/K_X)=\dim ({\tilde K}/\tilde{K}_X).$
So this implies that $K\cdot X=\tilde{K}\cdot X$.
The vector space spanned by
$K\cdot X$ is then an ideal of $\tilde \kgot$ contained in $\kgot$, contradicting our hypothesis.
\end{proof}

\medskip

In the next section, we will give some qualitative information concerning the behavior of the
multiplicity function $\mm(\tilde{\lambda},\lambda)$ on the cone $\Delta(\T^*\tilde{K})$.

\subsection{Quasi polynomial behavior}

We consider the slice
$$
Y=\Phi^{-1}(\tilde{\tgot}^*_{>0}\times \tgot^*_{>0})
$$
which is a (non-empty) symplectic submanifold of $\T^*\tilde{K}$ equipped with the action of $\tilde{T}\times T$.


We denote $\Phi_Y :Y\to \tilde{\tgot}^*\times \tgot^*$ the restriction of the moment map on $Y$.
Let $\Delta^0\subset \Delta(\T^*\tilde{K})$ be the open cone formed by the regular values of
$\Phi_Y : Y\to \tilde{\tgot}^*\times \tgot^*$. If $\agot\subset \Delta^0$ is a connected component of $\Delta^0$,
we choose a point $(\tilde{\xi},\xi)\in\agot$, and  we consider:

\begin{itemize}
\item  the compact connected symplectic orbifold $N_\agot:= \Phi_Y^{-1}(\tilde{\xi},\xi)/ \tilde{T}\times T$,

\item  the family of line orbi-bundle $L_\agot^{\tilde{\alpha},\alpha}\to N_\agot$ parametrized by
$(\tilde{\alpha},\alpha)\in \tilde{\Lambda}\times\Lambda$~:
$$
L_\agot^{\tilde{\alpha},\alpha}:=\Phi_Y^{-1}(\tilde{\xi},\xi)\times_{\tilde{T}\times T} \C_{-\tilde{\alpha},-\alpha}.
$$
\end{itemize}
Here $\C_{-\tilde{\alpha},-\alpha}$ denotes the one dimensional representation of $\tilde{T}\times T$ where the group acts through
the character $\chi_{-\tilde{\alpha},-\alpha}(t_1,t_2)=t_1^{-\tilde{\alpha}}t_2^{-\alpha}$.

\begin{defi}
Let $\agot$ be a connected component of $\Delta(\T^*K)^0$. We consider the function
$p_\agot: \tilde{\Lambda}\times\Lambda\to \Z$ defined by the relation
$$
p_\agot(\tilde{\alpha},\alpha):=RR(N_\agot,L_\agot^{\tilde{\alpha},\alpha}).
$$
\end{defi}

The function $p_\agot$ does not depend on the choice of the choice of  $\tilde{\xi},\xi$ in $\agot$, since the $\tilde{T}\times T$-manifolds
$\Phi_Y^{-1}(\tilde{\xi_1},\xi_1)$ and $\Phi_Y^{-1}(\tilde{\xi_2},\xi_2)$
 are diffeomorphic for two different choices. The Kawasaki-Riemann-Roch Theorem \cite{Kawasaki81}
implies that the function $p_\agot$ have a quasi polynomial behavior on the lattice
$\tilde{\Lambda}\times\Lambda$.

We can now state the main result of this section. Note that $\Delta(\T^*K)=\bigcup_{\agot} \overline{\agot}$.

\begin{theo}\label{theo:quasipolynomial-K-Ktilde}
If $(\tilde{\lambda},\lambda)\in \overline{\agot}$, we have
$\mm(\tilde{\lambda},\lambda)=p_\agot(\tilde{\lambda},\lambda)$.
\end{theo}

\begin{proof}
Let $b:=(\tilde{\xi},\xi)\in\agot$ and $(\tilde{\lambda},\lambda)\in \overline{\agot}$. By (\ref{eq:localisation-Phi-0-hamiltonien}),
the quantity $p_\agot(\tilde{\lambda},\lambda)$ is equal to
$$
\left[RR_{\tilde{T}\times T}(Y, \C, \Phi_Y^{-1}(b),\Phi_Y-b)\otimes \C_{-\tilde{\lambda},-\lambda}\right]^{\tilde{T}\times T}.
$$
We write $\Phi_Y=\Phi_{\tilde{T}}\oplus\Phi_T$. An element $(g,\eta)$ belongs to $Y$ if and only if
$-\eta\in \tilde{\tgot}^*_{>0}$  and $\pi(g^{-1}\eta)\in\tgot^*_{>0}$.
The group $\tilde{T}$ acts freely on $Y$, and the smooth reduced space
$Y_{{\tilde{\xi}}}=\Phi_{\tilde{T}}^{-1}(\tilde{\xi})/\tilde{T}$ is symplectomorphic to the slice
$$
\{\eta\in \tilde{K}\tilde{\xi}^*\ |\ \pi(\eta)\in \tgot^*_{>0}\}\subset \tilde{K}\tilde{\xi}^*
$$
through the map $(g,\eta)\to g^{-1}\eta$. We have
\begin{eqnarray*}
\lefteqn{\left[RR_{\tilde{T}\times T}(Y, \C, \Phi_Y^{-1}(b),\Phi_Y-b)\otimes \C_{-\tilde{\lambda},-\lambda}\right]^{\tilde{T}\times T}}&&\\
&=&
\left[RR_{T}(Y_{{\tilde{\xi}}}, [\C_{-\tilde{\lambda}}], \Phi_T^{-1}(\xi),\Phi_T-\xi)\otimes \C_{-\lambda}\right]^{T}\qquad (1)\\
&=&\left[RR_{K}(\tilde{K}\tilde{\xi}^*\times (K\xi)^-, [\C_{-\tilde{\lambda}}]\boxtimes [\C_{-\lambda}], \Phi_b^{-1}(0),\Phi_b)\right]^{K}\qquad (2)
\end{eqnarray*}
where $\Phi_b :\tilde{K}\tilde{\xi}^*\times \overline{K\xi}\to \kgot^*$ is the moment map defined by $\Phi_b(\tilde{\eta},\eta)=\pi(\tilde{\eta})- \eta$. Equality $(1)$ follows from the reduction in stage property (see Proposition \ref{prop:localisation-stage}),
and $(2)$ is a particular case of (\ref{eq:reduction-slice}).

If $b=(\tilde{\xi},\xi)\in\agot$ is close enough to $(\tilde{\lambda},\lambda)$, the line bundle $[\C_{-\tilde{\lambda}}]\boxtimes [\C_{-\lambda}]$ is weakly $\Phi_b$-positive (see Lemma \ref{lem:close-enough}). Since $\Phi_b^{-1}(0)\neq \emptyset$ we
can conclude that
\begin{eqnarray*}
\lefteqn{\left[RR_{K}(\tilde{K}\tilde{\xi}^*\times (K\xi)^-,
[\C_{-\tilde{\lambda}}]\boxtimes [\C_{-\lambda}], \Phi_b^{-1}(0),\Phi_b)\right]^{K}}\\
&=&\left[RR_{K}(\tilde{K}\tilde{\xi}^*\times (K\xi)^-, [\C_{-\tilde{\lambda}}]\boxtimes [\C_{-\lambda}])\right]^{K}\\
&=& \left[(\tilde{V}_{\tilde{\lambda}})^*\vert_{K}\otimes V_\lambda^*\right]^K\\
&=& \mm(\tilde{\lambda},\lambda).
\end{eqnarray*}

\end{proof}

\begin{rem}\label{rem:prop-ressayre}
In the proof above, we obtain the following identity
$$
\mm(\tilde{\lambda},\lambda)=
\left[RR_{\tilde{T}\times T}(Y, \C, \Phi_Y^{-1}(b),\Phi_Y-b)\otimes \C_{-\tilde{\lambda},-\lambda}\right]^{\tilde{T}\times T},
$$
that holds for any $b=(\tilde{\xi},\xi)\in \tilde{\tgot}^*_{>0}\times \tgot^*_{>0}$ that is close enough to
$(\tilde{\lambda},\lambda)$. The identity $\mm(\tilde{\lambda},\lambda)=p_\agot(\tilde{\lambda},\lambda)$
follows by taking $(\tilde{\xi},\xi)$ a regular value of $\Phi$. But we could have taken a quasi-regular value of $\Phi$
(see Section \ref{multiplicity-face-K-Ktilde}).
\end{rem}

\medskip

We can also view this theorem as a consequence of  $[Q,R]=0$ for  the non compact manifold $N=\T^*\tilde{K}$.
Philosophically, we have used here reduction in stages to reduce to the $\tilde K\times K$ reduced space of $N$
 to the reduced space with respect to $K$ of the reduced space with respect to $\tilde K$ of $N$.

 We could however, at the price of proving $[Q,R]=0$ directly for the non compact manifold $N$, prove directly our theorem.
 Let us sketch the proof. If we use the result of Section \ref{subsec:cotangent}, knowing that the critical set $Z_\Phi$ is compact,
 we see  that the localized index $RR_{\tilde{K}\times K}(N,[\C],Z_\Phi)$ is equal to $\sum_{\tilde{\eta}}\tilde{V}_{\tilde{\eta}}^*\otimes
 \tilde{V}_{\tilde{\eta}}\vert_K$ in $\hat{R}(\tilde{K}\times K)$.

 The quantity $\mm(\tilde{\lambda},\lambda)$ is then the multiplicity of $\tilde{V}_{\tilde{\lambda}}\otimes
 V_{\lambda}$ in the character $RR_{\tilde{K}\times K}(N,[\C],Z_\Phi)$. Note that here the trivial line bundle $[\C]$ is a
 $\Phi$-moment line bundle on $N$. By an adaptation of the shifting trick to this non-compact setting one can prove
 directly that $\mm(\tilde{\lambda},\lambda)=RR(N_{{\tilde{\lambda}},{\lambda}},[\C]_{\tilde{\lambda},\lambda})$. This
 is the content of Theorem \ref{theo:quasipolynomial-K-Ktilde} since the term $RR(N_{{\tilde{\lambda}},\lambda},[\C]_{\tilde{\lambda},\lambda})$ is equal to $p_\agot(\tilde{\lambda},\lambda)$ for any chamber $\agot$ containing $(\tilde{\lambda},\lambda)$ in its closure.

\subsection{Multiplicities on a face}\label{multiplicity-face-K-Ktilde}

Let us give the  structure of the general faces of the cone $\Delta(\T^*\tilde{K})\subset \tilde{\tgot}_{\geq 0}^*\times \tgot^*_{\geq 0}$.
Let $F$ be a closed face of $\Delta(\T^*\tilde{K})$ that intersects $\tilde{\tgot}_{> 0}^*\times \tgot^*_{> 0}$. We say that $F$ is a general face.
We denote $\langle F\rangle$ the vector subspace generated by $F$, and $\ggot_F$ its orthogonal in $\tilde{\tgot}\times \tgot$. Note that
$F=\langle F\rangle\cap \Delta(\T^*\tilde{K})$.

We denote $\tilde{W}$ the Weyl group of $\tilde{K}$ relatively to the maximal torus $\tilde{T}$.

\begin{lem}
$\bullet$ There exists a vector subspace $\tgot_F\subset \tgot$ and
$\tilde{w}\in\tilde{W}$ such that $\ggot_F=\{(\tilde{w}X,X),\ X\in \tgot_F\}$.

$\bullet$ For any weights $(\tilde{\lambda},\lambda)$ contained in $F$, the weight $\tilde{w}^{-1}\tilde{\lambda}+\lambda$ vanishes on
$\tgot_F$.

$\bullet$ The connected component of $(\T^*\tilde{K})^{\ggot_F}$ that contains
$\Phi^{-1}(F\cap\tilde{\tgot}_{> 0}^*\times \tgot^*_{> 0} )$ is $\tilde{w}\cdot\T^* \tilde{K}_F$. Here
$\tilde{K}_F$ is the subgroup of $\tilde{K}$  that centralizes
$\tgot_F$.

\end{lem}

\begin{proof}We know that any element $(\tilde{X},X)\in\ggot_F$ fixes the points of
$\Phi^{-1}(F\cap\tilde{\tgot}_{> 0}^*\times \tgot^*_{> 0} )$ (see Lemma \ref{lem:phi-F-sigma}), and that forces to have
$\tilde{X}=\tau X$ for some $\tau\in \tilde{W}$ (see Lemma \ref{lem:point-fixe-cotangent}). The first point follows.
The second point is obvious and the third follows directly from Lemma \ref{lem:point-fixe-cotangent}.
\end{proof}

Berenstein-Sjamaar \cite{Berenstein-Sjamaar} and Ressayre \cite{Ressayre-inventiones,Ressayre-AIF}
have determined, in terms of Schubert calculus,
what are the data $(\tilde{w},\tgot_F)$ that comes from a face $F$.  It is usually  difficult to
list them explicitly.

\medskip

We denote by $K_F$ the subgroup of $K$ that centralizes
$\tgot_F$.
The element $\tilde w$ is defined modulo the Weyl group of $K_F$.
We thus can choose   $\tilde{w}$ such that for any root $\alpha$ of $\tilde{K}_F$, we have
$\alpha >0 \Leftrightarrow \tilde{w}\alpha >0$.

 The subgroup of $\tilde{K}\times K$ that centralizes $\ggot_F$ is  of the form
$\tilde{w}\tilde{K}_F\tilde{w}^{-1}\times K_F$.
So for weights $(\tilde{\lambda},\lambda)$ contained in $F$, the weight
$$
\widehat{\lambda}:=\tilde{w}^{-1}\tilde{\lambda}
$$
 is dominant for $\tilde{K}_F$: we denote $\tilde{V}^F_{\widehat{\lambda}}$ the corresponding irreducible
 representation of $\tilde{K}_F$. Similarly the weight $\lambda$ is dominant for $K_F$ and we denote
$V^F_{\lambda}$ the corresponding irreducible representation of $K_F$.

\medskip

We end this section with the following result obtained by Ressayre in \cite{Ressayre-reduction}, which is
a non-compact analogue of Proposition \ref{QR-mu-face}.

\begin{prop}
 For any dominant weights $(\tilde{\lambda},\lambda)$ contained in $F$,  the dimension of
 $\left[\tilde{V}_{\tilde{\lambda}}\vert_{K}\otimes V_{\lambda}\right]^K$ is equal to the dimension of
 $\left[\tilde{V}^F_{\widehat{\lambda}}\vert_{K_F}\otimes V^F_{\lambda}\right]^{K_F}$.
\end{prop}

\begin{proof} We denote $G_F$ the subtorus of $\tilde{T}\times T$ with Lie algebra $\ggot_F$.

We start with the identity
\begin{equation}\label{eq:intermediaire}
\left[\tilde{V}_{\tilde{\lambda}}\vert_{K}\otimes V_{\lambda}\right]^K=
\left[RR_{\tilde{T}\times T}(Y, \C, \Phi_Y^{-1}(b),\Phi_Y-b)\otimes \C_{-\tilde{\lambda},-\lambda}\right]^{\tilde{T}\times T},
\end{equation}
that holds for any $b=(\tilde{\xi},\xi)\in \tilde{\tgot}^*_{>0}\times \tgot^*_{>0}$ that is close enough to
$(\tilde{\lambda},\lambda)$. See Remark \ref{rem:prop-ressayre}. Let us use it with
$(\tilde{\xi},\xi)$ a generic element contained in $F\cap (\tilde{\tgot}^*_{>0}\times \tgot^*_{>0})$ and
close enough to $(\tilde{\lambda},\lambda)$ : it is a quasi-regular value of $\Phi$ such that the fiber
$\Phi^{-1}(\tilde{\xi},\xi)\subset \tilde{w}\cdot\T^* \tilde{K}_F$
is a smooth submanifold with a locally free action of the group $(\tilde{T}\times T)/G_F$,

Thanks to Theorem \ref{theo:RR-M-a-E-alpha}, we know that the right hand side of
(\ref{eq:intermediaire}) is equal to $RR(N_{\tilde{\xi},\xi},\Lbb_{\tilde{\lambda},\lambda})$ where
$$
\Lbb_{\tilde{\lambda},\lambda}=\Phi^{-1}(\tilde{\xi},\xi)\times_{(\tilde{T}\times T)/G_F} \C_{-\tilde{\lambda},-\lambda}
$$
is an line orbibundle on the symplectic orbifold $N_{\tilde{\xi},\xi}=\Phi^{-1}(\tilde{\xi},\xi)/((\tilde{T}\times T)/G_F)$

We consider now the Hamiltonian action of $\tilde{K}_F\times K_F$ on $N_F:=\T^*\tilde{K}_F$. Let us denote
$\Phi_F:N_F\to \tilde{\kgot}_F^*\times \kgot^*_F$ the corresponding moment map. The torus $T_F$ with Lie algebra
$\tgot_F$ is embedded diagonaly in $\tilde{T}\times T$. For a generic element $(\alpha,\beta)\in \Delta(N_F)$, the fiber
$\Phi_F^{-1}(\alpha,\beta)$ is a smooth submanifold with a locally free action of $ (\tilde{T}\times T)/T_F$.

As in the previous setting, we know that the multiplicity
$\left[\tilde{V}^F_{\widehat{\lambda}}\vert_{K_F}\otimes V^F_{\lambda}\right]^{K_F}$ is equal to
$RR((N_F)_{\widehat{\xi},\xi},\Lbb_{\widehat{\lambda},\lambda})$ when $(\widehat{\xi},\xi)\in\Delta(N_F)$ is a generic element
close enough to $(\widehat{\lambda},\lambda)$. Here
$$
\Lbb_{\widehat{\lambda},\lambda}=\Phi^{-1}_F(\widehat{\xi},\xi)\times_{(\tilde{T}\times T)/T_F} \C_{-\widehat{\lambda},-\lambda}
$$
is an line orbibundle on the symplectic orbifold $(N_F)_{\widehat{\xi},\xi}=\Phi^{-1}_F(\widehat{\xi},\xi)/((\tilde{T}\times T)/T_F)$.

Now our proposition follows from the fact that the Riemann-Roch numbers $RR(N_{\tilde{\xi},\xi},\Lbb_{\tilde{\lambda},\lambda})$ and
$RR((N_F)_{\widehat{\xi},\xi},\Lbb_{\widehat{\lambda},\lambda})$ are equal. It is due to the fact that the multiplication by
$\tilde{w}^{-1}$ induces a symplectomorphism $N_{\tilde{\xi},\xi}\simeq (N_F)_{\widehat{\xi},\xi}$ and an isomorphism of
line bundles $\Lbb_{\tilde{\lambda},\lambda}\simeq \Lbb_{\widehat{\lambda},\lambda}$.

\end{proof}


\end{document}